\documentclass[12pt]{amsart}
\usepackage{mathrsfs}
\usepackage{amsmath,amssymb,amsthm,upref,graphicx,mathrsfs}
\usepackage{enumerate}
\usepackage[utf8]{inputenc}

\numberwithin{equation}{section}

\newtheorem{theorem}{Theorem}[section]
\newtheorem{proposition}[theorem]{Proposition}
\newtheorem{lemma}[theorem]{Lemma}
\newtheorem{corollary}[theorem]{Corollary}

\theoremstyle{definition}
\newtheorem{definition}[theorem]{Definition}
\newtheorem{example}[theorem]{Example}
\newtheorem{remark}[theorem]{Remark}

\newcommand{\E}{\mathbb{E}}
\newcommand{\F}{\mathcal{F}}
\newcommand{\PP}{\mathbb{P}}
\newcommand{\nn}{{n \in \N}}
\newcommand{\n}{\nonumber}

\newcommand{\R}{\mathbb{R}}

\newcommand{\N}{\mathbb{N}}

\begin{document}

\title[Variable Martingale Hardy Spaces]
{Variable Martingale Hardy Spaces and Their Applications in Fourier Analysis}

\authors

\author{Yong Jiao}
\address{School of Mathematics and statistics, Central South University, Changsha 410075, China}
\email{jiaoyong@csu.edu.cn}

\author{Ferenc Weisz}
\address{Department of Numerical Analysis, E\"otv\"os L. University,
H-1117 Budapest, P\'azm\'any P. s\'et\'any 1/C., Hungary}
\email{weisz@inf.elte.hu}

\author{Lian Wu}
\address{School of Mathematics and statistics, Central South University, Changsha 410075, China}
\email{wulian@csu.edu.cn}

\author{Dejian Zhou}
\address{School of Mathematics and statistics, Central South University, Changsha 410075, China}
\email{zhoudejian@csu.edu.cn}


\begin{abstract}
Let $p(\cdot)$ be a measurable function defined on a probability space satisfying $0<p_-:={\rm ess}\inf_{x\in \Omega}p(x)\leq {\rm ess}\sup_{x\in\Omega}p(x)=:p_+<\infty$. We investigate five types of martingale Hardy spaces $H_{p(\cdot)}$ and $H_{p(\cdot),q}$ and prove their atomic decompositions when each $\sigma$-algebra $\mathcal F_n$ is generated by countably many atoms. Martingale inequalities and the relation of the different martingale Hardy spaces are proved as application of the atomic decomposition.
In order to get these results, we introduce the following condition to replace (generalize) the so-called log-H\"{o}lder continuity condition in harmonic analysis:
$$
\mathbb P(A)^{p_-(A)-p_+(A)}\leq C_{p(\cdot)} \quad \mbox{ for all atom $A$}.
$$
Some applications in Fourier analysis are given by use of the previous results. We generalize the classical results and show that the partial sums of the Walsh-Fourier series converge to the function in norm if $f \in L_{p(\cdot)}$ or $f \in L_{p(\cdot),q}$ and $p_->1$. The boundedness of the maximal Fej{\'e}r operator on $H_{p(\cdot)}$ and $H_{p(\cdot),q}$ is proved whenever $p_->1/2$ and the condition $\frac{1}{p_-}-\frac{1}{p_+} <1$ hold. It is surprising that this last condition does not appear for trigonometric Fourier series. One of the key points of the proof is that we introduce two new dyadic maximal operators and prove their boundedness on $L_{p(\cdot)}$ with $p_->1$. 
The method we use to prove these results is new even in the classical case. As a consequence, we obtain theorems about almost everywhere and norm convergence of the Fejér means.
\end{abstract}

\maketitle

 \section{Introduction}

Let $p(\cdot): \mathbb{R}^n\rightarrow (0,\infty)$ be a measurable
function (a variable exponent). The variable Lebesgue space $L_{p(\cdot)}(\mathbb{R}^n)$ consists of
 all measurable functions $f$ such that $\int_{\mathbb R^n}|f(x)|^{p(x)}dx<\infty$. It generalizes the classical Lebesgue space: when $p(\cdot)\equiv p$ is a constant, then
 $L_{p(\cdot)}(\mathbb{R}^n)\equiv L_{p}(\mathbb{R}^n)$. Interest in the variable Lebesgue spaces has increased since the 1990s because of their use in a variety of applications. As we all know, the variable function spaces can be well applied in fluid dynamics \cite{am02,am05,rm00}, image processing
\cite{clr06,hhlt13,Ti14}, partial differential equations and variational calculus
\cite{bds15,ds14,t17} and harmonic analysis \cite{ahhl15,Cruz2013, Diening2011, Izuki2014,ylk17}.

In order to extend the techniques and results of constant exponent case to the setting of variable Lebesgue spaces, a central problem is to determine conditions on an exponent $p(\cdot)$ such that the Hardy-Littlewood maximal operator is bounded on $L_{p(\cdot)}(\mathbb R^n)$.
The first major result is due to Diening \cite{Diening2004}, who showed that it is sufficient to assume that $p(\cdot)$ satisfies the so-called  local log-H\"{o}lder condition:
 \begin{equation}\label{local log}
 |p(x)-p(y)|\leq \frac{C}{-\log{|x-y|}},\quad \forall x,y\in \mathbb R^n,\quad  |x-y|<1/2,
 \end{equation}
 and is constant outside of a large ball. This result was generalized independently by Cruz-Uribe et al. \cite{Cruz2003} and Nekvinda \cite{Nekvinda2004},
 who in addition assumed that $p(\cdot)$ is log-H\"{o}lder continuous at infinity:
 there is $p_\infty>1$ such that
 \begin{equation}\label{log at infinity}
  |p(x)-p_\infty|\leq \frac{C}{\log({e+|x|})},\quad \forall x\in \mathbb R^n.
\end{equation}
The conditions \eqref{local log} and \eqref{log at infinity} are called log-H\"{o}lder continuity condition. We point out that the log-H\"{o}lder continuity condition is not necessary for the boundedness of the Hardy-Littlewood maximal operator on $L^{p(\cdot)}(\mathbb R^n)$ (see \cite{Nekvinda2008}).
Heavily basing on the log-H\"{o}lder continuity condition, harmonic analysis with variable exponent has gotten a rapid development. Nakai and Sawano \cite{Nakai2012} first introduced
the Hardy space $H^{p(\cdot)}(\mathbb R^n)$ with a variable exponent $p(\cdot)$ and established the atomic decompositions. As applications, they proved the duality and the boundedness of
singular integral operators. Independently, Cruz-Uribe and Wang \cite{Cruz2014} also investigated the variable Hardy space $H^{p(\cdot)}(\mathbb R^n)$ with $p(\cdot)$ satisfying
some conditions slightly weaker than those used in \cite{Nakai2012}.
Sawano \cite{Sawano2013} improved the results in \cite{Nakai2012}. Ho \cite{Ho2017} studied weighted Hardy spaces with variable exponents. Zhuo et al. \cite{ZSY2016} investigated Hardy spaces with variable exponents on RD-spaces and applications. Very recently, Yan et al. \cite{Yang2016} introduced the variable weak Hardy space $H_{p(\cdot),\infty}(\mathbb R^n)$ and characterized these spaces via the radial maximal functions, atoms and Littlewood-Paley functions. The Hardy-Lorentz spaces $H_{p(\cdot),q}(\mathbb R^n)$ were investigated by Jiao et al. in a very recent paper \cite{Jiao2019}. Similar results for the anisotropic Hardy spaces $H_{p(\cdot)}(\mathbb R^n)$ and $H_{p(\cdot),q}(\mathbb R^n)$ can be found in Liu et al. \cite{Liu2018,Liu2019}. 
Martingale Musielak–Orlicz Hardy spaces were investigated in Xie et al. \cite{Xie2019,Xie2019a,XWYJ2018}. 
We also refer to \cite{Almeida2010, Almeida2016,Yang2015} for Besov spaces with variable smoothness and integrability and their applications.

In the early 70's of the last century, with the development of the theory of Hardy spaces on $\mathbb R^n$ in harmonic analysis, martingale Hardy spaces theory was born. Until now, most of the important facts in harmonic analysis have been found to have their satisfactory counterparts in the martingale setting. For example, in martingale setting, the duality between $H_1$ and $BMO$, and the Doob maximal inequality can be found in Garsia \cite{Garsia1973}; the Burkholder martingale transforms \cite{Burkholder1966} can be considered as an analogue to the classical singular integral operators. On the other hand, the theory of martingale Hardy spaces has influenced the development of harmonic analysis. For example, the atomic decomposition of $H_p$, which is one of the most powerful tool in harmonic analysis nowadays, was first shown in martingale setting by Herz \cite{Herz974}. Later, the theory of atomic decomposition of martingale spaces was developed in Weisz \cite{Weisz1994book}. The good-$\lambda$ inequality, which is a useful tool to compare the integrability of two related measurable functions,
was discovered by Burkholder and Gundy \cite{Burkholder1970, Burkholder1973} in martingale setting. A much more simplified proof of $T(b)$ theorem was given by Coifman et al. \cite{Coifman1989} via martingale approach.
The theory of martingale Hardy spaces was considered in the books \cite{Garsia1973,Long1993book,Weisz1994book,Weisz2002book}. The applications of martingale theory to Fourier analysis were developed by many people, see for example monographs \cite{sws,Weisz1994book,Weisz2002book,Weisz2017book} and the references therein.

Although the theory of variable Hardy spaces on $\mathbb R^n$ has rapidly been developed in recent years,
the variable exponent framework has not yet been applied to the martingale setting. The first main difficulty we need to overcome is to find a suitable replacement for the log-H\"{o}lder continuity conditions \eqref{local log} and \eqref{log at infinity} when the variable exponent $p(\cdot)$ is defined on a probability space. Unlike the Euclidean space $\mathbb R^n$, there is no natural metric in a probability space. In order to better explain it, we first introduce some basic notation.
Let $(\Omega,\F,\mathbb P; (\mathcal F_n)_{n\geq0})$ be a complete probability space such that $\mathcal F=\sigma(\cup_n\mathcal F_n)$. A measurable function $p(\cdot)$ :
$\Omega\rightarrow(0,\infty)$ is called a variable exponent.
For a measurable set $A\subset\Omega$, we denote
$$ p_-(A):=\text{ess} \inf \limits_{x\in A}p(x),\quad p_+(A):=\text{ess} \sup\limits_{x\in A}p(x), $$
and for convenience
$$ p_-:=p_-(\Omega),\quad p_+:=p_+(\Omega).$$
Denote by $\mathcal{P}(\Omega)$ the collection of all variable exponents  $p(\cdot)$ such that $0<p_-\leq p_+<\infty.$
If $p(\cdot)\in \mathcal{P}(\Omega)$, the variable Lebesgue space $L_{p(\cdot)}(\Omega)$ consists of all measurable functions $f$ for which
$$
\|f\|_{p(\cdot)}=\inf\left\{\lambda>0: \int_\Omega \left(\frac{|f(x)|}{\lambda}\right)^{p(x)}d\mathbb P\leq 1 \right\} <\infty.
$$
For a martingale $f=(f_n)_{n\geq0}$ with respect to $(\Omega,\F,\mathbb P; (\mathcal F_n)_{n\geq0})$, the Doob maximal operator is defined by
$$M_mf= \sup_{0\leq n\leq m} |f_n|,\qquad Mf=\sup_{n\geq0} |f_n|.$$
The classical Doob maximal inequality implies that $M$ is bounded on $L_p(\Omega)$ for $1<p\leq \infty.$
However, according to the facts in \cite[Example 3.21]{Cruz2013}, the Doob maximal operator is not bounded on $L_{p(\cdot)}(\Omega)$ for general variable exponent $p(\cdot)$ with $p_->1$. It is natural to find sufficient condition imposed on $p(\cdot)$
such that the Doob maximal operator M is bounded on $L_{p(\cdot)}(\Omega)$.
Aoyama \cite{Aoyama2009} proved the Doob maximal inequality under the condition that $p(\cdot)$ is $\mathcal F_n$-measurable for all $n\geq0$. Obviously, this kind of condition is quite strong. Moreover, Nakai and Sadasue \cite{Nakai2013}
showed that $\mathcal F_0$-measurability of $p(\cdot)$ is not necessary for this maximal inequality. Note that a weak type inequality was proved in
\cite[Theorem 3.2]{Jiao2016}. Namely, given $p(\cdot)\in \mathcal{P}(\Omega)$ with $1\leq p_-\leq p_+<\infty,$
\begin{equation}\label{weak jiao}
\mathbb P(Mf>\lambda)\leq C_{p(\cdot)}\int_{\Omega}\Big(\frac{|f_\infty(x)|}{\lambda}\Big)^{p(x)}d\mathbb P,\quad \forall \lambda>0.
\end{equation}
Unfortunately, we can not obtain the Doob maximal inequality by means of the weak type inequality \eqref{weak jiao} as in the classical case. The essential reason is that the space $L_{p(\cdot)}(\Omega)$ is no longer a rearrangement invariant space and the important formula
$$\int_\Omega|f(x)|^pd\mathbb P=p\int_0^\infty t^{p-1}\mathbb P(x\in \Omega:|f(t)|>t)dt$$
has no variable exponent analogue. 

In this paper, we introduce a condition without metric characterization of $p(\cdot)$
to replace the log-H\"{o}lder continuity condition mentioned above. We suppose that there exists a constant $K_{p(\cdot)}\geq1$ depending only on $p(\cdot)$ such that
\begin{equation}
\label{log1}
 \mathbb{P}(A)^{p_-(A)-p_+(A)} \leq K_{p(\cdot)}, \quad\forall A\in \bigcup_n A(\mathcal{F}_n) ,
\end{equation}
where $\mathcal{F}_n$ is generated by countably many atoms and
$A(\mathcal{F}_n)$ denotes the family of all atoms in $\mathcal F_n$ for each $n\in \mathbb N$. It is known that the log-Hölder continuity \eqref{local log} and \eqref{log at infinity} imply our condition \eqref{log1} on $[0,1)$ or even on $\R^{n}$ (see Remark \ref{rem:variable}). We should mention that Nakai et al. \cite{Nakai2014, Nakai2013Morrey, Nakai2012S} and  Ho \cite{Ho2016} studied the martingale Morrey-Hardy and Campanato-Hardy spaces associated with $\mathcal{F}_n$ generated by countably many atoms.

We give a systematic study of martingale Hardy spaces $H_{p(\cdot)}$ and $H_{p(\cdot),q}$  associated with a variable exponent $p(\cdot)$.
A powerful tool used in the paper is the atomic decomposition of variable martingale Hardy and Hardy-Lorentz spaces. Our first main result, without any restriction on $p(\cdot)$,  is the $(1,p(\cdot),\infty)$-atomic
characterization of the Hardy spaces $H_{p(\cdot)}^s$ and $H_{p(\cdot),q}^s$ associated with the conditional square operator $s$, that is, $H_{p(\cdot)}^s=H_{p(\cdot)}^{\rm at,1,\infty}$ and $H_{p(\cdot),q}^s=H_{p(\cdot),q}^{\rm at,1,\infty}$ with equivalent
quasi-norms.
As one of the applications of atomic decompositions, we get martingale inequalities between different Hardy spaces $H_{p(\cdot)}$ and $H_{p(\cdot),q}$ in Section \ref{sec4}.

Finally, we consider the applications of the theory of variable martingale Hardy spaces in Fourier analysis. In the constant exponent case, the theory of martingales has an extensive  application in dyadic harmonic analysis; see for example the monographs \cite{sws,Weisz1994book,Weisz2002book,Weisz2017book}, the papers Gát and Goginava \cite{gat1998a,Gat1999,Gat2014d,Goginava2000,Goginava2003,Goginava2007b} and Schipp and Simon \cite{sch2,schs,si7,Simon2004,wgyenge}.
Particularly, in \cite[Theorem 3.10]{Weisz2002book}, by using dyadic martingale theory, Weisz proved that the maximal Fej{\'e}r operator $\sigma_*$ is bounded from $H_{p,q}$ to $L_{p,q}$ with $p> 1/2$. Then, it can be deduced from Weisz's result that the Fej{\'e}r means of $f$ converge almost everywhere to $f$. Inspired by this result, we generalize these theorems and prove that $\sigma_*$ is bounded from $H_{p(\cdot)}$ to $L_{p(\cdot)}$ and from $H_{p(\cdot),q}$ to $L_{p(\cdot),q}$ under the conditions $1/2<p_-< \infty$, $0<q\leq \infty$ and $\frac{1}{p_-}-\frac{1}{p_+} <1$. This last condition is very surprising because the corresponding results for Fourier transforms hold without this condition (see Liu et al. \cite{Liu2018,Liu2019} and Weisz \cite{wriesz-variable}). This gives a serious difference between the trigonometric Fourier analysis and Walsh-Fourier analysis. Unlike the classical case, our proof does not depend on interpolation method. One of the key points of the proof is that we introduce two new dyadic maximal operators and prove their boundedness on $L_{p(\cdot)}$ with $p(\cdot)$ satisfying \eqref{log1} and $p_->1$. This method is new even in the classical case \cite{wces3,Weisz2002book}. Finally, we show that the boundedness of $\sigma_*$ implies almost everywhere and norm convergence of the Fejér means as well.

The structure of the paper is as follows. In Section \ref{sec2}, we present preliminaries, definitions and lemmas used later in the paper.
We also introduce the definition of variable Lebesgue space $L_{p(\cdot)}(\Omega)$ and Lorentz space $L_{p(\cdot),q}(\Omega)$. Some basic properties of these spaces are given, including duality of $L_{p(\cdot)}(\Omega)$ and dominated
convergence theorem in variable Lorentz space $L_{p(\cdot),q}(\Omega)$. We also introduce five types of variable Hardy spaces
$H_{p(\cdot)}$ and $H_{p(\cdot),q}$ in this section.
Moreover, we show some basic inequalities for the Doob maximal operator,
including the boundedness of the operator on $L_{p(\cdot)}(\Omega)$ and $L_{p(\cdot),q}(\Omega)$ (see Theorem \ref{thm:maximal inequality} below) and the variable version
of the dual Doob inequality. We also
prove that the Doob maximal operator is bounded from $L_{p(\cdot)}(\Omega)$ to $L_{p(\cdot),\infty}(\Omega)$ with $p_-\geq 1$.  

The objective of Section \ref{sec3} is the atomic decomposition  for variable Hardy spaces $H_{p(\cdot)}$ and $H_{p(\cdot),q}$. We prove the desired atomic decompositions for all kinds of Hardy spaces.

Section  \ref{sec4}  is devoted to the applications of atomic decompositions established in Section \ref{sec3}.
In Section \ref{sec4}, we obtain some continuous embedding relationships among different variable martingale Hardy spaces and martingale Hardy-Lorentz spaces.
Moreover, if $(\mathcal F_n)_{n\geq0}$ is regular, then different kinds of $H_{p(\cdot)}$ (resp. $H_{p(\cdot),q}$) are all equivalent.

In the last section, we deal with some applications in Walsh-Fourier analysis. We prove the above mentioned results about the Walsh-Fourier series.

Throughout this paper, $\mathbb{Z}$ and $\mathbb{N}$
denote the integer set and nonnegative integer set, respectively. We denote by $C$ a positive constant,
which can vary from line to line, and denote by $C_{p(\cdot)}$ a
constant depending only on $p(\cdot).$ The symbol $A\lesssim B$
stands for the inequality $A \leq C B$ or $A \leq C_{p(\cdot)} B$.  If we
write $A\thickapprox B$, then it stands for $A\lesssim B\lesssim A$.
We will use $\chi_E$ or $1_E$ to denote the indicator function of the measurable set $E$.

We would like to thank the referee for reading the paper carefully and for his/her useful comments and suggestions.

\section{Preliminaries}\label{sec2}

\subsection{Variable Lebesgue spaces $L_{p(\cdot)}$}
Let $(\Omega,\F,\PP)$ be a complete probability space. A measurable function $p(\cdot)$ :
$\Omega\rightarrow(0,\infty)$ is called a variable exponent.
For a measurable set $A\subset\Omega$, we denote
$$ p_-(A):=\text{ess} \inf \limits_{x\in A}p(x),\quad p_+(A):=\text{ess} \sup\limits_{x\in A}p(x)$$
and for convenience
$$ p_-:=p_-(\Omega),\quad p_+:=p_+(\Omega).$$
Denote by $\mathcal{P}(\Omega)$ the collection of all variable exponents  $p(\cdot)$ such that $0<p_-\leq p_+<\infty.$
 The variable Lebesgue space $L_{p(\cdot)}=L_{p(\cdot)}(\Omega)$ is the
collection of all measurable functions $f$ defined on $(\Omega,\mathcal {F},\mathbb{P})$ such that for some $\lambda>0$,
$$\rho({f}\left/{\lambda}\right)=\int_{\Omega}\Bigg(\frac{|f(x)|}{\lambda}\Bigg)^{p(x)}d\mathbb{P}<\infty.$$
This becomes a quasi-Banach function space when it is equipped with the
quasi-norm
$$\|f\|_{p(\cdot)}:=\inf\{\lambda>0:\rho({f}\left/{\lambda}\right)\leq1\}.$$
For any $f\in L_{p(\cdot)}$, we have $\rho(f)\leq 1$   if and only if $\|f\|_{p(\cdot)}\leq 1$; see \cite[Theorem 1.3]{Fan2001}.
In the sequel, we always use the symbol
$$\underline{p}=\min\{p_-,1\}.$$
Throughout the paper, the variable exponent $p'(\cdot)$ is defined pointwise by
$$\frac{1}{p(x)}+\frac{1}{p'(x)}=1, \quad x\in \Omega. $$
We present some basic properties here (see \cite{Nakai2012}):
\begin{enumerate}
\item  $\|f\|_{p(\cdot)} \geq 0$;  $\|f\|_{p(\cdot)}=0
\Leftrightarrow f\equiv 0$.

\item  $\|cf\|_{p(\cdot)}=|c| \cdot \|f\|_{p(\cdot)}$ for
$c\in \mathbb{C}$.

\item for $0<b\leq \underline{p}$, we have
\begin{equation} \label{equation of b trangile inequality}
\|f+g\|_{p(\cdot)}^{b} \leq
\|f\|_{p(\cdot)}^{b} +\|g\|_{p(\cdot)}^{b}.\end{equation}
\end{enumerate}

\begin{lemma}[{\cite[Proposition 2.21]{Cruz2013}}] \label{lemma for equal 1} Let $p(\cdot) \in
\mathcal{P}(\Omega)$. If $f\in L_{p(\cdot)}(\Omega)$ and
$\|f\|_{p(\cdot)}\neq 0$, then
$$\int_\Omega \left| \frac{f(x)}{\|f\|_{p(\cdot)}} \right|^{p(x)} d\mathbb{P}=1.$$
\end{lemma}

\begin{lemma} [{\cite[Corollary 2.28]{Cruz2013}}] \label{lemma of Holer inequality} Let
$p(\cdot), q(\cdot),r(\cdot) \in \mathcal{P}(\Omega)$ satisfy
$$\frac{1}{p(x)} = \frac{1}{q(x)} + \frac{1}{r(x)},\quad x\in\Omega.$$
Then there exists a constant $C$ such that for all $f \in
L_{q(\cdot)}$ and $g\in L_{r(\cdot)}$, we have $fg \in L_{p(\cdot)}$ and
$$\|fg\|_{p(\cdot)} \leq C \|f\|_{q(\cdot)}\|g\|_{r(\cdot)}.$$
\end{lemma}

\begin{lemma} [{\cite[Theorem 2.8]{Kovacik1991}}] \label{lemma for embedding}
Let $p(\cdot),q(\cdot)\in \mathcal P(\Omega)$. If $p(\cdot)\leq q(\cdot)$, then for every $f\in L_{q(\cdot)}$ we have
$$\|f\|_{p(\cdot)}\leq 2\|f\|_{q(\cdot)}.$$
\end{lemma}

\begin{lemma}[{\cite[Theorem 2.34]{Cruz2013}}] \label{lem:duality for variabl p}
Let $p(\cdot)\in \mathcal P(\Omega)$ with $1\leq p_-$. Then
$$\|f\|_{p(\cdot)}\approx \sup\int_{\Omega}fgd\mathbb P,$$
where the supremum is taken over all $g\in L_{p'(\cdot)}$ with $\|g\|_{p'(\cdot)}\leq 1$.
\end{lemma}

\subsection{Variable Lorentz spaces $L_{p(\cdot),q}$} In this subsection, we introduce the definition of Lorentz spaces $L_{p(\cdot),q}(\Omega)$ with variable exponents $p(\cdot)\in \mathcal P(\Omega)$ and $0<q\leq \infty $ is a constant. For more information about general cases
$L_{p(\cdot),q(\cdot)}(\Omega)$, we refer the reader to \cite{Kempka2014}.
Following \cite{Kempka2014}, we introduce the  definition below.
\begin{definition}
Let  $p(\cdot)\in\mathcal{P}(\Omega)$ and $0<q\leq \infty$. Then $L_{p(\cdot), q}(\Omega) $ is the collection of all measurable functions $f$
such that
 \begin{equation*}
 \|f\|_{L_{p(\cdot),q}}:=\left\{
 \begin{array}{lr}
 \left( \int_0^\infty \lambda^q \|\chi_{\{|f|>\lambda\}}\|_{p(\cdot)}^q\frac{d\lambda}{\lambda}\right)^{1/q}, \quad \quad &q<\infty,\\
 \sup_\lambda\lambda\|\chi_{\{|f|>\lambda\}}\|_{p(\cdot)},\quad \quad & q=\infty
 \end{array}
 \right.
 \end{equation*}
is finite.
\end{definition}

\begin{remark} \label{remark for lorentz norm} According to \cite[Theorem 3.1]{Kempka2014},  the spaces $L_{p(\cdot),q}$ are quasi-Banach spaces. Moreover,
\begin{enumerate}
\item It is similar to the classical case that  the equations above can be  discretized:
$$\|f\|_{L_{p(\cdot),q}}\approx \left(\sum_{k=-\infty}^\infty2^{kq}\|\chi_{\{|f|>2^k\}}\|_{p(\cdot)}^q\right)^{1/q},$$
and
$$\|f\|_{L_{p(\cdot),\infty}}\leq 2 \sup_{k\in \mathbb Z}2^k \|\chi_{\{|f|>2^k\}}\|_{p(\cdot)} \leq 2\|f\|_{L_{p(\cdot),\infty}}.$$

\item By \cite[Lemma 2.4]{Kempka2014}, we have
$$ \sup_{k\in \mathbb Z}2^k \|\chi_{\{|f|>2^k\}}\|_{p(\cdot)}=\inf\left\{\lambda>0:\sup_{k\in\mathbb Z}\int_\Omega\left(\frac{2^k\chi_{\{|f|>2^k\}}}{\lambda}\right)^{p(x)}d\mathbb P\leq 1\right\}.$$

\item  Let $A\in \mathcal F$. A simple calculation based on (1) above shows that
 $$\|\chi_{A}\|_{L_{p(\cdot),q}}\approx\|\chi_{A}\|_{p(\cdot)}.$$
\end{enumerate}
\end{remark}

We now  show the dominated convergence theorem in $L_{p(\cdot),q}$. We begin with the following definition.

\begin{definition}
Let $p(\cdot)\in\mathcal P(\Omega)$ and   $0<q\leq\infty$.
A function $f \in L_{p(\cdot),q}$ is said to have absolutely continuous quasi-norm in $L_{p(\cdot),q}$ if
$$\lim_{n\rightarrow \infty}\|f\chi_{A_n}\|_{L_{p(\cdot),q}}=0$$
for every sequence
$(A_n)_{n\geq0}$ satisfying $\mathbb P(A_n)\rightarrow 0$ as $n\rightarrow \infty$.
\end{definition}

The next result shows that if $0<p_-\leq p_+<\infty$ and $0<q<\infty$, then  all the elements in $L_{p(\cdot),q}(\Omega)$ have absolutely continuous quasi-norm.
\begin{lemma} \label{lemma of abs q}
Let $p(\cdot)\in \mathcal P(\Omega)$ and $0<q<\infty$. Then for every $f\in L_{p(\cdot),q}$,
$f$ has absolutely continuous quasi-norm.
\end{lemma}
\begin{proof}
Since $f\in L_{p(\cdot),q}$, for any $\varepsilon>0$, there exists $N_1\in \mathbb N$ such that
$$\left(\sum_{k=N_1}^\infty2^{kq}\|\chi_{\{|f|>2^k\}}\|_{p(\cdot)}^q\right)^{1/q}<\varepsilon.$$
By the definition of $\{A_n\}_{n}$, there exists $N_2\in \mathbb N $ such that $\mathbb P(A_n) <\big(\frac{\varepsilon}{2^{N_1}}\big)^{p_+}$
for $n\geq N_2$. Now let $n\geq N_2$. By Lemma \ref{lemma for embedding}, we have
\begin{align*}
\|f\chi_{A_n}\|_{L_{p(\cdot),q}} &\lesssim  \Big(\sum_{k=N_1}^\infty2^{kq}\|\chi_{\{|f|>2^k\}}\|_{p(\cdot)}^q\Big)^{1/q}+
\Big(\sum_{k=-\infty}^{N_1-1}2^{kq}\|\chi_{A_n}\|_{p(\cdot)}^q\Big)^{1/q}
\\&\leq \varepsilon + \Big(\sum_{k=-\infty}^{N_1-1}2^{kq}\cdot (2\|\chi_{A_n}\|_{p_+})^q\Big)^{1/q}
\\&< \varepsilon + 2\Big(\sum_{k=-\infty}^{N_1-1}2^{(k-N_1)q}\Big)^{1/q}\varepsilon =3\varepsilon.
\end{align*}
This finishes the proof.
\end{proof}

The following well-known example (see \cite[Example 2.5]{Liu2010}) shows that not all functions in $L_{p(\cdot),\infty}$ have absolutely continuous quasi-norm.
\begin{example}  Consider the function $f(x)=x^{-1/p}$ on $\Omega=(0,1)$ associated with Lebesgue measure $\mathbb P$. Then, by a simple calculation, $f\in L_{p,\infty}$ $(0<p<\infty)$ and $f$ does not have  absolutely continuous quasi-norm.
\end{example}

Next we introduce a closed subspace of $L_{p(\cdot),\infty}$, in which simple functions are dense.

\begin{definition}
Let $p(\cdot)\in \mathcal P(\Omega)$.
We define $\mathscr L_{p(\cdot),\infty}(\Omega)$ as the set of measurable functions $f$ such that
$$\lim_{n\rightarrow \infty}\|f\chi_{A_n}\|_{L_{p(\cdot),\infty}}=0$$
for every sequence
$(A_n)_{n\geq0}$ satisfying $\mathbb P(A_n)\rightarrow 0$ as $n\rightarrow \infty$.
\end{definition}

\begin{lemma} \label{lemma for closed absolute norm}Let $p(\cdot)\in \mathcal P(\Omega)$. Then  $\mathscr L_{p(\cdot),\infty}$ is a closed subspace of
 $ L_{p(\cdot),\infty}$.
\end{lemma}

\begin{proof}Let $(f_n)_{n\geq1}\subset \mathscr L_{p(\cdot),\infty}$ be a Cauchy sequence in $ L_{p(\cdot),\infty}$. Then there exists $f\in  L_{p(\cdot),\infty}$ such that
$$\lim_{n\rightarrow \infty}\|f_n-f\|_{ L_{p(\cdot),\infty}}=0.$$
Choose $A_k \subset \Omega$ such that $\mathbb P (A_k)\rightarrow0$ as $k\rightarrow \infty$.
Hence, for any $k$ we have
\begin{align*}
\|f\chi_{A_k}\|_{ L_{p(\cdot),\infty}}&\leq 2^{\frac {1} {\underline p}}(\|f_n\chi_{A_k}\|_{ L_{p(\cdot),\infty}}+\|(f_n-f)\chi_{A_k}\|_{ L_{p(\cdot),\infty}})
\\& \leq 2^{\frac {1} {\underline p}}(\|f_n\chi_{A_k}\|_{ L_{p(\cdot),\infty}}+\|f_n-f\|_{ L_{p(\cdot),\infty}}).
\end{align*}
Since $f_n \in \mathscr L_{p(\cdot),\infty}$, by taking $n\rightarrow \infty$, we get
$$\lim_{k\rightarrow\infty}\|f\chi_{A_k}\|_{L_{p(\cdot),\infty}}=0,$$
which implies that $f\in \mathscr L_{p(\cdot),\infty}$.
\end{proof}

\begin{lemma}
Let $p(\cdot)\in \mathcal P(\Omega)$. Then we have $L_{p(\cdot)}\subset \mathscr L_{p(\cdot),\infty}\subset  L_{p(\cdot),\infty}$.
\end{lemma}

\begin{proof} By Lemma \ref{lemma for closed absolute norm}, it suffices to show  $L_{p(\cdot)}\subset \mathscr L_{p(\cdot),\infty}$.
Let $f\in L_{p(\cdot)}$. Applying  Lemma \ref{lemma for equal 1}, we get
\begin{align*}
\int_\Omega\left( \frac{2^k\chi_{\{|f|>2^k\}}}{\|f\|_{p(\cdot)}}(x)\right)^{p(x)}d\mathbb P &\leq \int_\Omega\left( \frac{|f(x)|}{\|f\|_{p(\cdot)}}\right)^{p(x)}d\mathbb P=1
\end{align*}
for any $k\in \mathbb Z$. We conclude by Remark \ref{remark for lorentz norm} that
\begin{equation}\label{inequality for embedding}
\|f\|_{L_{p(\cdot),\infty}} \leq \|f\|_{p(\cdot)}.
\end{equation}
Let $A_k \subset \Omega$ such that $\mathbb P (A_k)\rightarrow0$ as $k\rightarrow \infty$. We get $\rho (f\chi_{A_k})\rightarrow0$ as $k\rightarrow \infty$ since $f\in L_{p(\cdot)}$.
Note that $0<p_-\leq p_+<\infty$.
Hence, by \cite[Theorem 2.68]{Cruz2013}, we deduce that $\|f\chi_{A_k}\|_{p(\cdot)}\rightarrow0$ as $k\rightarrow \infty$.
It follows  from inequality \eqref{inequality for embedding} that
$$\lim_{k\rightarrow \infty}\|f\chi_{A_k}\|_{L_{p(\cdot),\infty}}=0.$$
Now, we obtain that $f\in \mathscr L_{p(\cdot),\infty}$. The proof is complete.
\end{proof}

\begin{lemma} \emph{(Dominated convergence theorem) }\label{lemma of control}Let $p(\cdot)\in \mathcal P(\Omega)$ and $0<q\leq \infty$. Assume  that $f_n,f, g\in L_{p(\cdot),q}$ satisfies $f_n\rightarrow f$ a.e. and $|f_n|\leq g$ for every $n\geq1$.
If $g$ has absolutely continuous quasi-norm, then
$$\lim_{n\rightarrow \infty}\|f_n-f\|_{L_{p(\cdot),q}}=0.$$
\end{lemma}

\begin{proof}
Since $g$ has absolutely continuous quasi-norm, for any $\varepsilon>0$ there exists $N_1$ such that $\|g\chi_{\{g>N_1\}}\|_{L_{p(\cdot),q}}<\varepsilon.$ Clearly, $|f_n-f|\leq 2N_1$ on the set $\{g\leq N_1\}$.
Note that, by Remark \ref{remark for lorentz norm}(3),
 \begin{align*}
 \|(f_n-f)\chi_{\{g\leq N_1\}}\|_{L_{p(\cdot),q}}&=\|(f_n-f)\chi_{\{g\leq N_1\}}\chi_{\{f_n\neq f\}}\|_{L_{p(\cdot),q}}
 \\&\leq 2N_1\|\chi_{\{f_n\neq f\}}\|_{L_{p(\cdot),q}}=2N_1\|\chi_{\{f_n\neq f\}}\|_{p(\cdot)}.
 \end{align*}
Then, by the facts $\chi_{\{f_n\neq f\}}\rightarrow 0$ as $n\rightarrow\infty$ and the dominated convergence theorem in $L_{p(\cdot)}$ (see \cite{Cruz2013}), there exists $N_2$ such that $\|(f_n-f)\chi_{\{g\leq N_1\}}\|_{L_{p(\cdot),q}}<\varepsilon$ for $n\geq N_2$. Finally, for $n\geq N_2$,
\begin{align*}
\|f_n-f\|_{L_{p(\cdot),q}} &\lesssim \|(f_n-f)\chi_{\{g\leq N_1\}}\|_{L_{p(\cdot),q}}+\|(f_n-f)\chi_{\{g>N_1\}}\|_{L_{p(\cdot),q}}
\\& \lesssim \varepsilon + 2\|g\chi_{\{g>N_1\}}\|_{L_{p(\cdot),q}} <3\varepsilon,
\end{align*}
which completes the proof.
\end{proof}

\subsection{Variable martingale Hardy spaces }
In this subsection, we introduce some standard notations from martingale theory. We refer to the books  \cite{Garsia1973, Long1993book, Weisz1994book} for the theory of classical
martingale space. Let
$(\Omega,\mathcal {F},\mathbb{P})$ be a complete probability space.
Let the subalgebras $(\mathcal F_n)_{n\geq0}$ be increasing such that $\mathcal F= \sigma(\cup_{n\geq0}\mathcal F_n)$,
and let  $\E_{n}$ denote the conditional expectation operator relative to
${\mathcal{F}_n}$. A sequence of
measurable functions $f=(f_n)_{n\geq0}\subset L_1(\Omega)$ is called
a martingale with respect to $(\mathcal {F}_n)_{n\geq0}$ if $\E_{n}(f_{n+1})=f_n$ for every $n\geq0.$
For a martingale $f=(f_n)_{n\geq0}$ let $f_{-1}:=0$ and
$$
d_nf=f_n-f_{n-1},\quad n\geq0,
$$
denote the martingale difference.
If in addition $f_n\in L_{p(\cdot)}$, then $f$ is called an $L_{p(\cdot)}$-martingale with
respect to $(\mathcal {F}_n)_{n\geq 0}$. In this case, we set
$$\|f\|_{p(\cdot)}=\sup_{n\geq0}\|f_n\|_{p(\cdot)}.$$
If $\|f\|_{p(\cdot)}<\infty$, $f$ is called a bounded
$L_{p(\cdot)}$-martingale and it is denoted by $f\in L_{p(\cdot)}$. For a
martingale relative to $(\Omega,\mathcal {F},\mathbb{P};(\mathcal
{F}_n)_{n\geq 0})$, we define the maximal function,  the square
function and the conditional square
function of $f$, respectively, as follows $(f_{-1}=0)$:
$$M_m(f)=\sup_{0\leq n\leq m}{|f_n|}, \quad M(f)=\sup_{n\geq0} |f_n|;$$
$$S_m(f)=\left(\sum_{n=0}^m|d_nf|^2\right)^{1/2},\quad S(f)=\left(\sum_{n=0}^{\infty}|d_nf|^2\right)^{1/2};$$
$$s_m(f)=\left(\sum\limits_{n=0}^{m}\E_{{n-1}}|d_n f|^2\right)^{\frac{1}{2}},
\quad s(f)=\left(\sum\limits_{n=0}^{\infty}\E_{{n-1}}|d_n f|^2\right)^{\frac{1}{2}}.$$
 Denote by $\Lambda$ the collection of all sequences
 $(\lambda_n)_{n\geq0}$ of non-decreasing, non-negative and adapted
 functions with $\lambda_\infty=\lim_{n\rightarrow
 \infty}\lambda_n$. Let $p(\cdot)\in \mathcal P(\Omega) $ and $0<q\leq \infty$. The variable martingale Hardy spaces associated with variable Lebesgue spaces $L_{p(\cdot)}$ are defined as follows:
  $$H_{p(\cdot)}^M = \{f=(f_n)_{n\geq 0}:\|f\|_{H_{p(\cdot)}^M}=\|M(f)\|_{{p(\cdot)}}<\infty\};$$
 $$H_{p(\cdot)}^S = \{f=(f_n)_{n\geq
 0}:\|f\|_{H_{p(\cdot)}^S}=\|S(f)\|_{{p(\cdot)}}<\infty\};$$
 $$ H_{p(\cdot)}^s = \{f=(f_n)_{n\geq 0}:\|f\|_{H_{p(\cdot)}^s}=\|s(f)\|_{{p(\cdot)}}<\infty\};$$
 $${Q}_{p(\cdot)} = \{f=(f_n)_{n\geq 0}:\exists(\lambda_n)_{n\geq0}\in\Lambda,\,\, \textrm{s.t.}\,\, S_n(f)
 \leq \lambda_{n-1}, \lambda_\infty\in L_{p(\cdot)}\},$$
 $$\|f\|_{{Q}_{p(\cdot)}}=\inf_{(\lambda_n) \in \Lambda}\|\lambda_\infty\|_{{p(\cdot)}};$$
 $${P}_{p(\cdot)} = \{f=(f_n)_{n\geq 0}:\exists(\lambda_n)_{n\geq0}\in\Lambda,\,\, \textrm{s.t.}\,\,
 |f_n| \leq \lambda_{n-1}, \lambda_\infty\in L_{p(\cdot)}\},$$
 $$\|f\|_{{P}_{p(\cdot)}}=\inf_{(\lambda_n) \in \Lambda}\|\lambda_\infty\|_{{p(\cdot)}}.$$
Similarly, the variable martingale Lorentz-Hardy spaces associated with variable Lorentz spaces $L_{p(\cdot),q}$ are defined as follows:
 $$H_{p(\cdot),q}^M = \{f=(f_n)_{n\geq 0}:\|f\|_{H_{p(\cdot),q}^M}=\|M(f)\|_{L_{p(\cdot),q}}<\infty\};$$
 $$H_{p(\cdot),q}^S = \{f=(f_n)_{n\geq
 0}:\|f\|_{H_{p(\cdot),q}^S}=\|S(f)\|_{L_{p(\cdot),q}}<\infty\};$$
 $$ H_{p(\cdot),q}^s = \{f=(f_n)_{n\geq 0}:\|f\|_{H_{p(\cdot),q}^s}=\|s(f)\|_{L_{p(\cdot),q}}<\infty\};$$
 $${Q}_{p(\cdot),q} = \{f=(f_n)_{n\geq 0}:\exists(\lambda_n)_{n\geq0}\in\Lambda,\,\, \textrm{s.t.}\,\, S_n(f)
 \leq \lambda_{n-1}, \lambda_\infty\in L_{p(\cdot),q}\},$$
 $$\|f\|_{{Q}_{p(\cdot),q}}=\inf_{(\lambda_n) \in \Lambda}\|\lambda_\infty\|_{L_{p(\cdot),q}};$$
 $${P}_{p(\cdot),q} = \{f=(f_n)_{n\geq 0}:\exists(\lambda_n)_{n\geq0}\in\Lambda,\,\, \textrm{s.t.}\,\,
 |f_n| \leq \lambda_{n-1}, \lambda_\infty\in L_{p(\cdot),q}\},$$
 $$\|f\|_{{P}_{p(\cdot),q}}=\inf_{(\lambda_n) \in \Lambda}\|\lambda_\infty\|_{L_{p(\cdot),q}}.$$
We define $\mathscr H_{p(\cdot),\infty}^M$ as the space of all martingales such that $M(f)\in \mathscr L_{p(\cdot),\infty}$.
Analogously, we can define
 $\mathscr H_{p(\cdot),\infty}^S$ and $\mathscr H_{p(\cdot),\infty}^s$, respectively.

\begin{remark}
If $p(\cdot)=p$ is a constant, then the above definitions of variable Hardy spaces go back to the classical definitions stated in \cite{Garsia1973} and \cite{Weisz1994book}.
\end{remark}

\subsection{The Doob maximal operator}

In the sequel of the paper, we will often suppose that every
$\sigma$-algebra $\mathcal{F}_n$ is generated by countably many atoms. Recall that $B\in\mathcal{F}_n$ is called an atom, if for any $A\subset B$
with $A\in\mathcal{F}_n$ satisfying $\mathbb{P}(A)<\mathbb{P}(B)$, we have $\mathbb{P}(A)=0$. We denote by $A(\mathcal{F}_n)$ the set of all atoms in $\mathcal{F}_n$.
It is clear that for $f\in L_1(\Omega)$
$$\E_n(f)=\sum_{A\in A(\mathcal{F}_n)} \Bigg( \frac{1}{\mathbb{P}( A)} \int_{A}f(x) d\mathbb{P} \Bigg) \chi_{A}, \quad n\in \N.$$

We now recall the definition of regularity. The stochastic basis $(\mathcal {F}_n)_{n\geq 0}$ is said to be
 regular, if for $n\geq 0$ and $A\in \mathcal {F}_n$, there exists
 $B\in \mathcal {F}_{n-1} $ such that $A\subset B$ and $P(B)\leq R
 P(A)$, where $R$ is a positive constant independent of $n$. A martingale is
 said to be regular if it is adapted to a regular $\sigma$-algebra
 sequence. This implies that there exists a constant $R>0$
 such that
  \begin{equation}\label{regular constant}
  f_n\leq Rf_{n-1}
  \end{equation}
  for all non-negative
 martingales $(f_n)_{n\geq0}$ adapted to the stochastic basis $(\mathcal {F}_n)_{n\geq
 0}$. We refer the reader to \cite[Chapter 7]{Long1993book} for more details.

In the following example, the so-called dyadic stochastic basis $(\mathcal {F}_n)_{n\geq 0}$ is regular and every $\mathcal{F}_n$ is generated by finitely many atoms.
\begin{example}\label{ex:regular}
Let $([0,1),\mathcal F, \mu)$ be a probability space such that $\mu$ is the Lebesgue measure and the subalgebras $\{\F_n\}_{n\geq0}$ are defined by
$$\F_n= \sigma\left\{[j 2^{-n},(j+1)2^{-n}),j=0,\cdots,2^n-1 \right\}.
$$
Then $\{\mathcal{F}_n\}_{n\geq0}$ is regular. A martingale with respect to $\{\mathcal{F}_n\}_{n\geq0}$ is called a dyadic martingale. There are a lot of other examples for (regular) $\sigma$-algebras generated by finitely many atoms, see the Vilenkin $\sigma$-algebras in Weisz \cite{Weisz1994book}.
\end{example}

In the sequel of the paper, instead of the log-Hölder continuity \eqref{local log} and \eqref{log at infinity}, we will suppose that every $\sigma$-algebra $\mathcal{F}_n$ is generated by countably many atoms and there exists an absolute constant $K_{p(\cdot)}\geq1$ depending only on $p(\cdot)$ such that
\begin{equation}\label{log} 
\mathbb{P}(A)^{p_-(A)-p_+(A)} \leq K_{p(\cdot)}, \quad\forall A\in \bigcup_n A(\mathcal{F}_n).
\end{equation}

Note that in this paper, under condition \eqref{log}, we also mean that every $\sigma$-algebra $\mathcal{F}_n$ is generated by countably many atoms.

\begin{remark}\label{rem:variable}
There exist a lot of functions $p(\cdot)$ satisfying \eqref{log}. In fact, if the measurable function $p(\cdot)$, which is defined on $[0,1)$, satisfies the log-Hölder continuity \eqref{local log}, then, by
\cite[Lemma 3.24]{Cruz2013}, we find that $p(\cdot)$ satisfies \eqref{log}. For concrete examples we mention the function $a+cx$ for parameters $a$ and $c$ such that the function is positive ($x \in  [0,1)$). All positive Lipschitz functions with order $0< \alpha \leq 1$ also satisfy \eqref{log}. Note that the condition \eqref{log at infinity} disappears on $[0,1)$.
\end{remark}

In this subsection, we prove a variable version of the dual Doob inequality and the weak type inequality for Doob's maximal operator.
We  provide two lemmas from \cite{Hao2015}.

\begin{lemma} [{\cite[Lemma 4.1]{Hao2015}}]\label{lemma of norm of set}  Let $p(\cdot)\in \mathcal{P}(\Omega)$ satisfy \eqref{log}. Then,
for any atom $B\in \cup_n A(\mathcal F_n)$,
$$\mathbb{P}(B)^{1/{p_-(B)}}  \approx  \mathbb{P}(B)^{1/{p(x)}} \approx \mathbb{P}(B)^{1/{p_+(B)}} \approx  \|\chi_B\|_{p(\cdot)},\quad \forall x\in B.$$
\end{lemma}

\begin{lemma}[{\cite[Lemma 4.1]{Hao2015}}]\label{lemma of log condition}
Let $p(\cdot) \in \mathcal{P}(\Omega)$ satisfy \eqref{log}  with $p_->1$.
\begin{enumerate}
\item Then for any atom $B\in \cup_n A(\mathcal F_n)$,
 $$\|\chi_B\|_1 \approx \|\chi_B\|_{p(\cdot)}\|\chi_B\|_{p'(\cdot)}.$$

\item Let $q(\cdot) \in \mathcal{P}(\Omega)$ satisfy \eqref{log} with $q_->1$. Then,
for any atom $B\in \cup_n A(\mathcal F_n)$,
$$\|\chi_B\|_{r(\cdot)} \approx \|\chi_B\|_{p(\cdot)}\|\chi_B\|_{q(\cdot)},$$
where
$$\frac{1}{r(x)}= \frac{1}{p(x)}+\frac{1}{q(x)},\quad x\in \Omega.$$
\end{enumerate}
\end{lemma}

The following lemma can be proved in the same way as Lemma 3.4 in Jiao et al. \cite{Jiao2016}.

\begin{lemma}\label{lemma of conditional expc} Let $p(\cdot)\in \mathcal{P}(\Omega)$, $1\leq p_-\leq p_+ < \infty$, satisfy \eqref{log}.
Suppose that $f\in L_{p(\cdot)}$ with $\left\|f\right\|_{p(\cdot)}\leq 1/2$ and $f=f \chi_{\{|f| \geq 1\}}$. Then, for any atom $A\in \cup_n A(\mathcal F_n)$, $x\in A$,
$$
\left(\frac{1}{\mathbb{P}(A)} \int_A |f(t)| \, dt\right)^{p(x)} \leq
\left(\frac{K_{p(\cdot)}}{\mathbb{P}(A)} \int_A |f(t)|^{p(t)}\, dt\right).
$$
\end{lemma}

The next result is taken from \cite[Theorem 4.1]{Kempka2014}.
\begin{lemma}\label{lemma for interpolation}
Let $p(\cdot)\in \mathcal P(\Omega)$ with $p_+<\infty$, $0<q\leq \infty$, $0<\theta<1$ and
$$\frac{1}{\tilde p (\cdot)}=\frac {1-\theta}{p(\cdot)}.$$
Then
$$(L_{p(\cdot)}, L_\infty)_{\theta,q}=L_{\tilde p(\cdot),q}.$$
\end{lemma}

With the help of Lemma \ref{lemma of conditional expc}, we can prove the following Doob maximal inequality similarly to \cite{Jiao2016}.

\begin{theorem}\label{thm:maximal inequality}
Let $p(\cdot)\in\mathcal P(\Omega)$ satisfy \eqref{log} with $1< p_-\leq p_+ < \infty$.
Then, for any  $f\in L_{p(\cdot)}$,
$$\|M(f)\|_{p(\cdot)}\lesssim \|f\|_{p(\cdot)}.$$
\end{theorem}

\begin{corollary}\label{c10}
Let $p(\cdot)\in\mathcal P(\Omega)$ satisfy \eqref{log} with $1< p_-\leq p_+ < \infty$ and $0<q\leq \infty$. Then
$$
\left\|M(f)\right\|_{L_{p(\cdot),q}} \lesssim \left\|f\right\|_{L_{p(\cdot),q}}.
$$
\end{corollary}

\begin{proof}
It follows from Theorem \ref{thm:maximal inequality} that $M$ is bounded from $L_{p(\cdot)}$ to $L_{p(\cdot)}$.
Hence, by combining the fact that $M$ is bounded from $L_{\infty}$ to $L_{\infty}$ and  Lemma \ref{lemma for interpolation},  we find that $M$ is bounded from $L_{\tilde p(\cdot),q}$ to $L_{\tilde p(\cdot),q}$.
\end{proof}

The previous two results imply immediately the next corollary.

\begin{corollary}\label{c100}
Let $p(\cdot)\in\mathcal P(\Omega)$ satisfy \eqref{log} with $1< p_-\leq p_+ < \infty$ and $0<q\leq \infty$. Then $H_{p(\cdot)}^M$ is equivalent to $L_{p(\cdot)}$ and $H_{p(\cdot),q}^M$ to $L_{p(\cdot),q}$ with the inequalities
$$
\left\|f\right\|_{L_{p(\cdot)}}\leq \left\|f\right\|_{H_{p(\cdot)}^M} \lesssim \left\|f\right\|_{L_{p(\cdot)}},
\qquad
\left\|f\right\|_{L_{p(\cdot),q}}\leq \left\|f\right\|_{H_{p(\cdot),q}^M} \lesssim \left\|f\right\|_{L_{p(\cdot),q}}.
$$
\end{corollary}

The following result is a variable version of the dual Doob's inequality. For its classical case, we refer the reader to \cite{Pisier2016} or \cite{Hytonen2016book}.

\begin{proposition}\label{dual Doob}
Let $p(\cdot)\mathcal \in P(\Omega)$ satisfy \eqref{log} with $1< p_-\leq p_+ < \infty$.
Let $(\theta_n)_{n\geq0}$ be a sequence of arbitrary random variables.
Then
$$\left\|\sum |\mathbb E_n(\theta_n)|\right\|_{p(\cdot)}\lesssim \left\|\sum |\theta_n|\right\|_{p(\cdot)}.$$
\end{proposition}

\begin{proof}
Since $|\mathbb E_n (\theta_n)|\leq \mathbb E_n (|\theta_n|)$,  it suffices to prove this result assuming $\theta_n\geq0$.
Consider $0\leq f\in L_{p'(\cdot)}$ with $\|f\|_{p'(\cdot)}=1$ (see Lemma \ref{lem:duality for variabl p}) such that
$$
\left\|\sum |\mathbb E_n(\theta_n)|\right\|_{p(\cdot)}= \mathbb E \Big(\sum \mathbb E_n(\theta_n) f\Big).
$$
Then, by Doob's inequality (Theorem \ref{thm:maximal inequality}),
\begin{align*}
\left\|\sum |\mathbb E_n(\theta_n)|\right\|_{p(\cdot)}&=\sum \mathbb E  [\theta_n\mathbb E_n(f)]
\lesssim \left\|\sum |\theta_n|\right\|_{p(\cdot)} \|\sup_n \mathbb E_n(f)\|_{p'(\cdot)}
\lesssim \left\|\sum |\theta_n|\right\|_{p(\cdot)}.
\end{align*}
\end{proof}

By a similar argument to the proof of the above proposition,  we may get the following Stein's inequality:
Let $p(\cdot)$ satisfy \eqref{log} with $1\leq r< p_-\leq p_+ < \infty$ for some $r$. Let $(\theta_n)_{n\geq0}$ be a sequence of arbitrary random variables.
Then
$$\left\|\left[\sum |\mathbb E_n(\theta_n)|^r\right]^{\frac 1r}\right\|_{p(\cdot)}\lesssim \left\|\left(\sum |\theta_n|^r\right)^{\frac 1r}\right\|_{p(\cdot)}.$$
Taking $(\theta_n)_{n\geq0}=(|d_{n+1}f|^2)_{n\geq0}$,  the following result can be deduced from above proposition.
\begin{corollary}
Let $p(\cdot)\in\mathcal P(\Omega)$ satisfy \eqref{log} with $2< p_-\leq p_+ < \infty$.
Then
$$\|f\|_{H_{p(\cdot)}^s}\lesssim\|f\|_{H_{p(\cdot)}^S}.$$
\end{corollary}

Theorem \ref{thm:maximal inequality} says that $M$ is bounded from $L_{p(\cdot)}$ to $L_{p(\cdot)}$ and hence from $L_{p(\cdot)}$ to $L_{p(\cdot),\infty}$ if $p_->1$. In the next theorem, we extend the last statement to $p_-\geq 1$. This result covers the classical weak $(1,1)$ inequality for the Doob maximal operator.

\begin{theorem}
Let $p(\cdot)$ satisfy \eqref{log} with $1\leq p_-\leq p_+ < \infty$.
Then, for any  $f\in L_{p(\cdot)}$,
$$\sup_{\lambda>0} \lambda\|\chi_{\{M(f)>\lambda\}}\|_{p(\cdot)}\lesssim \|f\|_{p(\cdot)}.$$
\end{theorem}

\begin{proof}For $\lambda>0$ and $n\in \mathbb N$, we define the stopping time
$$\tau_n=\inf\{i\leq n: |f_i|>\lambda\},$$
with the convention that $\inf \emptyset =\infty$. Then
$$
\{M_nf>\lambda\}=\{\tau_n<\infty\}.
$$
It is easy to see that the $\sigma$-algebra $\mathcal F_{\tau_n}$ is generated by countable atoms as well.
Without loss of generality, we assume that $\|f\|_{p(\cdot)}\leq 1/2$. Then,
 by Lemma \ref{lemma of conditional expc}, we have the following estimate,
\begin{align*}
\int(\lambda \chi_{\{M_nf>\lambda\}})^{p(x)}d\mathbb P
&\leq \int_{\{\tau_n<\infty\}}|f_{\tau_n}|^{p(x)}d\mathbb P
\\& = \int \sum_{A\in  A(\mathcal F_{\tau_n})} \Bigg( \frac{1}{\mathbb{P}(A)} \int_{A}|f(x)| d\mathbb{P} \Bigg)^{p(x)} \chi_{A}d\mathbb P
\\
& \leq K \int \sum_{A\in  A(\mathcal F_{\tau_n})} \Bigg( \frac{1}{\mathbb{P}(A)} \int_{A} |f(x)|^{p(x)}+1 d\mathbb{P} \Bigg) \chi_{A} d\mathbb P
\\
&=K \int \E_{\tau_n}(|f(x)|^{p(x)}+1) d\mathbb P
\leq C.
\end{align*}

Set
$$
g_n=(\lambda \chi_{\{M_nf>\lambda\}})^{p(x)}, \qquad g=(\lambda \chi_{\{M(f)>\lambda\}})^{p(x)}.
$$
Then $g_n\leq g_{n+1}$ and $g_n$ converges to $g$ a.e. as $n\rightarrow \infty$. Using the monotone convergence theorem, we have
$$\int (\lambda \chi_{\{M(f)>\lambda\}})^{p(x)}d\mathbb P=\lim_{n\rightarrow \infty}\int (\lambda \chi_{\{M_nf>\lambda\}})^{p(x)}d\mathbb P \leq C,$$
which completes the proof.
\end{proof}

\section{Atomic decompositions} \label{sec3}

In this section, we investigate atomic decompositions for variable martingale Hardy spaces.
Let $\mathcal{T}$ be the set of all stopping times with respect to $(\mathcal{F}_n)_{n \geq 0}$. For a martingale
$f=(f_n)_{n\geq0}$ and $\tau \in \mathcal{T}$, we denote the stopped martingale by $f^\tau=(f_n^\tau)_{n\geq0}=(f_{n\wedge\tau})_{n\geq0}$,
where $a\wedge b =\min(a,b)$. We recall the definition of an atom.
\begin{definition} \label{definition for atom}
Let $p(\cdot)\in\mathcal{P}(\Omega)$. A
measurable function $a$ is called a $(1,p(\cdot),\infty)$-atom (or $(2,p(\cdot),\infty)$-atom or $(3,p(\cdot),\infty)$-atom, respectively) if
there exists a stopping time $\tau\in\mathcal{T}$ such that
\begin{enumerate}
\item $\E(a|\mathcal {F}_n)=0$,  $\forall\; n\leq \tau$,
\item $\|s(a)\|_\infty(\,\text{or} \ \|S(a)\|_\infty \ \text{,} \ \|M(a)\|_\infty, \ \text{respectively})\leq \frac{1}{\|\chi_{\{\tau<\infty\}}\|_{p(\cdot)}}.$
\end{enumerate}
$a$ is called a simple $(i,p(\cdot),\infty)$-atom if $\tau$ is the special stopping time $\tau = n \chi_A$ for some $A \in A(\mathcal F_n)$ and $n\in \mathbb N$ $(i=1,2,3)$.
\end{definition}

\subsection{Atomic decompositions of $H_{p(\cdot)}$}

The atomic characterization of $H_{p(\cdot)}$ has been shown in \cite{Jiao2016}. In this section, we generalize this atomic decomposition.

\begin{definition} Let $p(\cdot)\in\mathcal{P}(\Omega)$  and  $1<r\leq \infty$.
The atomic Hardy space $H_{p(\cdot)}^{\rm at, d,r}$ is defined as the space of all martingales $f=(f_n)_{n\geq 0}$ such that
\begin{equation}\label{decomposition Leb1}
f_n= \sum_{k\in \mathbb Z} \mu_{k}a^{k}_n,
\end{equation}
where $(a_{k})_{k\in \mathbb Z}$ is a sequence of  $(d,p(\cdot),r)$-atoms $(d=1,2,3)$ associated with stopping times $(\tau_k)_{k\in\mathbb Z}$ and  $(\mu_k)_{k\in \mathbb Z}$ is a sequence of positive numbers.
For a fixed $0<t<\underline p$ and $f\in H_{p(\cdot)}^{\rm at,d,r}$, define
$$\|f\|_{H_{p(\cdot)}^{\rm at,d,r}} = \inf \left\| \left[\sum_{k\in \mathbb Z} \left(\frac{\mu_k\chi_{\{\tau_k<\infty\}}}{\|\chi_{\{\tau_k<\infty\}}\|_{p(\cdot)}}\right)^{t}\right]^{\frac{1}{t}}\right\|_{p(\cdot)},$$
where the infimum is taken over all the decompositions of the form \eqref{decomposition Leb1}.
\end{definition}

The following atomic decompositions can be proved in the same way as in \cite{Jiao2016}.
\begin{theorem}\label{ad Leb}
Let $p(\cdot)\in \mathcal P(\Omega)$. Then
$$H_{p(\cdot)}^s= H_{p(\cdot)}^{\rm at, 1,\infty}$$
with equivalent quasi-norms.
\end{theorem}

Similarly, we have the following result (the proof is omitted, and in fact, the proof is just similar to the one of Theorem \ref{theorem of atomic decomposition for rest}).

\begin{theorem}\label{ad 23 Leb}
Let $p(\cdot)\in \mathcal P(\Omega)$. Then
$$Q_{p(\cdot)}= H_{p(\cdot)}^{\rm at, 2,\infty}, \quad P_{p(\cdot)}= H_{p(\cdot)}^{\rm at, 3,\infty}$$
with equivalent quasi-norms.
\end{theorem}

\begin{theorem}\label{proposition of atomic decomposition of MS Leb}
 Let $p(\cdot)\in\mathcal{P}(\Omega)$ satisfy  condition \eqref{log}. If $\{\mathcal {F}_n\}_{n\geq 0}$ is regular, then
 $$H_{p(\cdot)}^S= H_{p(\cdot)}^{\rm at,2,\infty},\quad H_{p(\cdot)}^M= H_{p(\cdot)}^{\rm at,3,\infty}$$
 with equivalent quasi-norms.
\end{theorem}

\begin{proof}
We only give the proof for the second equality since the other one is similar.
Take $f\in H_{p(\cdot)}^M$.
Consider the following stopping times with respect to $(\mathcal F_n)$,
$$\rho_k:=\inf\{n\in \mathbb N:|f_n|>2^k\}, \quad k\in \mathbb Z.$$
For fixed $k\in \mathbb Z$, let $F_j^k\in \mathcal F_{j-1}$ be the smallest set which contains $\{\rho_k=j\}$. In other words, if $\{\rho_k=j\}\in \mathcal F_j$ is decomposed into the disjoint union of atoms $\overline {I_{k,j,i}}\in A(\mathcal F_j)$ and $ {I_{k,j,i}}\in A(\mathcal F_{j-1})$ denotes the atom which
contains $\overline {I_{k,j,i}}$ and $\mathbb P({I_{k,j,i}}) \leq R \mathbb P(\overline{ {I_{k,j,i}}})$ (this is due to regularity, and $R$ is the constant as in \eqref{regular constant}), then $F_j^k=\cup_i I_{k,j,i}$. Define a new family of stopping times by
$$\tau_k(x):=\inf\{n\in \mathbb N:x\in F_{n+1}^k\}.$$
It is obvious that $\tau_k$ is non-decreasing. By Lemma \ref{lem:new ad2}, we have
$$\|\chi_{\{\tau_k<\infty\}}\|_{p(\cdot)}\lesssim \|\chi_{\{\rho_k<\infty\}}\|_{p(\cdot)}=\|\chi_{\{M(f)>2^k\}}\|_{p(\cdot)}\leq 2^{-k}\|M(f)\|_{p(\cdot)}\rightarrow 0$$
as $k\rightarrow \infty$,
which  deduces that
$$\lim_{k\rightarrow \infty }\mathbb P(\tau_k=\infty)=1.$$
Thus $\lim_{k\rightarrow \infty } \tau_k=\infty $ a.e. and
$$\lim_{k\rightarrow \infty } f_n^{\tau_k} =f_n \quad \mbox{a.e.}\quad (n\in \mathbb N).$$

We  define
$$\mu_k=3\cdot2^k \|\chi_{\{\tau_k<\infty\}} \|_{p(\cdot)}\quad\mbox{and}\quad a_n^k= \ \frac{ f_n^{\tau_{k+1}}-f_n^{\tau_{k}} }{\mu_k}.$$
It is not hard to check that $a^{k}=(a_n^k)_n$ is a $(3,p(\cdot),\infty)$-atom,
and $f=\sum_{k\in\mathbb{Z}} \mu_ka_k$.
Note that for every $k\in \mathbb Z$,
$$\{\tau_k<\infty\}=\sum_{j=0}^\infty\{\tau_k=j\}=\sum_{j=0}^\infty\sum_i I_{k,j,i},$$
where $I_{k,j,i}$'s are atoms in $A(\mathcal F_j)$.

Obviously,
\begin{align*}
Z:= \left\|\left[\sum_{k\in \mathbb Z}\left(\frac{\mu_{k}\chi_{\{\tau_k<\infty\}}}
{\|\chi_{\{\tau_k<\infty\}}\|_{p(\cdot)}}\right)^{t}\right]^{\frac{1}{t}}\right\|_{p(\cdot)} =
\left\|\left[\sum_{k\in \mathbb Z}\sum_{j=0}^{\infty}\sum_i \left(3\cdot 2^k\right)^{t} {\chi_{I_{k,j,i}}}
\right]^{\frac{1}{t}}\right\|_{p(\cdot)}.
\end{align*}
Using Lemma \ref{lem:duality for variabl p}, we may choose a positive function $g\in L_{(\frac{p(\cdot)}{t})'}$ with $\|g\|_{{(\frac{p(\cdot)}{t})'}}\leq 1$ such that
\begin{align*}
Z^{t}&= \int_\Omega \sum_{k\in \mathbb Z}\sum_{j=0}^{\infty}\sum_i \left(3\cdot 2^k\right)^{t}  \chi_{I_{k,j,i}} gd\mathbb P.
\end{align*}
Applying H\"older's inequality for some $r$ satisfying $\max(1,p_+/t)<r<\infty$, we find that
\begin{align*}
Z^{t}&\leq \sum_{k\in \mathbb Z}\sum_{j=0}^{\infty}\sum_i \left(3\cdot 2^k\right)^{t}  \mathbb P( {I_{k,j,i}})^{\frac {1}{r}} \left(\int_\Omega \chi_{ {I_{k,j,i}}} g^{r'}d\mathbb P\right)^{\frac 1{r'}}
\\&= \sum_{k\in \mathbb Z}\sum_{j=0}^{\infty}\sum_i \left(3\cdot 2^k\right)^{t}  \mathbb P( {I_{k,j,i}})
\left(\frac{1}{\mathbb P( {I_{k,j,i}})}\int_\Omega \chi_{{I_{k,j,i}}} g^{r'}d\mathbb P\right)^{\frac 1{r'}}
\\&\lesssim \sum_{k\in \mathbb Z}\sum_{j=0}^{\infty}\sum_i \left(3\cdot 2^k\right)^{t}  \mathbb P(\overline {I_{k,j,i}})
\left(\frac{1}{\mathbb P( {I_{k,j,i}})}\int_\Omega \chi_{{I_{k,j,i}}} g^{r'}d\mathbb P\right)^{\frac 1{r'}}
\\&\leq   \sum_{k\in \mathbb Z}\sum_{j=0}^{\infty}\sum_i \left(3\cdot 2^k\right)^{t} \int_\Omega \chi_{\overline {I_{k,j,i}}} [M(g^{r'})]^{\frac 1{r'}}d\mathbb P
\\& \lesssim \left\|\sum_{k\in \mathbb Z}\sum_{j=0}^{\infty}\sum_i \left(3\cdot 2^k\right)^{t} \chi_{\overline {I_{k,j,i}}}\right\|_{p(\cdot)/t} \|[M(g^{r'})]^{\frac 1{r'}}\|_{{(p(\cdot)/t)'}},
\end{align*}
where the first ``$\lesssim$" is due to the regularity.
Since ${p_+}/t<r<\infty$, we have $((p(\cdot)/t)')_+<\infty$ and
$r'<(p(\cdot)/t)'$.
Then, using Theorem \ref{thm:maximal inequality}, we obtain
$$\|[M(g^{r'})]^{\frac 1{r'}}\|_{{(p(\cdot)/t)'}} = \|M(g^{r'})\|_{{\frac 1{r'} (p(\cdot)/t)'}}^{\frac 1{r'}}\lesssim
\|g^{r'}\|_{{\frac 1{r'} (p(\cdot)/t)'}}^{\frac 1{r'}}
=\|g\|_{{(p(\cdot)/t)'}}\leq 1.$$
Observe that
$$
\left\|\sum_{k\in \mathbb Z}\sum_{j=0}^{\infty}\sum_i \left(3\cdot 2^k\right)^{t} \chi_{\overline {I_{k,j,i}}}\right\|_{p(\cdot)/t}^{\frac 1 {t}}
= \left\|\left[\sum_{k\in \mathbb Z} (3\cdot 2^k \chi_{\{M(f)>2^k\}})^{t} \right]^{\frac {1} {t}}\right\|_{p(\cdot)}.
$$
Since
$$\sum_{k\in \mathbb Z}  2^{kt} \chi_{\{M(f)>2^k\}}=\sum_{k\in\mathbb{Z}} 2^{kt}\sum_{j=k}^{\infty}\chi_{\{2^j<M(f)<2^{j+1}\}}\approx\sum_{j\in\mathbb{Z}}2^{jt}\chi_{\{2^j<M(f)<2^{j+1}\}}\lesssim M(f)^t,
$$
we have
$\|f\|_{H_{p(\cdot)}^{\rm at,3,\infty}} \leq Z\lesssim \|f\|_{H_{p(\cdot)}^{M}}.$
The converse inequality
$\|f\|_{H_{p(\cdot)}^{M}} \lesssim\|f\|_{H_{p(\cdot)}^{\rm at,3,\infty}}$
 can be easily  proved.  The proof is complete.
\end{proof}

The next lemma is used in the proof of the previous theorem.

\begin{lemma} \label{lem:new ad2} Let $p(\cdot)\in\mathcal{P}(\Omega)$ satisfy  condition \eqref{log} and $\{\mathcal {F}_n\}_{n\geq 0}$ be regular. Take the same stopping tines $\tau_k$ and $\rho_k$ as in the proof of Theorem \ref{proposition of atomic decomposition of MS Leb}. Then
$$
\|\chi_{\{\tau_k<\infty\}}\|_{p(\cdot)}\lesssim \|\chi_{\{\rho_k<\infty\}}\|_{p(\cdot)}.
$$
\end{lemma}

\begin{proof}
It is easy too see that the lemma is equivalent to the inequality
$$
\|\chi_{\{\tau_k<\infty\}}\|_{\frac{p(\cdot)}\varepsilon}\lesssim \|\chi_{\{\rho_k<\infty\}}\|_{\frac{p(\cdot)}\varepsilon}
$$
for some $0<\varepsilon< \underline p$.

Observe that
$$\|\chi_{\{\tau_k<\infty\}}\|_{\frac{p(\cdot)}\varepsilon} \leq \left\|\sum_{j=1}^\infty \chi_{F_{j}^k}\right\|_{\frac{p(\cdot)}\varepsilon}=\left\|\sum_{j=1}^\infty \sum_i\chi_{\overline {I_{k,j,i}}}\right\|_{\frac{p(\cdot)}\varepsilon}=:Y
$$
and
$$\left\|\sum_{j=1}^\infty \sum_i\chi_{ I_{k,j,i}}\right\|_{\frac{p(\cdot)}\varepsilon}= \|\chi_{\{\rho_k<\infty\}}\|_{\frac{p(\cdot)}\varepsilon}.$$
Choose a positive function $g\in L_{(\frac{p(\cdot)}{\varepsilon})'}$ with $\|g\|_{L_{(\frac{p(\cdot)}{\varepsilon})'}}\leq 1$ such that
\begin{align*}
Y&= \int_\Omega \sum_{j=1}^\infty \sum_i\chi_{\overline {I_{k,j,i}}} gd\mathbb P.
\end{align*}
Applying H\"older's inequality for some $\frac{p_+}{\varepsilon}<r<\infty$ and regular property (one just follow the proof for $Z$ in Theorem \ref{proposition of atomic decomposition of MS Leb}), we find that
\begin{align*}
Y \lesssim \|\sum_{j=1}^\infty \sum_i \chi_{I_{k,j,i}}\|_{p(\cdot)/\varepsilon} \|[M(g^{r'})]^{\frac 1{r'}}\|_{{(p(\cdot)/\varepsilon)'}},
\end{align*}
where the first ``$\lesssim$" is due to the regularity.
Since $\frac{p_+}{\varepsilon}<r<\infty$, we have
$r'<(p(\cdot)/\varepsilon)'$.
Using Theorem \ref{thm:maximal inequality}, we obtain
$$\|[M(g^{r'})]^{\frac 1{r'}}\|_{{(p(\cdot)/\varepsilon)'}} \lesssim \|g\|_{{(p(\cdot)/\varepsilon)'}}\leq 1,$$
which completes the proof.
\end{proof}

Combining Theorem \ref{ad 23 Leb} and Theorem \ref{proposition of atomic decomposition of MS Leb}, we have the following corollary.
\begin{corollary} \label{cor:equi Leb}
 Let $p(\cdot)\in\mathcal{P}(\Omega)$ satisfy  condition \eqref{log}. If $\{\mathcal {F}_n\}_{n\geq 0}$ is regular, then
 $$H_{p(\cdot)}^S= Q_{p(\cdot)},\quad H_{p(\cdot)}^M= P_{p(\cdot)}$$
 with equivalent quasi-norms.
\end{corollary}

\subsection{Atomic decompositions of $H_{p(\cdot),q}$}
In this subsection, we give the atomic decomposition for variable Lorentz-Hardy spaces $H_{p(\cdot),q}$. If $p(\cdot)=p$ is a constant, the corresponding results are studied in \cite{Jiao2009}, \cite{Jiao2016Wu}. 

\begin{definition} Let $p(\cdot)\in\mathcal{P}(\Omega)$, $0<q\leq \infty$.
The atomic Hardy space $H_{p(\cdot),q}^{\rm at, d,\infty}$ is defined as the space of all martingales $f=(f_n)_{n\geq 0}$ such that, for all $n\in\mathbb N,$
\begin{equation}\label{decomposition}
f_n= \sum_{k\in \mathbb Z} \mu_{k}a^{k}_n,\quad \mbox{a.e.}
\end{equation}
where $(a_{k})_{k\in \mathbb Z}$ is a sequence of  $(d,p(\cdot),\infty)$-atoms $(d=1,2,3)$ associated with stopping times $(\tau_k)_{k\in\mathbb Z}$ and $\mu_k=3\cdot 2^k \|\chi_{\{\tau_k<\infty\}}\|_{p(\cdot)}$ for each $k$.
For $f\in H_{p(\cdot),q}^{\rm at,d,\infty}$, define
$$\|f\|_{H_{p(\cdot),q}^{\rm at,d,\infty}} = \inf \left(\sum_{k\in \mathbb Z} 2^{kq}\|\chi_{\{\tau_k<\infty\}}\|_{p(\cdot)}^q\right)^{\frac 1q}\approx \inf\|(\mu_k)_{k\in\mathbb Z}\|_{\ell_q},$$
where the infimum is taken over all the decompositions of $f$ of the form \eqref{decomposition}.
\end{definition}

\begin{theorem} \label{theorem of atomic decomposition pq}
Let $p(\cdot)\in\mathcal{P}(\Omega)$, $0<q\leq \infty$.
 Then
  $$H_{p(\cdot),q}^s= H_{p(\cdot),q}^{\rm at, 1,\infty}$$
with equivalent quasi-norms.
\end{theorem}

\begin{proof}
Assume that $f \in H_{p(\cdot),q}^s$. Let us consider the following
stopping times for all $k \in \mathbb{Z}$,
$$\tau_k=\inf\{n\in\mathbb{N}: s_{n+1}(f)>2^k \}.$$
The sequence of these stopping times is obviously non-decreasing.
It is easy to see that for each $n\in \mathbb N$,
$$f_n=\sum_{k\in\mathbb{Z}}(f_n^{\tau_{k+1}}-f_n^{\tau_{k}}).$$
For every $k\in\mathbb Z$, $n\in\mathbb N$, let
$$\mu_k=3\cdot2^k \left\|\chi_{\{ \tau_k<\infty \}} \right\|_{p(\cdot)} \quad\mbox{and}\quad a_n^k= \frac{ f_n^{\tau_{k+1}}-f_n^{\tau_{k}} }{\mu_k}.$$
If $\mu_k=0$, then set $a_n^k=0$ for all
$k\in\mathbb{Z},n\in\mathbb{N}$. Then $(a_n^k)_{n \geq 0}$ is a martingale for each fixed $k\in
\mathbb{Z}$. Since $s(f^{\tau_k}) = s_{\tau_k}(f) \leq 2^k$, and by the sublinearity of the operator $s$, we get
$$ s\left((a_n^k)_{n \geq 0} \right) \leq \frac{s(f^{\tau_{k+1}})+s(f^{\tau_k})}{\mu_k} \leq
\left\| \chi_{\{ \tau_k<\infty \}} \right\|_{p(\cdot)}^{-1}.$$ Hence
 $(a_n^k)_{n \geq 0}$ is a bounded $L_2$-martingale. Consequently, there exists an element $a^k \in L_2$ such
that $\E_{n}a^k= a_n^k$. If $n \leq \tau_k$, then $a_n^k=0$.
Thus we
conclude that $a^k$ is  a $(1,p(\cdot),\infty)$-atom. For $q=\infty$, we have
\begin{align*}
\sup_k \mu_k&=  3 \sup_k 2^k \|\chi_{\{\tau_k<\infty\}}\|_{p(\cdot)}
\\&=3 \sup_k 2^k \|\chi_{\{s(f)>2^k\}}\|_{p(\cdot)}
\\&\leq C \|s(f)\|_{L_{p(\cdot),\infty}}=C\|f\|_{H_{p(\cdot),\infty}^s}.
\end{align*}
For $q<\infty$, we have
\begin{eqnarray*}
\Big(\sum_{k\in\mathbb{Z}}|\mu_k|^q\Big)^{\frac{1}{q}}
&=&3\Big(\sum_{k\in\mathbb{Z}}2^{kq}\|\chi_{\{\tau_k<\infty\}}\|_{p(\cdot)}^q\Big)
^{\frac{1}{q}}\\
&=&3\Big(\sum_{k\in\mathbb{Z}}2^{kq}\|\chi_{\{s(f)>2^k\}}\|_{p(\cdot)}^q\Big)^{\frac{1}{q}}
\\&\leq&C\Big(\sum_{k\in\mathbb{Z}}\int_{2^{k-1}}^{2^k}\lambda^{q-1}d\lambda \|\chi_{\{s(f)>2^k\}}\|_{p(\cdot)}^q\Big)^{\frac{1}{q}}
\\&\leq&C\Big(\int_0^\infty \lambda^{q-1}\|\chi_{\{s(f)>\lambda\}}\|_{p(\cdot)}^qd\lambda \Big)^{\frac{1}{q}}
\\&\leq&C\|s(f)\|_{L_{p(\cdot),q}}=C\|f\|_{H_{p(\cdot),q}^s}.
\end{eqnarray*}

Conversely, assume that a martingale $f$  has the decomposition \eqref{decomposition}. For an arbitrary integer $k_0$, set
 $$f=\sum_k \mu_k a^k:= F_1+F_2,$$
 where
 $$ F_1=\sum_{k=-\infty}^{k_0-1} \mu_k a^k\quad \mbox{and}\quad F_2=\sum_{k=k_0}^{\infty} \mu_k a^k.$$
By Remark \ref{remark for lorentz norm},
$$\|f\|_{H_{p(\cdot),q}^s}\lesssim \|s(F_1)\|_{L_{p(\cdot),q}}+ \|s(F_2)\|_{L_{p(\cdot),q}}.$$
Note that
$$s(F_1)\leq  \sum_{k=-\infty}^{k_0-1} \mu_ks(a^k),\quad s(F_2)\leq  \sum_{k=k_0}^{\infty} \mu_ks(a^k).$$

We deal with  $q=\infty$ firstly.
Since $a^k$ is a $(1,p(\cdot),\infty)$-atom for every $k\in \mathbb Z$, we find that
\begin{align*}
\|s(F_1)\|_\infty
&\leq  \sum_{k=-\infty}^{k_0-1} \mu_k \|s(a^k)\|_\infty
\leq \sum_{k=-\infty}^{k_0-1} \mu_k
\|\chi_{\{\tau_k<\infty\}}\|_{p(\cdot)}^{-1}\leq 3 \cdot 2^{k_0}.
\end{align*}
Thus we can deduce that
\begin{align} \label{weak_norm}
2^{k_0}\|\chi_{\{s(f)>6\cdot 2^{k_0}\}}\|_{p(\cdot)} &\leq 2^{k_0}\|\chi_{\{s(F_1)>3\cdot 2^{k_0}\}}
+\chi_{\{s(F_2)>3 \cdot 2^{k_0}\}}\|_{p(\cdot)} = 2^{k_0}\|\chi_{\{s(F_2)>3\cdot 2^{k_0}\}}\|_{p(\cdot)}.
\end{align}
Next we  estimate this expression.
Since $s(a^k)=0$ on the set $\{\tau_k=\infty\}$, we have
$\{s(a^k)>0\}\subset\{\tau_k<\infty\}$. Then,
\begin{equation} \label{estimate for F2}
 \{s(F_2)>3\cdot2^{k_0}\} \subset \{s(F_2)>0\}\subset\bigcup\limits_{k=k_0}^\infty\{s(a^k)>0\}\subset\bigcup\limits_{k=k_0}^\infty\{\tau_k<\infty\}.
\end{equation}
Hence, by Lemma \ref{lemma for equal 1} we obtain
\begin{align*}
\int_\Omega \left(\frac{3\cdot 2^{k_0}\chi_{\{s(F_2)>3\cdot 2^{k_0}\}}}{\sup_{k\in\mathbb Z}\mu_k}\right)^{p(x)}d\mathbb P & \leq\sum_{k=k_0}^\infty  \int_{\{\tau_k<\infty\}} \left(\frac{3\cdot 2^{k_0}}{\sup_{k\in\mathbb Z}\mu_k}\right)^{p(x)}d\mathbb P
\\&\leq \sum_{k=k_0}^\infty  \int_{\{\tau_k<\infty\}} \left(\frac{3\cdot 2^{k_0}}{3\cdot 2^k \|\chi_{\{\tau_k<\infty\}}\|_{p(\cdot)}}\right)^{p(x)}d\mathbb P
\\&  \leq \sum_{k=k_0}^\infty 2^{-(k-k_0)p_-}\int_{\Omega} \left(\frac{\chi_{\{\tau_k<\infty\}}}{ \|\chi_{\{\tau_k<\infty\}}\|_{p(\cdot)}}\right)^{p(x)}d\mathbb P \leq C_{p_{-}},
\end{align*}
which implies that
\begin{equation*}
3\cdot 2^{k_0}\|\chi_{\{s(F_2)>3\cdot 2^{k_0}\}}\|_{p(\cdot)}\leq C_{p_{-}} \sup_{k\in\mathbb Z}\mu_{k}.
\end{equation*}
Combining \eqref{weak_norm} and the above inequality, we have
$$
\|f\|_{H_{p(\cdot),\infty}^s}=\|s(f)\|_{L_{p(\cdot),\infty}} \leq  2 C_{p_{-}}  \inf\sup_{k\in\mathbb Z}\mu_k \lesssim \|f\|_{H_{p(\cdot),\infty}^{\rm at, 1,\infty}},
$$
where the infimum is taken over all the decompositions of $f$ of the form \eqref{decomposition}.

Now we consider the case  $q<\infty$.
 According to \eqref{weak_norm}, it suffices to estimate  $\|\chi_{\{s(F_2)>3\cdot 2^{k_0}\}}\|_{p(\cdot)}$.
Let
$$
0<\varepsilon <\min(\underline{p},q)\quad\text{ and}\quad 0<\delta<1.
$$
Applying \eqref{estimate for F2} and \eqref{equation of b trangile inequality}, we have
\begin{align*}
\|\chi_{\{s(F_2)>3\cdot 2^{k_0}\}}\|_{p(\cdot)} &\leq\left\| \sum_{k=k_0}^\infty\chi_{\{\tau_k<\infty\}}\right\|_{p(\cdot)}
\leq \left(\sum_{k=k_0}^\infty \|\chi_{\{\tau_k<\infty\}}\|_{p(\cdot)}^{{\varepsilon}}\right)^{1/{\varepsilon}} \n
\\&=  \left(\sum_{k=k_0}^\infty 2^{-k\delta{\varepsilon}}2^{k\delta{\varepsilon}}\|\chi_{\{\tau_k<\infty\}}\|_{p(\cdot)}
^{{\varepsilon}}\right)^{1/{\varepsilon}}.\n
\end{align*}
Using H\"{o}lder's inequality for $\frac{ q-{\varepsilon}}{q}+\frac{{\varepsilon}}{q}=1$, we get
\begin{align*}
\|\chi_{\{s(F_2)>3\cdot 2^{k_0}\}}\|_{p(\cdot)}&\leq \left(\sum_{k=k_0}^\infty 2^{-k\delta{\varepsilon}\frac{q}{q-{\varepsilon}}}\right)^{\frac{q-{\varepsilon}}{{\varepsilon}q}}
\left(\sum_{k=k_0}^\infty 2^{k\delta q}\|\chi_{\{\tau_k<\infty\}}\|_{p(\cdot)}
^q\right)^{1/q}
\\&\lesssim 2^{-k_0\delta} \left(\sum_{k=k_0}^\infty 2^{k\delta q}\|\chi_{\{\tau_k<\infty\}}\|_{p(\cdot)}
^q\right)^{1/q}.
 \end{align*}
 Consequently,
 \begin{align*}
 \sum_{k_0=-\infty}^\infty 2^{k_0q}\|\chi_{\{s(F_2)>3\cdot 2^{k_0}\}}\|_{p(\cdot)}^q &\lesssim \sum_{k_0=-\infty}^\infty2^{k_0(1-\delta)q} \sum_{k=k_0}^\infty 2^{k\delta q}\|\chi_{\{\tau_k<\infty\}}\|_{p(\cdot)}^q
 \\&= \sum_{k=-\infty}^\infty 2^{k\delta q}\|\chi_{\{\tau_k<\infty\}}\|_{p(\cdot)}^q \sum_{k_0=-\infty}^k 2^{k_0(1-\delta)q}
 \\&\lesssim \sum_{k=-\infty}^\infty 2^{kq}\|\chi_{\{\tau_k<\infty\}}\|_{p(\cdot)}^q,
 \end{align*}
 where the last $``\lesssim"$ is due to $1-\delta>0$. Then we obtain
$$
\|F_2\|_{H_{p(\cdot),q}^s}=\|s(F_2)\|_{L_{p(\cdot),q}}\lesssim \inf \Bigg(\sum_{k=-\infty}^\infty\mu_k^q\Bigg)^{1/q}\lesssim  \|f\|_{H_{p(\cdot),q}^{\rm at,1,\infty}},
$$
where the infimum is taken over all decompositions of the form \eqref{decomposition}.
\end{proof}

\begin{theorem}\label{theorem of atomic decomposition for rest}
Let $p(\cdot)\in\mathcal{P}(\Omega)$ and $0<q\leq \infty$.
Then
$$Q_{p(\cdot),q}= H_{p(\cdot),q}^{\rm at,2,\infty},\quad  P_{p(\cdot),q}=H_{p(\cdot),q}^{\rm at,3,\infty}$$
with equivalent quasi-norms.
\end{theorem}
\begin{proof}The proof is similar to those of Theorems \ref{theorem of atomic decomposition pq}, so we only sketch the outline.
Let $ f=(f_n)_{n\geq0}\in Q_{p(\cdot),q}$ (or $P_{p(\cdot),q}$). The stopping times $\tau_k$ are defined  by
$$\tau_k=\inf\{n\in\mathbb{N}: \lambda_n>2^k\}, \ \ \ (\inf\emptyset=\infty),$$
where $(\lambda_n)_{n\geq0}$ is the sequence in the definition of
$Q_{p(\cdot),q}$. Let $a_n^k$ and $\mu_k$ $(k\in\mathbb{Z})$
be the same  as in the proof of Theorem \ref{theorem of atomic decomposition pq}.
Then we get \eqref{decomposition},
where $(a^k)_{k\in\mathbb Z}$ is a sequence of $(2,p(\cdot),\infty)$-atoms (or $(3,p(\cdot),\infty)$-atoms).
 Moreover,
$$
\|(\mu_k)_{k\in\mathbb{Z}}\|_{\ell_q}\lesssim \|f\|_{Q_{p(\cdot),q}} \quad \mbox{(or} \quad \|(\mu_k)_{k\in\mathbb{Z}}\|_{\ell_q}
\lesssim\|f\|_{P_{p(\cdot),q}})
$$
still holds.

To prove the converse part, let
$$\lambda_n=\sum\limits_{k\in\mathbb{Z}}\mu_k\chi_{\{\tau_k\leq n\}}\|S(a^k)\|_\infty \quad\,(\text{or}  \quad\lambda_n=\sum\limits_{k\in\mathbb{Z}}\mu_k\chi_{\{\tau_k\leq n\}}\|M(a^k)\|_\infty).$$
Then $(\lambda_n)_{n\geq 0}$ is a nondecreasing, nonnegative and
adapted sequence with $S_{n+1}(f)\leq\lambda_n$ (or $|f_{n+1}|\leq \lambda_n$) for any $n\geq 0$.
For any given integer $k_0$, let
$$\lambda_\infty=\lambda_\infty^{(1)}+\lambda_\infty^{(2)},$$
where
$$\lambda_\infty^{(1)}=\sum\limits_{k=-\infty}^{k_0-1}\mu_k\chi_{\{\tau_k<\infty\}}\|S(a^k)\|_\infty \quad\,(\text{or}  \quad \lambda_\infty^{(1)}=\sum\limits_{k=-\infty}^{k_0-1}\mu_k\chi_{\{\tau_k<\infty\}}\|M(a^k)\|_\infty),$$
and
$$\lambda_\infty^{(2)}=
\sum\limits_{k=k_0}^\infty\mu_k
\chi_{\{\tau_k<\infty\}}\|S(a^k)\|_\infty \quad\,(\text{or} \quad \lambda_\infty^{(2)}=\sum\limits_{k=k_0}^\infty\mu_k\chi_{\{\tau_k<\infty\}}\|M(a^k)\|_\infty).$$

By replacing $s(F_1)$ and $s(F_2)$ in the proof of Theorem \ref{theorem of atomic decomposition pq} with
$\lambda_\infty^{(1)}$ and $\lambda_\infty^{(2)}$, respectively,
we obtain $f\in Q_{p(\cdot),q}$ (or $f\in P_{p(\cdot),q}$) and
$$
\|f\|_{Q_{p(\cdot),q}}\approx \inf\|(\mu_k)_{k\in\mathbb{Z}}\|_{\ell_q} \quad (\text{or} \quad \|f\|_{P_{p(\cdot),q}}\approx \inf\|(\mu_k)_{k\in\mathbb{Z}}\|_{\ell_q}),
$$
where the infimum is taken over all the decompositions \eqref{decomposition}.
\end{proof}

\begin{remark}
Theorems \ref{ad Leb}, \ref{ad 23 Leb}, \ref{theorem of atomic decomposition pq} and \ref{theorem of atomic decomposition for rest} do not need any restriction on $p(\cdot)$ and they do not need $\mathcal {F}_n$ to be generated by countably many
atoms. If $p(\cdot)=p$ is a constant, then the above atomic decompositions go back to \cite{Jiao2009}, \cite{Jiao2016Wu} and \cite{Weisz1994book}.
\end{remark}

\begin{remark} \label{remark of desenty}
 (1) In \eqref{decomposition}, if $q<\infty$, then
 the sum $\sum_{k=m}^n \mu_ka^k$ converges to $f$ in $H_{p(\cdot),q}^s$
as $m\rightarrow -\infty$, $n\rightarrow \infty$.
 Indeed,
 $$\sum_{k=m}^n \mu_ka^k = \sum_{k=m}^n (f^{\nu_{k+1}}-f^{\nu_k})=f^{\nu_{n+1}}-f^{\nu_m}.$$
 By the sublinearity of $s$, we have
 \begin{align*}
 \left\|f-\sum_{k=m}^n \mu_ka^k\right\|_{H_{p(\cdot),q}^s}&=\|s(f-f^{\nu_{n+1}}+f^{\nu_m})\|_{L_{p(\cdot),q}}\\&\leq
 \|s(f-f^{\nu_{n+1}})+s(f^{\nu_m})\|_{L_{p(\cdot),q}}
 \\&\lesssim  \|s(f-f^{\nu_{n+1}})\|_{L_{p(\cdot),q}}+
 \|s(f^{\nu_m})\|_{L_{p(\cdot),q}}.
 \end{align*}
Observe that
$$
s(f-f^{\nu_{n+1}})^2=s(f)^2- s(f^{\nu_{n+1}})^2, \quad
s(f-f^{\nu_{n+1}})\leq s(f), \qquad s(f^{\nu_m})\leq s(f)
$$
and
$$
s(f-f^{\nu_{n+1}}),\; s(f^{\nu_m})\rightarrow 0 \qquad \mbox{a.e. as $m\rightarrow -\infty$, $n\rightarrow \infty$}.
$$
Thus, by Lemma \ref{lemma of abs q} and  \ref{lemma of control}, we have
 $$\|s(f-f^{\nu_{n+1}})\|_{L_{p(\cdot),q}},\;\|s(f^{\nu_m})\|_{L_{p(\cdot),q}} \rightarrow 0\;\;\;\; as\; m\rightarrow -\infty,\;n\rightarrow \infty,$$
 which implies
$$
\left\|f-\sum_{k=m}^n \mu_k a^k\right\|_{H_{p(\cdot),q}^s}\rightarrow 0 \qquad \mbox{as $m\rightarrow -\infty$, $n\rightarrow
 \infty$}.
$$
Further, for $k\in \mathbb Z$, $a^k=(a_n^k)_{n\geq 0}$ (here $a^k$ is a $(1,p(\cdot),\infty)$-atom) is $L_2$
 bounded, hence $H_2^s=L_2$ is dense in $H_{p(\cdot),q}^s.$
 Similarly, $L_\infty$ is dense in $P_{p(\cdot),q}$.

 (2) If $q=\infty$ and $s(f)\in \mathscr L_{p(\cdot),\infty}(\Omega)$,  then by Lemma \ref{lemma of control}, the sum $\sum_{k=m}^n \mu_ka^k$ converges to $f$ in $H_{p(\cdot),\infty}^s$
as $m\rightarrow -\infty$, $n\rightarrow \infty$.
 \end{remark}

We can show the next atomic decomposition corresponding to Theorem \ref{proposition of atomic decomposition of MS Leb}. The proof is omitted.

\begin{theorem}\label{proposition of atomic decomposition of MS}
 Let $p(\cdot)\in\mathcal{P}(\Omega)$ satisfy  condition \eqref{log} and  $0<q\leq \infty$. If $\{\mathcal {F}_n\}_{n\geq 0}$ is regular, then
 $$H_{p(\cdot),q}^S= H_{p(\cdot),q}^{\rm at,2,\infty},\quad H_{p(\cdot),q}^M= H_{p(\cdot),q}^{\rm at,3,\infty}$$
 with equivalent quasi-norms.
\end{theorem}

Similarly to Corollary \ref{cor:equi Leb}, we have a corresponding result for variable Lorentz-Hardy spaces.
\begin{corollary}\label{cor:equi Lor}
 Let $p(\cdot)\in\mathcal{P}(\Omega)$ satisfy  condition \eqref{log} and  $0<q\leq \infty$. If $\{\mathcal {F}_n\}_{n\geq 0}$ is regular, then
 $$H_{p(\cdot),q}^S= Q_{p(\cdot),q},\quad H_{p(\cdot),q}^M= P_{p(\cdot),q}$$
 with equivalent quasi-norms.
\end{corollary}

\section{Boundedness of the martingale operators} \label{sec4}
This section is devoted to the applications of the atomic characterizations established in Section \ref{sec3} while we are dealing with martingale inequalities between different Hardy spaces. Furthermore, if we suppose that $\{\mathcal F_n\}_{n\geq0}$ is regular, then the equivalence of different Hardy spaces will be proved.

\subsection{Martingale inequalities between $H_{p(\cdot)}$}

As an application of the atomic decompositions, we shall  obtain a sufficient condition for a $\sigma$-sublinear operator to be bounded from the martingale Hardy spaces to $L_{p(\cdot)}$.

An operator $T:X\rightarrow Y$ is called a $\sigma$-sublinear operator if for any $\alpha\in \mathbb C$ it satisfies
$$
\left|T\left(\sum_{k=1}^{\infty}f_k\right)\right|\leq \sum_{k=1}^{\infty} |T(f_k)|\quad \text{and}\quad |T(\alpha f)|= |\alpha||T(f)|,$$
where $X$ is a martingale space and $Y$ is a measurable function space.

Suppose that $\tau$ is a stopping time. Denote
$$\mathcal{F}_{\tau}=\{F\in \mathcal{F}:F\cap \{\tau\leq n\}\in \mathcal{F}_n, \quad n\geq1\}.$$
$\mathcal{F}_{\tau}$ is a sub-$\sigma$-algebra of $\mathcal{F}$. Then the conditional expectation with respect to $\mathcal{F}_{\tau}$ is denoted by $\mathbb{E}_{\tau}$.

We need the following property taken from \cite{XWYJ2018}.

\begin{lemma}[{\cite[Lemma 5.1]{XWYJ2018}}] \label{lem-op}
Let $a$ be a $(1,p(\cdot),\infty)$-atom associated with stopping time $\tau$.
If $T\in\{s,S,M\}$, then
$$T(a\chi_F)=T(a)\chi_F,\quad \forall F\in \mathcal{F}_{\tau}.$$
\end{lemma}

\begin{lemma}\label{lem-atom}
Let $p(\cdot)\in\mathcal{P}(\Omega)$ and let $a$ be a $(1,p(\cdot),\infty)$-atom associated with stopping time $\tau$. If for $p_+<r<\infty$ and $T:H_r^s \rightarrow L_r$ is a bounded $\sigma$-sublinear operator and
\begin{equation}\label{condition for boundedness}
T(a)\chi_F=T(a\chi_F),\quad \forall F\in \mathcal{F}_{\tau},
\end{equation}
then $T(a)=T(a)\chi_{\{\tau<\infty\}}$ and
$$\mathbb{E}_{\tau}(|T(a)|^r) \lesssim \|\chi_{\{\tau<\infty\}}\|_{p(\cdot)}^{-r}. $$
\end{lemma}
\begin{proof}
Take $F\in \mathcal{F}_{\tau}$.
Then, by \eqref{condition for boundedness} and the boundedness of $T$,  we have
$$\int_F|T(a)|^rd\mathbb{P}= \int_{\Omega}|T(a\chi_F)|^rd\mathbb{P} \lesssim \int_{\Omega}s(a\chi_F)^rd\mathbb{P}.$$
Since $F$ is arbitrary and $s(a)$ is supported by $\{\tau<\infty\}$, we obtain the first assertion. By Lemma \ref{lem-op} and the fact that $a$ is a $(1,p(\cdot),\infty)$-atom, we obtain
$$\int_F|T(a)|^rd\mathbb{P}\lesssim \int_{\Omega}s(a)^r\chi_F d\mathbb{P}\leq \int_{\Omega}\|\chi_{\{\tau<\infty\}}\|_{p(\cdot)}^{-r}\chi_Fd\mathbb{P}.$$
Since $F$ is arbitrary, the second assertion follows.
\end{proof}

\begin{theorem} \label{theorem of boundedness 1 Leb}
 Let $p(\cdot)\in\mathcal{P}(\Omega)$ satisfy condition \eqref{log} and $1<r< \infty$ with $p_+<r$.
If $T:H_r^s \rightarrow L_r$ is a bounded $\sigma$-sublinear operator and
\begin{equation}\label{condition for boundedness 1 Leb}
T(a)\chi_F=T(a\chi_F),\quad \forall F\in \mathcal{F}_{\tau}
\end{equation}
for all $(1,p(\cdot),\infty)$-atoms $a$, where $\tau$ is the stopping time associated with $a$, then
$$\|Tf\|_{L_{p(\cdot)}}\lesssim \|f\|_{H_{p(\cdot)}^s},\quad f\in H_{p(\cdot)}^s.$$
\end{theorem}

\begin{proof} Let a martingale $f\in H_{p(\cdot)}^s$. By  Theorem \ref{ad Leb},  we know that
 $f$  has a decomposition as \eqref{decomposition Leb1} such that $a^{k}$  is a $(1,p(\cdot),\infty)$-atom and $\mu_{k}=3\cdot 2^k \|\chi_{\{\tau<\infty\}}\|_{p(\cdot)}$. According to the boundedness of $T$,
$$\|T(a^{k})\|_r \lesssim \|s(a^{k})\|_r \leq \frac {\|\chi_{\{\tau<\infty\}}\|_r}{\|\chi_{\{\tau<\infty\}}\|_{p(\cdot)}}.$$
By  the $\sigma$-sublinearity  of the
operator $T$, we have
$$
\left| T(f) \right|\leq  \sum_{k\in \mathbb Z} \mu_{k}|T(a^{k})|.
$$
Then, for $0<t<\underline p\leq 1$, we have
$$\|T(f)\|_{p(\cdot)}\leq  \left\|\left[\sum_{k\in \mathbb Z} \left(\mu_{k}T(a^{k})\right)^{t}\right]^{\frac {1}{t}}\right\|_{p(\cdot)}=:Z.$$

By  Lemma \ref{lem:duality for variabl p}, we may choose a positive function $g\in L_{(\frac{p(\cdot)}{t})'}$ with $\|g\|_{{(\frac{p(\cdot)}{t})'}}\leq 1$ such that
\begin{align*}
Z^{t}&= \int_\Omega \sum_{k\in \mathbb Z}  \left[ 3\cdot 2^k \|\chi_{\{\tau<\infty\}}\|_{p(\cdot)}|T(a^{k})|\right]^{t} gd\mathbb P\\
&=\int_\Omega \sum_{k\in \mathbb Z}  3\cdot 2^{kt} \|\chi_{\{\tau<\infty\}}\|_{p(\cdot)}^t \chi_{\{\tau<\infty\}} \mathbb{E}_{\tau_k}(|T(a^{k})|^tg)d\mathbb P.
\end{align*}
The H\"{o}lder inequality for conditional expectation implies that
$$\mathbb{E}_{\tau_k}(|T(a^{k})|^tg)\leq \mathbb{E}_{\tau_k}(|T(a^{k})|^r)^{t/r}\mathbb{E}_{\tau_k}(g^{(r/t)'})^{1/{(r/t)'}}.$$
Applying Lemma \ref{lem-atom}, we can see that
\begin{align*}
Z^t&\lesssim \int_\Omega \sum_{k\in \mathbb Z}  3\cdot 2^{kt}  \chi_{\{\tau_k<\infty\}} \mathbb{E}_{\tau_k}(g^{(r/t)'})^{1/{(r/t)'}} d\mathbb P\\
\\&\leq \int_\Omega \sum_{k\in \mathbb Z}  3\cdot 2^{kt}  \chi_{\{\tau_k<\infty\}} [M(g^{(\frac{r}{t})'})]^{1/(\frac{r}{t})'} d\mathbb P
\\&\leq  \left \|\sum_{k\in \mathbb Z}  (3\cdot 2^k)^{t} \chi_{\{\tau_k<\infty\}} \right \|_{{p(\cdot)/t}}  \|[M(g^{(\frac{r}{t})'})]^{1/(\frac{r}{t})'}\|_{{(p(\cdot)/t)'}}.
\end{align*}
 Since  $p_+<r$, we deduce that
 $$\frac r {t}> \frac {p_+} {t}  \quad \mbox{and}
\quad \left(\frac{r}{t}\right)' <\left(\frac{p(\cdot)}{t}\right)'.$$
Note that $t<p_-$. Hence, $((p(\cdot)/t)')_+<\infty$. Using the maximal inequality (Theorem \ref{thm:maximal inequality}), we have
$$\|[M(g^{(\frac{r}{t})'})]^{1/(\frac{r}{t})'}\|_{{(p(\cdot)/t)'}} \lesssim \|g\|_{{(p(\cdot)/t)'}}\leq 1.$$
Thus, by Theorem \ref{ad Leb}, we obtain
$$Z\lesssim  \left \|\sum_{k\in \mathbb Z}  (3\cdot 2^k)^{t} \chi_{\{\tau_k<\infty\}} \right \|_{{p(\cdot)/t}}^{\frac 1t} \lesssim \|f\|_{H_{p(\cdot)}^s}, $$
which completes the proof.
\end{proof}

Similarly to Theorem \ref{theorem of boundedness 1 Leb}, we obtain the following theorem by applying Theorem \ref{ad 23 Leb}.

\begin{theorem} \label{theorem of boundedness 2 and 3 Leb}
 Let $p(\cdot)\in\mathcal{P}(\Omega)$ satisfy condition \eqref{log} and $1<r<\infty$ with $p_+<r$.
If $T:H_r^S\rightarrow L_r$ (or $H_r^M\rightarrow L_r$) is a bounded $\sigma$-sublinear operator and
\eqref{condition for boundedness 1 Leb} holds
for all $(2,p(\cdot),\infty)$-atoms (or $(3,p(\cdot),\infty)$-atoms), then
\begin{align*}
&\|Tf\|_{L_{p(\cdot)}}\lesssim \|f\|_{Q_{p(\cdot)}},\quad f\in Q_{p(\cdot)},
\\ (or \quad  &\|Tf\|_{L_{p(\cdot)}}\lesssim \|f\|_{P_{p(\cdot)}},\quad f\in P_{p(\cdot)}).
\end{align*}
\end{theorem}

Now we prove our main result of this section.

\begin{theorem}\label{theorem of martingale inequalities Leb}
 Let $p(\cdot)\in\mathcal{P}(\Omega)$ satisfy  condition \eqref{log}.
Then the following inequalities hold:
 \begin{equation}\label{martingale inequality 1 Leb}
 \|f\|_{H_{p(\cdot)}^M}\lesssim\|f\|_{H_{p(\cdot)}^s},\quad  \|f\|_{H_{p(\cdot)}^S}\lesssim \|f\|_{H_{p(\cdot)}^s},\quad \mbox{if}\quad0<p_-\leq p_+<2;
 \end{equation}
 \begin{equation}\label{martingale inequality 2 Leb}
\|f\|_{H_{p(\cdot)}^M}\leq \|f\|_{P_{p(\cdot)}},\quad  \|f\|_{H_{p(\cdot)}^S}\leq \|f\|_{Q_{p(\cdot)}};
\end{equation}
 \begin{equation}\label{martingale inequality 3 Leb}
\|f\|_{H_{p(\cdot)}^S}\lesssim\|f\|_{P_{p(\cdot)}},\quad  \|f\|_{H_{p(\cdot)}^M}\lesssim\|f\|_{Q_{p(\cdot)}};
\end{equation}
 \begin{equation}\label{martingale inequality 4 Leb}
\|f\|_{H_{p(\cdot)}^s}\lesssim\|f\|_{P_{p(\cdot)}},\quad  \|f\|_{H_{p(\cdot)}^s}\lesssim\|f\|_{Q_{p(\cdot)}};
\end{equation}
\begin{equation}\label{martingale inequality 5 Leb}
\|f\|_{P_{p(\cdot)}} \lesssim \|f\|_{Q_{p(\cdot)}} \lesssim\|f\|_{P_{p(\cdot)}}.
\end{equation}
Moreover, if $\{\mathcal {F}_n\}_{n\geq 0}$ is regular, then
$$H_{p(\cdot)}^S=Q_{p(\cdot)}=P_{p(\cdot)}=H_{p(\cdot)}^M =
H_{p(\cdot)}^s$$
with equivalent quasi-norms.
\end{theorem}

\begin{proof}
According to Lemma \ref{lem-op}, we know that the operators $M,S$ and $s$ all satisfy \eqref{condition for boundedness 1 Leb}.
First we show \eqref{martingale inequality 1 Leb}. Let $f\in
H_{p(\cdot)}^s$. The maximal operator $T(f)=M(f)$ is $\sigma$-sublinear
and $\|M(f)\|_2\lesssim \|s(f)\|_2$ (see \cite[Theorem 2.11(i)]{Weisz1994book}).
Thus it follows from Theorem \ref{theorem of boundedness 1 Leb} that
$$\|f\|_{H_{p(\cdot)}^M}= \|M(f)\|_{{p(\cdot)}}\lesssim\|f\|_{H_{p(\cdot)}^s}.$$
Similarly, considering the operator $T(f)=S(f)$,  we get the second
inequality of \eqref{martingale inequality 1 Leb} by Theorem \ref{theorem of boundedness 1 Leb}.

\eqref{martingale inequality 2 Leb} comes easily from the definition of these martingale spaces.

Next we show \eqref{martingale inequality 3 Leb}. Consider the  operator $T(f)=M(f)$ or $S(f)$. Then  \eqref{martingale inequality 3 Leb} follows from
the combination of the Burkholder-Gundy and Doob maximal inequalities
$$
\|S(f)\|_r\approx \|M(f)\|_r \approx \|f\|_r \qquad (1<r<\infty)
$$
(see \cite[Theorem 2.12]{Weisz1994book}) and Theorem \ref{theorem of boundedness 2 and 3 Leb}.

 \eqref{martingale inequality 4 Leb} can be deduced by applying the inequalities (see \cite[Theorem 2.11(ii)]{Weisz1994book})
 $$\|s(f)\|_r\lesssim \|M(f)\|_{r},\quad \|s(f)\|_r\lesssim \|M(f)\|_{r} \approx \|S(f)\|_r, \quad 2<r<\infty ,$$
 and  Theorem \ref{theorem of boundedness 2 and 3 Leb}.

To prove \eqref{martingale inequality 5 Leb}, we use \eqref{martingale inequality 3 Leb}.
Assume
that $f=(f_n)_{n\geq 0}\in Q_{p(\cdot)}$, then there exists an
optimal control $(\lambda_n^{(1)})_{n\geq 0}$ such that $S_n(f)\leq
\lambda_{n-1}^{(1)}$ with $\lambda_\infty^{(1)}\in L_{p(\cdot)}$. Since
$$|f_n|\leq M_{n-1}(f) +\lambda_{n-1}^{(1)},$$
by the second inequality of \eqref{martingale inequality 3 Leb} we have
$$\|f\|_{P_{p(\cdot)}}\leq C \big(\|f\|_{H_{p(\cdot)}^M}+\|\lambda_\infty^{(1)}\|_{{p(\cdot)}} \big)\lesssim\|f\|_{Q_{p(\cdot)}}.$$
On the other hand, if $f=(f_n)_{n\geq 0}\in P_{p(\cdot)}$,
then there exists an optimal control $(\lambda_n^{(2)})_{n\geq 0}$
such that $|f_n|\leq \lambda_{n-1}^{(2)}$ with
$\lambda_\infty^{(2)}\in L_{p(\cdot)}$. Notice that
$$S_n(f)\leq
S_{n-1}(f)+2\lambda_{n-1}^{(2)}.$$
Using the first inequality of \eqref{martingale inequality 3 Leb},  we  get the rest of \eqref{martingale inequality 5 Leb}.

Further, assume that $\{\mathcal {F}_n\}_{n\geq 0}$ is regular.
Then according to \cite[p. 33]{Weisz1994book}, we have
$$S_n(f)\leq R^{1/2}s_n(f)\quad\text{and}\quad \|f\|_{H_{p(\cdot)}^S}\lesssim \|f\|_{H_{p(\cdot)}^s}.$$
Since $s_n(f)\in \mathcal {F}_{n-1}$, by the definition of
$Q_{p(\cdot)}$ we have
$$\|f\|_{Q_{p(\cdot)}}\lesssim \|s(f)\|_{{p(\cdot)}}=\|f\|_{H_{p(\cdot)}^s}.$$
Hence, by \eqref{martingale inequality 4 Leb} we obtain
$$Q_{p(\cdot)}=H_{p(\cdot)}^s.$$
Combining  this and Corollary \ref{cor:equi Leb}, we get
$$H_{p(\cdot)}^S=Q_{p(\cdot)}=H_{p(\cdot)}^s=P_{p(\cdot)}=H_{p(\cdot)}^M.$$
\end{proof}

The next result follows from Theorem \ref{ad Leb} and  Theorem \ref{theorem of martingale inequalities Leb}.
\begin{corollary}\label{cor:ad for M}
Let $p(\cdot)\in\mathcal{P}(\Omega)$ satisfy  condition \eqref{log}.
If $\{\mathcal {F}_n\}_{n\geq 0}$ is regular, then
$$H_{p(\cdot)}=H_{p(\cdot)}^{\rm at, d,\infty},\quad d=1,2,3$$
with equivalent quasi-norms. Here, $H_{p(\cdot)}$ denotes any one of the five Hardy spaces in Theorem \ref{theorem of martingale inequalities Leb}.
\end{corollary}

Next, we consider a special case of martingale transforms.

\begin{definition}\label{d1}
Let the martingale transform $T_b$ be defined by
$$
(T_b f)_n=\sum\limits_{k=1}^nb_{k-1} d_kf, \quad n\in\mathbb N,
$$
where $b_k$ is $\mathcal F_k$ measurable and $|b_k|\leq 1$.
\end{definition}

\begin{theorem}\label{t18 Leb}
Let $p(\cdot)$ satisfy \eqref{log} with $1< p_-\leq p_+ < \infty$. If $\{\mathcal {F}_n\}_{n\geq 0}$ is regular, then
$$
\left\|T_b f\right\|_{{p(\cdot)}} \lesssim \left\|f\right\|_{{p(\cdot)}}.
$$
\end{theorem}

\begin{proof}
Since $|b_k|\leq1$ for each $k$, we have $S(T_b f)\leq S(f)$.
By Theorem \ref{theorem of martingale inequalities Leb} and Burkholder-Gundy inequality, we have
$$
\left\|T_b f \right\|_{{p(\cdot)}} \leq \left\|M(T_b f)\right\|_{{p(\cdot)}} \lesssim \left\|S(T_b f)\right\|_{{p(\cdot)}} \lesssim \left\|S(f)\right\|_{{p(\cdot)}} \lesssim \left\|M(f)\right\|_{{p(\cdot)}}.
$$
Now the result follows from Theorem \ref{thm:maximal inequality}.
\end{proof}

\subsection{Martingale inequalities between $H_{p(\cdot),q}$}
In this subsection, we extend Theorem \ref{theorem of martingale inequalities Leb} to the variable Lorentz-Hardy setting; see Theorem \ref{theorem of martingale inequalities}. First, we prove a result that is corresponding to Theorem \ref{theorem of boundedness 1 Leb}.

\begin{theorem} \label{theorem of boundedness 1}
 Let $p(\cdot)\in\mathcal{P}(\Omega)$ satisfy condition \eqref{log}, $0<q\leq \infty$ and $1<r\leq \infty$ with $p_+<r$.
If $T:H_r^s \rightarrow L_r$ is a bounded $\sigma$-sublinear operator and
satisfies \eqref{condition for boundedness 1 Leb},
 then
$$\|Tf\|_{L_{p(\cdot),q}}\lesssim \|f\|_{H_{p(\cdot),q}^s},\quad f\in H_{p(\cdot),q}^s.$$
\end{theorem}

\begin{proof} Let a martingale $f\in H_{p(\cdot),q}^s$. By  Theorem \ref{theorem of atomic decomposition pq},  we know that
 $f$  has a decomposition as \eqref{decomposition} such that $a^{k}$  is a $(1,p(\cdot),\infty)$-atom and $\mu_{k}=3\cdot 2^k \|\chi_{\{\tau_k<\infty\}}\|_{p(\cdot)}$.  For an arbitrary integer $k_0$, we set again
$$f= \sum_{k\in \mathbb Z} \mu_{k}a^{k}=F_1+F_2,$$
where
$$F_1=\sum_{k=-\infty}^{k_0-1} \mu_{k}a^{k},\quad F_2=f-F_1.$$
By  the $\sigma$-sublinearity  of the
operator $T$, we have
$$
| T(F_1) |\leq  \sum_{k=-\infty}^{k_0-1} \mu_{k}T(a^{k}),\quad | T(F_2) |\leq  \sum_{k=k_0}^{\infty} \mu_{k}T(a^{k}).
$$
We need to estimate  $\|T(F_1)\|_{L_{p(\cdot),q}}$ and $\|T(F_2)\|_{L_{p(\cdot),q}}$, separately.

To estimate  $\|T(F_1)\|_{L_{p(\cdot),q}}$,  we let $0<\varepsilon<\underline{p}$. Fix
 $L\in (1,\frac{1}{\varepsilon})$ such that $L<r/{p_+}$ and choose $\ell$ such that $0<\ell <1-1/L.$ By H\"{o}lder's inequality for $\frac 1L+\frac 1{L'}=1$, we have
 \begin{align*}
 T(F_1)\leq \Big(\sum_{k=-\infty}^{k_0-1}2^{k\ell L'}\Big)^{1/{L'}}\Big(\sum_{k=-\infty}^{k_0-1}[2^{-k\ell }\mu_kT(a^k)]^L\Big)^{1/L}\lesssim 2^{k_0\ell}\Big(\sum_{k=-\infty}^{k_0-1}[2^{-k\ell }\mu_kT(a^k)]^L\Big)^{1/L}.
 \end{align*}
 Then, we have
\begin{align*}
\|\chi_{\{T(F_1)>2^{k_0}\}}\|_{p(\cdot)}&\lesssim \Big\|\frac{T(F_1)^L}{2^{k_0L}}\Big\|_{p(\cdot)}
\lesssim 2^{k_0L(\ell-1)} \left\|\sum_{k=-\infty}^{k_0-1}[2^{-k\ell }\mu_kT(a^k)]^L \right\|_{p(\cdot)}\\
&\leq 2^{k_0L(\ell-1)} \left\{\sum_{k=-\infty}^{k_0-1} 2^{(1-\ell) kL\varepsilon}\|\chi_{\{\tau_k<\infty\}} \|_{p(\cdot)}^{L\varepsilon} \||T(a^{k}) |^{L\varepsilon}\|_{{p(\cdot)}/\varepsilon}\right\}^{\frac 1\varepsilon}\\
&\lesssim 2^{k_0L(\ell-1)} \left\{\sum_{k=-\infty}^{k_0-1} 2^{(1-\ell) kL\varepsilon}\|\chi_{\{\tau_k<\infty\}} \|_{p(\cdot)}^{L\varepsilon} \|s(a^{k})\|_{p(\cdot)L}^{L\varepsilon}\right\}^{\frac 1\varepsilon},
 \end{align*}
 where the last ``$\lesssim $" is due to Theorem \ref{theorem of boundedness 1 Leb} (note that $p_+L<r$). Since $\|s(a^k)\|_{\infty}\leq \|\chi_{\{\tau_k<\infty\}}\|_{p(\cdot)}^{-1}$ for every $k\in\mathbb{Z}$, it follows
 $$
 \|\chi_{\{T(F_1)>2^{k_0}\}}\|_{p(\cdot)} \leq 2^{k_0L(\ell-1)} \left\{\sum_{k=-\infty}^{k_0-1} 2^{(1-\ell) kL\varepsilon}\|\chi_{\{\tau_k<\infty\}} \|_{p(\cdot)}^{\varepsilon} \right\}^{\frac 1\varepsilon}.$$
 For the case $q=\infty$, we have
 \begin{align*}
\|\chi_{\{T(F_1)>2^{k_0}\}}\|_{p(\cdot)} &\lesssim  2^{k_0L(\ell-1)} \left(\sum_{k=-\infty}^{k_0-1} 2^{k((1-\ell) L-1)) \varepsilon} \right)^{1/\varepsilon} \sup_{k\in\mathbb Z} 2^k\left\|\chi_{\{\tau_k<\infty\}}\right\|_{{p(\cdot)}} \\
&\lesssim 2^{-k_0}  \sup_{k\in\mathbb Z} 2^k\left\|\chi_{\{\tau_k<\infty\}}\right\|_{{p(\cdot)}},
 \end{align*}
(note that $(1-\ell)L-1>0$) which implies that
 $$\|T(F_1)\|_{L_{p(\cdot),\infty}}\lesssim \|f\|_{H_{p(\cdot),\infty}^s}.$$
Now we turn to deal with the case $q<\infty$. Set
$$\delta=\frac{(1-\ell)L+1}{2}>1.$$
By the H\"{o}lder inequality for $\frac{\varepsilon}{q}+\frac{q-\varepsilon}{q}=1$, we have
\begin{align*}
\|\chi_{\{T(F_1)>2^{k_0}\}}\|_{p(\cdot)} &\leq 2^{k_0L(\ell-1)} \left( \sum_{k=-\infty}^{k_0-1} 2^{k((1-\ell)L-\delta)\varepsilon \frac{q}{q-\varepsilon}} \right)^{(q-\varepsilon)/q}\left(\sum_{k=-\infty}^{k_0-1} 2^{k\delta q}\|\chi_{\{\tau_k<\infty\}} \|_{p(\cdot)}^{q} \right)^{\frac 1q}\\
&\lesssim 2^{-k_0\delta} \left(\sum_{k=-\infty}^{k_0-1} 2^{k\delta q}\|\chi_{\{\tau_k<\infty\}} \|_{p(\cdot)}^{q} \right)^{\frac 1q}.
\end{align*}
From this, by basic calculation, we get
 \begin{align*}
 \sum_{k_0=-\infty}^\infty 2^{k_0q}\|\chi_{\{T(F_1)>2^{k_0}\}}\|_{p(\cdot)}^{q}&\lesssim   \sum_{k=-\infty}^\infty 2^{kq}\left\| \chi_{\{\tau_k<\infty\}}\right\|_{{p(\cdot)}}^{q}.
 \end{align*}
We deduce that
$$
\|T(F_1)\|_{L_{p(\cdot),q}}\lesssim \|f\|_{H_{p(\cdot),q}^s}.
$$

Now we start to estimate $\|T(F_2)\|_{L_{p(\cdot),q}}$. According to  condition \eqref{condition for boundedness 1 Leb} and Lemma \ref{lem-atom},
\begin{equation*}
 \{T(F_2)>2^{k_0}\} \subset \{T(F_2)>0\}\subset\bigcup\limits_{k=k_0}^\infty\{T(a^{k})>0\}\subset\bigcup\limits_{k\geq k_0} \{\tau_k<\infty\}.
\end{equation*}
 Then
$$\|\chi_{\{T(F_2)>2^{k_0}\}}\|_{p(\cdot)} \leq\left\| \sum_{k=k_0}^\infty\chi_{\{\tau_k<\infty\}}\right\|_{p(\cdot)}.$$
So, repeating  the  calculation of the proof of Theorem \ref{theorem of atomic decomposition pq}, we  easily obtain
$$\|T(F_2)\|_{L_{p(\cdot),q}}\lesssim \|f\|_{H_{p(\cdot),q}^s}.$$
The proof is complete now.
\end{proof}

The following result can be similarly prove. We omit the proof.

\begin{theorem} \label{theorem of boundedness 2 and 3}
 Let $p(\cdot)\in\mathcal{P}(\Omega)$ satisfy condition \eqref{log}, $0<q\leq \infty$ and $1<r\leq \infty$ with $p_+<r$.
If $T:H_r^S\rightarrow L_r$ (or $H_r^*\rightarrow L_r$) is a bounded $\sigma$-sublinear operator and
\eqref{condition for boundedness 1 Leb} holds
for all $(2,p(\cdot),\infty)$-atoms (or $(3,p(\cdot),\infty)$-atoms), then we have
\begin{align*}
&\|Tf\|_{L_{p(\cdot),q}}\lesssim \|f\|_{Q_{p(\cdot),q}},\quad f\in Q_{p(\cdot),q},
\\ (or \quad  &\|Tf\|_{L_{p(\cdot),q}}\lesssim \|f\|_{P_{p(\cdot),q}},\quad f\in P_{p(\cdot),q}).
\end{align*}
\end{theorem}

Similarly to Theorem \ref{theorem of martingale inequalities Leb}, applying the two theorems above, we can prove the result below. We omit the details
of the proof.

\begin{theorem}\label{theorem of martingale inequalities}
 Let $p(\cdot)\in\mathcal{P}(\Omega)$ satisfy  condition \eqref{log} and $0<q\leq \infty$.
Then the following inequalities hold:
 \begin{equation}\label{martingale inequality 1}
 \|f\|_{H_{p(\cdot),q}^M}\lesssim\|f\|_{H_{p(\cdot),q}^s},\quad  \|f\|_{H_{p(\cdot),q}^S}\lesssim\|f\|_{H_{p(\cdot),q}^s},\quad \mbox{if}\quad0<p_-\leq p_+<2;
 \end{equation}
 \begin{equation}\label{martingale inequality 2}
\|f\|_{H_{p(\cdot),q}^M}\leq \|f\|_{P_{p(\cdot),q}},\quad  \|f\|_{H_{p(\cdot),q}^S}\leq \|f\|_{Q_{p(\cdot),q}};
\end{equation}
 \begin{equation}\label{martingale inequality 3}
\|f\|_{H_{p(\cdot),q}^S}\lesssim\|f\|_{P_{p(\cdot),q}},\quad  \|f\|_{H_{p(\cdot),q}^M}\lesssim\|f\|_{Q_{p(\cdot),q}};
\end{equation}
 \begin{equation}\label{martingale inequality 4}
\|f\|_{H_{p(\cdot),q}^s}\lesssim\|f\|_{P_{p(\cdot),q}},\quad  \|f\|_{H_{p(\cdot),q}^s}\lesssim\|f\|_{Q_{p(\cdot),q}};
\end{equation}
\begin{equation}\label{martingale inequality 5}
\|f\|_{P_{p(\cdot),q}} \lesssim \|f\|_{Q_{p(\cdot),q}} \lesssim\|f\|_{P_{p(\cdot),q}}.
\end{equation}
Moreover, if $\{\mathcal {F}_n\}_{n\geq 0}$ is regular, then
$$H_{p(\cdot),q}^S=Q_{p(\cdot),q}=P_{p(\cdot),q}=H_{p(\cdot),q}^M =
H_{p(\cdot),q}^s$$
with equivalent quasi-norms.
\end{theorem}

 Combining Theorem \ref{theorem of atomic decomposition pq}  and Theorem \ref{theorem of martingale inequalities}, we obtain the following result.

\begin{corollary}\label{cor:equivalence}
Let $p(\cdot)\in\mathcal{P}(\Omega)$ satisfy  condition \eqref{log},  $0<q\leq \infty$ and $1<r\leq \infty$ with $p_+<r$. If $\{\mathcal {F}_n\}_{n\geq 0}$ is regular, then
 $$H_{p(\cdot),q}= H_{p(\cdot),q}^{{\rm at},d,\infty},\quad d=1,2,3$$
 with equivalent quasi-norms. Here, $H_{p(\cdot),q}$ denotes any one of the five Hardy spaces in Theorem \ref{theorem of martingale inequalities}.
\end{corollary}

The following result is corresponding to Theorem \ref{t18 Leb}. The proof is omitted.

\begin{theorem}\label{t18}
Let $p(\cdot)$ satisfy \eqref{log} with $1< p_-\leq p_+ < \infty$ and $0<q\leq \infty$. If $\{\mathcal {F}_n\}_{n\geq 0}$ is regular, then
$$
\left\|T_b f\right\|_{L_{p(\cdot),q}} \lesssim \left\|f\right\|_{L_{p(\cdot),q}}.
$$
\end{theorem}

\section{Applications in Fourier analysis}\label{sec5}

This section is devoted to applications of the previous results in Fourier Analysis. We mainly investigate the boundedness of the maximal Fej{\'e}r operator on variable Hardy space $H_{p(\cdot)}$ and variable Lorentz-Hardy space $H_{p(\cdot),q}$ (see Corollary \ref{cor:ad for M} and Corollary \ref{cor:equivalence}). To this end, in Section \ref{sec71}, we first introduce two new dyadic maximal operators $U$ and $V$ which play a crucial role in this section. We also prove that they are bounded on $L_{p(\cdot)}$ with $p(\cdot)$ satisfying \eqref{log} and $1<p_-\leq p_+<\infty$.

\subsection{Walsh system and Fej{\'e}r means}\label{sec71}

Let us investigate the dyadic martingales. Namely, let $\Omega=[0,1)$, $\mathbb P$ be the Lebesgue measure  and $\mathcal F$ be the Lebesgue measurable sets. By a {\it dyadic interval}, we mean one of the form $[k2^{-n},(k+1)2^{-n})$ for some $k,n \in \mathbb N$, $0 \leq k <2^n$. Given $n\in \mathbb N$ and $x \in [0,1)$,
let $I_n(x)$ denote the dyadic interval of length $2^{-n}$ which
contains $x$. The $\sigma$-algebras generated by the dyadic intervals $\{I_{n}(x): x\in [0,1)\}$
will be denoted by $\mathcal F_{n}$ $(n \in \mathbb N)$. Such $(\mathcal F_n)_{n\geq0}$ is regular, see Example \ref{ex:regular} or \cite{Long1993book}.

The \emph{Rademacher functions} are defined by
$$
r(x):= \left\{
         \begin{array}{ll}
           1, & \hbox{if $x\in [0,\frac{1}{2})$;} \\
           -1, & \hbox{if $x\in [\frac{1}{2},1)$,}
         \end{array}
       \right.
$$
and
$$
r_n(x):=r(2^n x) \qquad (x\in [0,1), \nn).
$$
The product system generated by the Rademacher functions is the {\it Walsh system}:
$$
w_{n}:=\prod_{k=0}^{\infty}{r_k}^{n_k}\qquad (n\in\N),
$$
where
\begin{equation}\label{e25}
n=\sum_{k=0}^{\infty} n_k 2^k, \qquad (0 \leq n_k <2).
\end{equation}

Recall (see Fine \cite{{Fine1949}}) that the {\it Walsh-Dirichlet kernels}
$$
D_n := \sum_{k=0}^{n-1} w_k
$$
satisfy
\begin{equation}\label{e5}
D_{2^n}(x) = \left\{
               \begin{array}{ll}
                 2^n, & \hbox{if $x \in [0,2^{-n})$;} \\
                 0, & \hbox{if $x \in [2^{-n},1)$}
               \end{array} \qquad (n\in \N).
             \right.
\end{equation}
If $f \in L_1$, then the number
$$
\widehat {f}(n) := \mathbb E(f w_{n}) \qquad (n\in\N)
$$
is said to be the $n$th {\it Walsh-Fourier coefficient} of $f$. We can extend this definition to martingales as follows. If $f=(f_{k})_{k\geq0}$ is a martingale, then let
$$
\widehat {f}(n) := \lim_{k \to \infty} \mathbb  E(f_{k} w_{n}) \qquad (n \in \N).
$$
Since $w_{n}$ is $\mathcal F_{k}$ measurable for $n<2^k$,
it can immediately be seen that this limit does exist.
We remember that if $f \in L_1$, then $\mathbb E_{k}f \to f$ in the $L_1$-norm as
$k \to \infty$, hence
$$
\widehat {f}(n) = \lim_{k \to \infty} \mathbb E(( \mathbb E_{k}f) w_{n}) \qquad (n \in \N).
$$
Thus the Walsh-Fourier coefficients of $f \in L_1$ are the same as the ones of
the martingale $(\mathbb E_{k}f)_{k\geq0}$ obtained from $f$.

Denote by $s_{n}f$ the $n$th partial sum of the Walsh-Fourier series of a martingale $f$, namely,
$$
s_{n}f := \sum_{k=0}^{n-1} \widehat {f}(k)w_{k}.
$$
If $f\in L_1$, then
$$
s_{n} f(x) = \int_0^1 f(t) D_{n}(x \dot + t) \, dt \qquad (n\in\N),
$$
where $\dot +$ denotes the dyadic addition (see e.g. Schipp, Wade, Simon and P{\'a}l \cite{sws} or Golubov, Efimov and Skvortsov \cite{Golubov1991}).
It is easy to see that
$$
s_{2^n}f=f_{n} \qquad (n\in \N)
$$
and so, by martingale results,
$$
\lim_{n\to\infty} s_{2^n}f=f \qquad \mbox{in the $L_p$-norm}
$$
when $f\in L_p$ and $1\leq p<\infty$. This theorem was extended in Schipp, Wade, Simon and P{\'a}l \cite{sws} (see also  Golubov, Efimov and Skvortsov \cite{Golubov1991}) for the partial sums $s_nf$ and for $1<p<\infty$. More exactly,
\begin{equation}\label{e100}
\lim_{n\to\infty} s_nf=f \qquad \mbox{in the $L_p$-norm}
\end{equation}
when $f\in L_p$ and $1<p<\infty$. We generalize this theorem as follows.

\begin{theorem}\label{t19 Leb}
Let $p(\cdot)$ satisfy \eqref{log} with $1< p_-\leq p_+ < \infty$.  If $f\in L_{p(\cdot)}$, then
$$
\sup_{n\in \mathbb N} \left\|s_nf\right\|_{{p(\cdot)}} \lesssim \left\|f\right\|_{{p(\cdot)}}.
$$
\end{theorem}

\begin{proof}
It was proved by Schipp, Wade, Simon and  P{\'a}l \cite[p. 95]{sws} that
$$
s_nf = w_n T_0(fw_n),
$$
where
$$
T_0 f := \sum_{k=1}^{\infty} n_{k-1} d_kf
$$
and the binary coefficients $n_k$ are defined in (\ref{e25}). Obviously $T_0$ is a martingale transform and Theorem \ref{t18 Leb} implies that
$$
\left\|s_nf\right\|_{{p(\cdot)}} = \left\|T_0(fw_n)\right\|_{{p(\cdot)}}
\lesssim \left\|f\right\|_{{p(\cdot)}},
$$
which shows the result.
\end{proof}

\begin{corollary}\label{c11 Leb}
Let $p(\cdot)$ satisfy \eqref{log} with $1< p_-\leq p_+ < \infty$. If $f\in L_{p(\cdot)}$, then
$$
\lim_{n\to\infty} s_{n} f=f \qquad \mbox{in the $L_{p(\cdot)}$-norm.}
$$
\end{corollary}

\begin{proof}
Note that (\ref{e100}) implies that the Walsh polynomials are dense in $L_p$ for $1\leq p<\infty$. Then it is easy to see that the Walsh polynomials are dense in $L_{p(\cdot)}$ as well.
Notice for any Walsh polynomial $T$ and any integer $n$ which exceeds the degree of this polynomial, that $s_n(T)=T$. Using Theorem \ref{t19 Leb} for such $n$, we obtain
$$\|T-s_n(f)\|_{p(\cdot)}=\|s_n(T-f)\|_{p(\cdot)}\leq C\|T-f\|_{p(\cdot)}.$$
We choose $T$ such that $\|T-f\|_{p(\cdot)}<\varepsilon/(C+1)$ for any $\varepsilon>0$.  Consequently,
$$\|f-s_n(f)\|_{p(\cdot)}\leq \|f-T\|_{p(\cdot)}+\|T-s_n(f)\|_{p(\cdot)}\leq (C+1)\|T-f\|_{p(\cdot)}<\varepsilon.$$
This finishes the proof.
\end{proof}

Similarly, for variable Lorentz spaces, we can prove the following two results.
\begin{theorem}\label{t19}
Let $p(\cdot)$ satisfy \eqref{log} with $1< p_-\leq p_+ < \infty$ and $0<q\leq \infty$.  If $f\in L_{p(\cdot),q}$, then
$$
\sup_{n\in \N} \left\|s_nf\right\|_{L_{p(\cdot),q}} \lesssim \left\|f\right\|_{L_{p(\cdot),q}}.
$$
\end{theorem}

\begin{corollary}\label{c11}
Let $p(\cdot)$ satisfy \eqref{log} with $1< p_-\leq p_+ < \infty$ and $0<q\leq \infty$. If $f\in L_{p(\cdot),q}$, then
$$
\lim_{n\to\infty} s_{n} f=f \qquad \mbox{in the $L_{p(\cdot),q}$-norm.}
$$
\end{corollary}

The above results are not true if $p_-\leq 1$, see e.g. \cite{Carleson1966} and Example 5.4.2 in \cite{Golubov1991}. However, in this case we can consider a summability method. For $n\in \mathbb N$ and a martingale $f$, the {\it Fej{\'e}r mean} of order $n$
of the Walsh-Fourier series of $f$ is given by
$$
\sigma_{n}f := \frac{1}{n} \sum_{k=1}^{n} s_{k} f.
$$
Of course, $\sigma_{n}f$ has better convergence properties than $s_{k} f$. It is simple to show that
$$
\sigma_{n}f(x) = \int_0^1 f(t) K_n(x \dot + t) \, dt \qquad (\nn)
$$
if $f \in L_1$, where the {\it Walsh-Fej\'er kernels} are defined by
$$
K_n := \frac{1}{n} \sum_{k=1}^n D_k \qquad (\nn).
$$
The maximal operator $\sigma_*$ is defined by
$$\sigma_*f =\sup_{n\in\mathbb N} |\sigma_nf|.$$

It is known (see Fine \cite{{Fine1949}} or Schipp, Wade, Simon and P{\'a}l \cite{sws}) that
\begin{equation}\label{e16}
|K_n(x)| \leq \sum_{j=0}^{N-1} 2^{j-N} \sum_{i=j}^{N-1}
\left(D_{2^i}(x) + D_{2^i}(x \dot + 2^{-j-1})\right)
\end{equation}
and
\begin{equation}\label{e17}
K_{2^n}(x) = \frac{1}{2} \left(2^{-n} D_{2^n}(x) + \sum_{j=0}^{n} 2^{j-n} D_{2^n}(x \dot + 2^{-j-1})\right),
\end{equation}
where $x \in [0,1)$, $n,N \in\N$ and $2^{N-1} \leq n < 2^N$.

\begin{remark}
In this section, if there is no special statement, we always assume that $(\mathcal F_n)_{n\geq0}$ is the sequence of the dyadic $\sigma$-algebras.
It follows from Theorem \ref{theorem of martingale inequalities Leb} that the five variable Hardy spaces in the theorem are equivalent if $(\mathcal F_n)_{n\geq0}$ is regular. We use $H_{p(\cdot)}$ to denote  one of them.
Similarly, according to Theorem \ref{theorem of martingale inequalities}, we use $H_{p(\cdot),q}$ to denote any one of the variable Lorentz Hardy spaces.
\end{remark}

\subsection{The maximal operator $U$}
Let us define $I_{k,n}:= [k2^{-n},(k+1)2^{-n})$ with $0\leq k<2^{n}$, $n\in \N$. Motivating by the kernel functions (\ref{e16}) and (\ref{e17}), we introduce two versions of dyadic maximal functions. For a martingale $f=(f_n)$, the first one is given by
$$
U_sf(x):= \sup_{x\in I} \sum_{j=0}^{n-1} 2^{(j-n)s}
\frac{1}{\mathbb P(I\dot + 2^{-j-1})} \left| \int_{I\dot + 2^{-j-1}} f_n d\mathbb{P}\right| ,
$$
where $I$ is a dyadic interval with length $2^{-n}$ and $s$ is a positive constants.
Of course, if $f\in L_1$, then we can write in the definition $f$ instead of $f_n$. The definition can be rewritten to
$$
U_sf(x)= \sup_{n\in \mathbb N} \sum_{k=0}^{2^n-1} \chi_{I_{k,n}}(x) \sum_{j=0}^{n-1} 2^{(j-n)s}
\frac{1}{\mathbb P(I_{k,n}^{j})} \left| \int_{I_{k,n}^{j}} f_n d\mathbb{P} \right|,
$$
where, for brevity, we use the notation
\[
	I_{k,n}^{j}:= I_{k,n}\dot + 2^{-j-1}.
\]

In order to show that $\sigma_*$ is bounded from $H_{p(\cdot)}$ to $L_{p(\cdot)}$, first we have to prove that $U$ is bounded from $L_{p(\cdot)}$ to $L_{p(\cdot)}$ $(p_->1)$.
We need to apply the following well-known theorem in martingale theory (see e.g. Weisz \cite{Weisz2002book}).

\begin{theorem}\label{t14}
Let $p$ be a constant and $0<p\leq 1<r\leq \infty$. Suppose that $T:L_r \rightarrow L_r$ is a bounded $\sigma$-sublinear operator and
\begin{equation}\label{e19}
\left\|Ta \chi_{I^c} \right\|_{p} \leq C_{p}
\end{equation}
for all  simple $(3,p,\infty)$-atoms $a$, where $I$ is the support of $a$. Then we have
$$
\|Tf\|_{{p}}\lesssim \|f\|_{H_{p}},\quad f\in H_{p}.
$$
\end{theorem}

\begin{theorem}\label{t12}
For all $0<p\leq \infty$ and all $0<s<\infty$, we have
\begin{equation}\label{e3}
\|U_s f\|_{p} \leq C_{p} \|f\|_{H_{p}} \qquad (f\in H_{p}).
\end{equation}
\end{theorem}

\begin{proof}
The theorem will be proved by applying Theorem \ref{t14} with $r=\infty$. Observe that \eqref{e3} holds for $p=\infty$. Indeed,
$$
\left\|U_sf\right\|_\infty \leq \sup_{n\in \N} \sum_{j=0}^{n-1} 2^{(j-n)s}
\left\|f\right\|_\infty \leq  C\left\|f\right\|_\infty.
$$
By interpolation, the proof will be complete if we show that the operator $U_s$ satisfies (\ref{e19}) for each $0<p\leq 1$. Choose a simple $(3,p,\infty)$-atom $a$ with support $I$, where $I$ is a dyadic interval with length $|I|= 2^{-K}$ $(K\in\N)$. We can assume that $I=[0,2^{-K})$. It is
easy to see that
$$
\sum_{j=0}^{n-1} 2^{(j-n)s}
\frac{1}{\mathbb P(J\dot + 2^{-j-1})} \left| \int_{J\dot + 2^{-j-1}} a d\mathbb{P}\right| =0
$$
if $n\leq K$, where $J$ is a dyadic interval with lenght $2^{-n}$. Therefore we can suppose that $n>K$.
Observe that $x \not\in [0,2^{-K})$ and $x\in  J$ imply that $J\dot + 2^{-j-1} \cap [0,2^{-K})=\emptyset$ if $j\geq K$. Thus $\int_{J\dot + 2^{-j-1}} a =0$ for $j\geq K$. Hence, we may assume that $j<K$. The same holds if $x\in [2^{-j-1}+2^{-K},2^{-j})$, because $x \dot + 2^{-j-1} \not\in [0,2^{-K})$. Hence
\begin{align*}
|U_sa(x) | &\leq \sup_{n>K} \chi_J(x) \sum_{j=0}^{K-1} 2^{(j-n)s} \chi_{[2^{-j-1},2^{-j-1}+2^{-K})}(x) \frac{1}{\mathbb P(J\dot + 2^{-j-1})} \left| \int_{J\dot + 2^{-j-1}} a d\mathbb{P}\right| \\
&\leq  2^{K/p} \sum_{j=0}^{K-1} 2^{(j-K)s} \chi_{[2^{-j-1},2^{-j-1}+2^{-K})}(x)
\end{align*}
and
$$
\int_{I^c} \left| U_sa(x) \right|^p \leq 2^{K} \sum_{j=0}^{K-1} 2^{(j-K)s p} 2^{-K} \leq C_p,
$$
which completes the proof of the theorem.
\end{proof}

Since $H_p$ is equivalent to $L_p$ when $1<p\leq \infty$ (see also Corollary \ref{c100}), the preceding result implies that
\begin{equation*}
\|U_s f\|_{p} \leq C_{p} \|f\|_{p} \qquad (1<p\leq \infty,0<s<\infty,f\in L_{p}).
\end{equation*}
This inequality remains true for Lebesgue spaces with variable exponents.

\begin{theorem} \label{t15}
Let $p(\cdot)\in\mathcal{P}(\Omega)$ satisfy condition \eqref{log}, $1<p_-\leq p_+<\infty$ and $0<s<\infty$. If
\begin{equation}\label{e102}
	\frac{1}{p_-}-\frac{1}{p_+} <s,
\end{equation}
then
$$
\|U_s f\|_{p(\cdot)} \leq C_{p(\cdot)} \|f\|_{{p(\cdot)}} \qquad (f\in L_{p(\cdot)}).
$$
\end{theorem}

\begin{proof}
We assume that $\|f\|_{p(\cdot)}\leq 1/2$. Then
\begin{align*}
	\int_{\Omega}|U_sf(x)|^{p(x)} \, dx &\lesssim \int_{\Omega}|U_s(f \chi_{|f| \geq 1})(x)|^{p(x)} \, dx + \int_{\Omega}|U_s(f \chi_{|f|<1})(x)|^{p(x)} \, dx \\
	&\lesssim \int_{\Omega}|U_s(f \chi_{|f| \geq 1})(x)|^{p(x)} \, dx + C.
\end{align*}
So it is enough to prove that
\begin{equation*}
	\int_{\Omega}|U_s(f \chi_{|f| \geq 1})(x)|^{p(x)} \,dx \lesssim C.
\end{equation*}
Let us denote by $\sum_{j=0}^{n-1}{}'$ the sum over all $j=1,\ldots,n-1$ for which
$$
\frac{1}{\mathbb P(I_{k,n}^{j})} \int_{I_{k,n}^{j}}|f(t)| \, dt \leq 1.
$$
In this case
\begin{align*}
	&\int_{\Omega} \left( \sup_{n\in \N} \sum_{k=0}^{2^n-1} \chi_{I_{k,n}}(x)\sum_{j=0}^{n-1}{}' 2^{(j-n)s} \frac{1}{\mathbb P(I_{k,n}^{j})} \int_{I_{k,n}^{j}}|f(t)| \, dt\right)^{p(x)} \, dx \\
	& \qquad \lesssim \int_{\Omega} \left(\sup_{n\in \N} \sum_{k=0}^{2^n-1} \chi_{I_{k,n}}(x)\sum_{j=0}^{n-1}{}' 2^{(j-n)s} \right)^{p(x)} \, dx \leq C.
\end{align*}
Hence, we may suppose that $\|f\|_{p(\cdot)}\leq 1/2$, $|f| \geq 1$ or $f=0$ and
\begin{equation}\label{e110}
\frac{1}{\mathbb P(I_{k,n}^{j})} \int_{I_{k,n}^{j}}|f(t)| \, dt > 1
\end{equation}
for  all $j=1,\ldots,n-1$, $k=0,\ldots,2^n-1$, $n\in \N$.

Let us denote by $I_{k,n,j,1}$ (resp. $I_{k,n,j,2}$) those points $x \in I_{k,n}$ for which $p(x) \leq p_+(I_{k,n}^{j})$ (resp. $p(x) > p_+(I_{k,n}^{j})$). Then
\begin{align*}
	\int_{\Omega}|U_sf(x)|^{p(x)} \,dx  
	&\lesssim \sum_{l=1}^{2}\int_{\Omega} \left( \sup_{n\in \N} \sum_{k=0}^{2^n-1} \chi_{I_{k,n}}(x) \sum_{j=0}^{n-1} 2^{(j-n)s} \frac{\chi_{I_{k,n,j,l}}(x)}{\mathbb P(I_{k,n}^{j})} \int_{I_{k,n}^{j}}|f(t)| \, dt\right)^{p(x)} \, dx \\
	&=: (A)+(B).
\end{align*}
Let $q(x):=p(x)/p_0>1$ for some $1<p_0<p_-$. Using the fact that the sets $I_{k,n}$ are disjoint for a fixed $n$ and the convexity of the function $t\mapsto t^{q(x)}$ ($x$ is fixed), we conclude
\begin{align*}
	(A) &\lesssim \int_{\Omega} \left( \sup_{n\in \N} \sum_{k=0}^{2^n-1} \chi_{I_{k,n}}(x) \left(\sum_{j=0}^{n-1} 2^{(j-n)s} \frac{\chi_{I_{k,n,j,1}}(x)}{\mathbb P(I_{k,n}^{j})} \int_{I_{k,n}^{j}}|f(t)| \, dt\right)^{q(x)}\right)^{p_0} \, dx \\
	&\lesssim \int_{\Omega} \left( \sup_{n\in \N} \sum_{k=0}^{2^n-1} \chi_{I_{k,n}}(x) \sum_{j=0}^{n-1} 2^{(j-n)s} \left(\frac{\chi_{I_{k,n,j,1}}(x)}{\mathbb P(I_{k,n}^{j})} \int_{I_{k,n}^{j}}|f(t)| \, dt\right)^{q(x)}\right)^{p_0} \, dx \\
	&\lesssim \int_{\Omega} \left( \sup_{n\in \N} \sum_{k=0}^{2^n-1} \chi_{I_{k,n}}(x) \sum_{j=0}^{n-1} 2^{(j-n)s} \left(\frac{\chi_{I_{k,n,j,1}}(x)}{\mathbb P(I_{k,n}^{j})} \int_{I_{k,n}^{j}}|f(t)| \, dt\right)^{q_+(I_{k,n}^{j})}\right)^{p_0} \, dx
\end{align*}
because of \eqref{e110} and the fact that $q(x) \leq q_+(I_{k,n}^{j})$ on $I_{k,n,j,1}$. By Lemma \ref{lemma of conditional expc},
\begin{align*}
	(A) &\lesssim \int_{\Omega} \left( \sup_{n\in \N} \sum_{k=0}^{2^n-1} \chi_{I_{k,n}}(x) \sum_{j=0}^{n-1} 2^{(j-n)s} \frac{\chi_{I_{k,n,j,1}}(x)}{\mathbb P(I_{k,n}^{j})} \int_{I_{k,n}^{j}}|f(t)|^{q(t)} \, dt\right)^{p_0} \, dx\\ 
	&\lesssim \left\| U_s(|f|^{q(\cdot)}) \right\|_{p_0}^{p_0} \lesssim \left\| |f|^{q(\cdot)} \right\|_{p_0}^{p_0} \leq C.
\end{align*}

To investigate $(B)$, let us observe that
\begin{align*}
	(B)	&\lesssim \int_{\Omega} \left( \sup_{n\in \N} \sum_{k=0}^{2^n-1} \chi_{I_{k,n}}(x)\left(\sum_{j=0}^{n-1} 2^{(j-n)(s-r)} 2^{(j-n)r} \frac{\chi_{I_{k,n,j,2}}(x)}{\mathbb P(I_{k,n}^{j})} \int_{I_{k,n}^{j}}|f(t)| \, dt\right)^{q(x)}\right)^{p_0} \, dx \\
	&\lesssim \int_{\Omega} \left( \sup_{n\in \N} \sum_{k=0}^{2^n-1} \chi_{I_{k,n}}(x) \sum_{j=0}^{n-1} 2^{(j-n)(s-r)} \left(2^{(j-n)r} \frac{\chi_{I_{k,n,j,2}}(x)}{\mathbb P(I_{k,n}^{j})} \int_{I_{k,n}^{j}}|f(t)| \, dt\right)^{q(x)}\right)^{p_0} \, dx
\end{align*}
for some $0<r<s$. By Hölder's inequality,
\begin{align*}
	(B)	&\lesssim \int_{\Omega} \left( \sup_{n\in \N} \sum_{k=0}^{2^n-1} \chi_{I_{k,n}}(x) \sum_{j=0}^{n-1} 2^{(j-n)(s-r)} 2^{(j-n)rq(x)} \right. \\
	&\qquad \left.\left(\frac{\chi_{I_{k,n,j,2}}(x)}{\mathbb P(I_{k,n}^{j})} \int_{I_{k,n}^{j}}|f(t)|^{q_-(I_{k,n}^{j})} \, dt\right)^{q(x)/q_-(I_{k,n}^{j})}\right)^{p_0} \, dx \\
	&\lesssim \int_{\Omega} \left( \sup_{n\in \N} \sum_{k=0}^{2^n-1} \chi_{I_{k,n}}(x) \sum_{j=0}^{n-1} 2^{(j-n)(s-r)} 2^{(j-n)rq(x)} 2^{nq(x)/q_-(I_{k,n}^{j})} \chi_{I_{k,n,j,2}}(x)\right. \\
	&\qquad \left.\left(\int_{I_{k,n}^{j}}|f(t)|^{q_-(I_{k,n}^{j})} \, dt\right)^{q(x)/q_-(I_{k,n}^{j})}\right)^{p_0} \, dx.
\end{align*}
Since $|f| \geq 1$ or $f=0$, $q(x)>q_-(I_{k,n}^{j})$ on $I_{k,n,j,2}$, $q_-(I_{k,n}^{j}) \leq q(t)< p(t)$ for all $t \in I_{k,n}^{j}$ and 
\[
	\int_{I_{k,n}^{j}}|f(t)|^{q_-(I_{k,n}^{j})} \, dt \leq \int_{I_{k,n}^{j}}|f(t)|^{p(t)} \, dt \leq  \frac{1}{2},
\]
we conclude
\begin{align*}
	(B)	&\lesssim \int_{\Omega} \left( \sup_{n\in \N} \sum_{k=0}^{2^n-1} \chi_{I_{k,n}}(x) \sum_{j=0}^{n-1} 2^{(j-n)(s-r)} 2^{(j-n)rq(x)} 2^{nq(x)/q_-(I_{k,n}^{j})} \chi_{I_{k,n,j,2}}(x)\right. \\
	&\qquad \left.\int_{I_{k,n}^{j}}|f(t)|^{q_-(I_{k,n}^{j})} \, dt \right)^{p_0} \, dx\\
	&\lesssim \int_{\Omega} \left( \sup_{n\in \N} \sum_{k=0}^{2^n-1} \chi_{I_{k,n}}(x) \sum_{j=0}^{n-1} 2^{(j-n)(s-r)} 2^{(j-n)rq(x)} 2^{nq(x)/q_-(I_{k,n}^{j})} \chi_{I_{k,n,j,2}}(x)\right. \\
	&\qquad \left. \frac{2^{-n}}{\mathbb P(I_{k,n}^{j})}\int_{I_{k,n}^{j}}|f(t)|^{q(t)} \, dt \right)^{p_0} \, dx.
\end{align*}
For fixed $k,n$ let $J_j$ denote the dyadic interval with length $2^{-j}$ and $I_{k,n} \subset J_j$. Then $I_{k,n}^{j} \subset J_j\dot + 2^{-j-1} = J_j$. Inequality \eqref{log} implies that $2^{-jp(x)} \sim 2^{-j p_-(I_{k,n}^{j})}$ for $x \in  I_{k,n}$. It is easy to check that for $x \in  I_{k,n,j,2}$,
\begin{align*}
	2^{^{jrq(x)}} &= 2^{^{jrq(x)}} 2^{jq(x)} 2^{-jq(x)} \lesssim 2^{^{j(r+1)q_-(I_{k,n}^{j})}} 2^{-jq(x)}< 2^{j \left(rq(x) - \frac{q(x)-q_-(I_{k,n}^{j})}{q_-(I_{k,n}^{j})}\right)}  = 2^{j \left(rq(x) - \frac{q(x)}{q_-(I_{k,n}^{j})}+1 \right)} 
\end{align*}
which is equivalent to the obvious inequality $q_-(I_{k,n}^{j})>1/(r+1)$. Furthermore,
\begin{align*}
	rq(x) - \frac{q(x)}{q_-(I_{k,n}^{j})}+1 &\geq q(x)\left(r - \frac{1}{q_-}\right)+1 \geq \left\{
		\begin{array}{ll}
		1, & \hbox{if $r - \frac{1}{q_-} \geq 0$;} \\ 
		q_+\left(r - \frac{1}{q_-}\right)+1, & \hbox{if $r - \frac{1}{q_-} < 0$.} 
	 \end{array} \right.
\end{align*}
Let $r_0:= \min \left(1,q_+\left(r - \frac{1}{q_-}\right)+1\right)$. Then $r_0>0$ if and only if 
\begin{equation}\label{e101}
	\frac{1}{q_-}-\frac{1}{q_+} <r.
\end{equation}
We estimate $(B)$ further by
\begin{align*}
	(B)	&\lesssim \int_{\Omega} \left( \sup_{n\in \N} \sum_{k=0}^{2^n-1} \chi_{I_{k,n}}(x) \sum_{j=0}^{n-1} 2^{(j-n)(s-r)} 2^{(j-n)\left(rq(x) - \frac{q(x)}{q_-(I_{k,n}^{j})}+1 \right)} \right. \\
	&\qquad \left. \frac{1}{\mathbb P(I_{k,n}^{j})}\int_{I_{k,n}^{j}}|f(t)|^{q(t)} \, dt \right)^{p_0} \, dx\\
	&\lesssim \int_{\Omega} \left( \sup_{n\in \N} \sum_{k=0}^{2^n-1} \chi_{I_{k,n}}(x) \sum_{j=0}^{n-1} 2^{(j-n)(s-r+r_0)} \frac{1}{\mathbb P(I_{k,n}^{j})}\int_{I_{k,n}^{j}}|f(t)|^{q(t)} \, dt \right)^{p_0} \, dx\\
	&\lesssim \left\| U_{s-r+r_0}(|f|^{q(\cdot)}) \right\|_{p_0}^{p_0} \lesssim \left\| |f|^{q(\cdot)} \right\|_{p_0}^{p_0} \leq C.
\end{align*}
Since $p_0$ can be arbitrarily near to $1$ and $r$ to $s$, inequality \eqref{e101} proves the theorem with the range \eqref{e102}.
\end{proof}

\begin{remark}
Inequality \eqref{e102} and Theorem \ref{t15} hold if $s \geq 1$, or more generally if $p_->\max(1/s,1)$.
\end{remark}

The operator $U_s$ is not bounded on $L_{p(\cdot)}$ outside the range of \eqref{e102}. More exactly, the following theorem holds.

\begin{theorem} \label{t10}
Let $p(\cdot)\in\mathcal{P}(\Omega)$ satisfy condition \eqref{log}, $1<p_-\leq p_+<\infty$ and $0<s<\infty$. If 
\begin{equation}\label{e103}
	\frac{1}{p_-(I_{0,n}\dot + 2^{-1})}-\frac{1}{p_+(I_{0,n})} >s
\end{equation}
for all $n \in \N$, then $U_s$ is not bounded on $L_{p(\cdot)}$.
\end{theorem}

\begin{proof}
	It is easy to see that
	\begin{align*}
	\int_{\Omega}&|U_sf(x)|^{p(x)} \,dx \geq \int_{\Omega} \chi_{I_{0,n}}(x) \left( 2^{-ns} \frac{1}{\mathbb P(I_{0,n}\dot + 2^{-1})} \int_{I_{0,n}\dot + 2^{-1}}|f(t)| \, dt\right)^{p(x)} \, dx.
	\end{align*}
	If
	\[
		f(t):= \chi_{I_{0,n}\dot + 2^{-1}}(t) 2^{n/p_-(I_{0,n}\dot + 2^{-1})},
	\]
	by Lemma \ref{lemma of norm of set}, 
	\[
		\|f\|_{p(\cdot)} =2^{n/p_-(I_{0,n}\dot + 2^{-1})} \|\chi_{I_{0,n}\dot + 2^{-1}}\|_{p(\cdot)} \leq C.
	\]
	This implies that
	\begin{align*}
	\int_{\Omega}|U_sf(x)|^{p(x)} \,dx &\geq \int_{I_{0,n}} 2^{-nsp(x)} 2^{np(x)/p_-(I_{0,n}\dot + 2^{-1})} \, dx \\
	&\geq C \int_{I_{0,n}} 2^{np_+(I_{0,n})(1/p_-(I_{0,n}\dot + 2^{-1})-s)} \, dx = C 2^{np_+(I_{0,n})(1/p_-(I_{0,n}\dot + 2^{-1})-s)} 2^{-n}
	\end{align*}
	which tends to infinity as $n\to \infty$ if \eqref{e103} holds.
\end{proof}

\begin{remark}\label{rem:U}
Combining the fact that $U$ is bounded on $L_\infty$, the above theorem and Lemma \ref{lemma for interpolation}, we know that
$U$ is bounded on $L_{p(\cdot),q}$ for $p(\cdot)\in \mathcal P(\Omega)$ satisfying $\eqref{log}$ and \eqref{e102}, $1<p_-\leq p_+<\infty$ and $0<q\leq\infty$.

These results, including the above theorem and the remark, should be compared with Theorem \ref{thm:maximal inequality} and Corollary \ref{c10}.
\end{remark}

\subsection{The maximal operator $V$}
We define the second version of dyadic maximal function by
\begin{eqnarray*}
V_{\alpha,s}f(x) := \sup_{x\in I} \sum_{j=0}^{n-1}\sum_{i=j}^{n-1} 2^{(j-n)\alpha} 2^{(j-i)s} \frac{1}{\mathbb P(I\dot + [2^{-j-1},2^{-j-1} \dot + 2^{-i}))}
\left| \int_{I\dot + [2^{-j-1},2^{-j-1} \dot + 2^{-i})} f_n d\mathbb{P} \right|,
\end{eqnarray*}
where $I$ is a dyadic interval with length $2^{-n}$ and $f=(f_n)$ is a martingale and  $s, \alpha$ are positive constants. Obviously,
\begin{eqnarray*}
V_{\alpha,s}f(x) := \sup_{n\in \N} \sum_{k=0}^{2^n-1} \chi_{I_{k,n}} \sum_{j=0}^{n-1}\sum_{i=j}^{n-1} 2^{(j-n)\alpha} 2^{(j-i)s} \frac{1}{\mathbb P(I_{k,n}^{j,i})}
\left| \int_{I_{k,n}^{j,i}} f_n d\mathbb{P}\right|,
\end{eqnarray*}
where, for brevity, we use the notation
\[
	I_{k,n}^{j,i}:= I_{k,n}\dot + [2^{-j-1},2^{-j-1} \dot + 2^{-i}).
\]

\begin{theorem}\label{t16}
Suppose that $0<p\leq \infty$ and $0< \alpha,s< \infty$. Then
\begin{equation*}
\|V_{\alpha,s} f\|_{p} \leq C_{p} \|f\|_{H_{p}} \qquad (f\in H_{p}).
\end{equation*}
\end{theorem}

\begin{proof}
The inequality holds for $p=\infty$ because
$$
\left\|V_{\alpha,s}f\right\|_\infty
\leq \sup_{n\in \N} \sum_{j=0}^{n-1}\sum_{i=j}^{n-1} 2^{(j-n)\alpha} 2^{(j-i)s}\left\|f\right\|_\infty \leq C\left\|f\right\|_\infty.
$$

Again, we are going to show that the operator $V_{\alpha,s}$ satisfies \eqref{e19} for each $0<p\leq 1$. We choose again a simple $(3,p,\infty)$-atom $a$ with support $I=[0,2^{-K})$. If $i\leq K$, then $\int_{I\dot + [2^{-j-1},2^{-j-1} \dot + 2^{-i})} a =0$. Thus $i>K$ and so $n>K$. Similarly to the proof of Theorem \ref{t12}, $j<K$ and $x\in [2^{-j-1},2^{-j-1}+2^{-K})$. Hence, in case $x\not\in [0,2^{-K})$,
\begin{eqnarray*}
\left| V_{\alpha,s}a(x) \right|
&\leq& \sup_{n>K} \chi_J(x) \sum_{j=0}^{K-1}\sum_{i=K}^{n-1} 2^{(j-n)\alpha} 2^{(j-i)s} \frac{1}{\mathbb P(J^{j,i})} \left| \int_{J^{j,i}} a \right| \chi_{[2^{-j-1},2^{-j-1}+2^{-K})}(x) \\
&\leq& 2^{K/p} \sup_{n>K} \chi_J(x) 2^{-n\alpha} \sum_{j=0}^{K-1} 2^{j \alpha} \sum_{i=K}^{n-1} 2^{(j-i)s} \chi_{[2^{-j-1},2^{-j-1}+2^{-K})}(x),
\end{eqnarray*}
where $J$ is a dyadic interval with length $2^{-n}$. Since
\begin{equation*}\label{}
	\sum_{i=K}^{n-1} 2^{(j-i)s} \leq 1,
\end{equation*}
we have
\begin{equation*}\label{}
	\left| V_{\alpha,s}a(x) \right| \leq 2^{K/p} 2^{-K\alpha} \sum_{j=0}^{K-1} 2^{j\alpha} \chi_{[2^{-j-1},2^{-j-1}+2^{-K})}(x).
\end{equation*}
Consequently,
$$
\int_{I^c} \left| V_{\alpha,s}a(x) \right|^p \leq 2^{K} 2^{-K\alpha p} \sum_{j=0}^{K-1} 2^{j\alpha p} 2^{-K} \leq C_p,
$$
which finishes the proof.
\end{proof}

Under the same conditions, the inequality
\begin{equation*}
\|V_{\alpha,s} f\|_{p} \leq C_{p} \|f\|_{p} \qquad (1<p\leq \infty,0<\alpha,s<\infty,f\in L_{p})
\end{equation*}
follows from Theorem \ref{t16}.

\begin{theorem} \label{t17}
Let $p(\cdot)\in\mathcal{P}(\Omega)$ satisfy condition \eqref{log}, $1<p_-\leq p_+<\infty$ and $0<\alpha,s<\infty$. If
\begin{equation}\label{e105}
	\frac{1}{p_-}-\frac{1}{p_+} < \alpha+s,
\end{equation}
then
$$
\|V_{\alpha,s}f\|_{p(\cdot)} \leq C_{p(\cdot)} \|f\|_{{p(\cdot)}} \qquad (f\in L_{p(\cdot)}).
$$
\end{theorem}

\begin{proof}
Similarly to the proof of Theorem \ref{t15},  we may suppose again that $\|f\|_{p(\cdot)}\leq 1/2$, $|f| \geq 1$ or $f=0$ and
\[
	\frac{1}{\mathbb P(I_{k,n}^{j,i})} \int_{I_{k,n}^{j,i}} |f(t)| \, dt >1.
\]
We denote by $I_{k,n,j,i,1}$ (resp. $I_{k,n,j,i,2}$) those points $x \in I_{k,n}$ for which $p(x) \leq p_+(I_{k,n}^{j,i})$ (resp. $p(x) > p_+(I_{k,n}^{j,i})$). Then
\begin{align*}
	\int_{\Omega}&|V_{\alpha,s}f(x)|^{p(x)} \,dx \\ 
	&\lesssim \sum_{l=1}^{2}\int_{\Omega} \left( \sup_{n\in \N} \sum_{k=0}^{2^n-1} \chi_{I_{k,n}}(x) \sum_{j=0}^{n-1}\sum_{i=j}^{n-1} 2^{(j-n)\alpha} 2^{(j-i)s} \frac{\chi_{I_{k,n,j,i,l}}(x)}{\mathbb P(I_{k,n}^{j,i})} \int_{I_{k,n}^{j,i}}|f(t)| \, dt\right)^{p(x)} \, dx \\
	&=: (C)+(D).
\end{align*}
Again, let $q(x):=p(x)/p_0>1$ for some $1<p_0<p_-$. By convexity and Lemma \ref{lemma of conditional expc},
\begin{align*}
	(C) &\lesssim \int_{\Omega} \left( \sup_{n\in \N} \sum_{k=0}^{2^n-1} \chi_{I_{k,n}}(x) \sum_{j=0}^{n-1}\sum_{i=j}^{n-1} 2^{(j-n)\alpha} 2^{(j-i)s} \left(\frac{\chi_{I_{k,n,j,i,1}}(x)}{\mathbb P(I_{k,n}^{j,i})} \int_{I_{k,n}^{j,i}}|f(t)| \, dt\right)^{q(x)}\right)^{p_0} \, dx \\
	&\lesssim \int_{\Omega} \left( \sup_{n\in \N} \sum_{k=0}^{2^n-1} \chi_{I_{k,n}}(x) \sum_{j=0}^{n-1}\sum_{i=j}^{n-1} 2^{(j-n)\alpha} 2^{(j-i)s} \left(\frac{\chi_{I_{k,n,j,i,1}}(x)}{\mathbb P(I_{k,n}^{j,i})} \int_{I_{k,n}^{j,i}}|f(t)| \, dt\right)^{q_+(I_{k,n}^{j,i})}\right)^{p_0} \, dx \\
	&\lesssim \int_{\Omega} \left( \sup_{n\in \N} \sum_{k=0}^{2^n-1} \chi_{I_{k,n}}(x) \sum_{j=0}^{n-1}\sum_{i=j}^{n-1} 2^{(j-n)\alpha} 2^{(j-i)s} \frac{\chi_{I_{k,n,j,i,1}}(x)}{\mathbb P(I_{k,n}^{j,i})} \int_{I_{k,n}^{j,i}}|f(t)|^{q(t)} \, dt\right)^{p_0} \, dx \\ 
	&\lesssim \left\| V_{\alpha,s}(|f|^{q(\cdot)}) \right\|_{p_0}^{p_0} \lesssim \left\| |f|^{q(\cdot)} \right\|_{p_0}^{p_0} \leq C.
\end{align*}

Again by convexity and Hölder's inequality, we obtain for some $0<\alpha_0 < \alpha$ and $0<r<s+\alpha_0$ that
\begin{align*}
	(D)	&\lesssim \int_{\Omega} \left( \sup_{n\in \N} \sum_{k=0}^{2^n-1} \chi_{I_{k,n}}(x)\left(\sum_{j=0}^{n-1}\sum_{i=j}^{n-1} 2^{(j-n)\alpha} 2^{(j-i)s} \frac{\chi_{I_{k,n,j,i,2}}(x)}{\mathbb P(I_{k,n}^{j,i})} \int_{I_{k,n}^{j,i}}|f(t)| \, dt\right)^{q(x)}\right)^{p_0} \, dx \\
	&\lesssim \int_{\Omega} \left( \sup_{n\in \N} \sum_{k=0}^{2^n-1} \chi_{I_{k,n}}(x) \sum_{j=0}^{n-1}\sum_{i=j}^{n-1} 2^{(j-n)(\alpha-\alpha_0)} 2^{(j-i)(\alpha_0+s-r)} \right. \\
	& \qquad \left. \left(2^{(j-i)r} \frac{\chi_{I_{k,n,j,i,2}}(x)}{\mathbb P(I_{k,n}^{j,i})} \int_{I_{k,n}^{j,i}}|f(t)| \, dt\right)^{q(x)}\right)^{p_0} \, dx
\end{align*}
and so
\begin{align*}
	(D)	&\lesssim \int_{\Omega} \left( \sup_{n\in \N} \sum_{k=0}^{2^n-1} \chi_{I_{k,n}}(x) \sum_{j=0}^{n-1}\sum_{i=j}^{n-1} 2^{(j-n)(\alpha-\alpha_0)} 2^{(j-i)(\alpha_0+s-r)} 2^{(j-i)rq(x)} \right. \\
	&\qquad \left.\left(\frac{\chi_{I_{k,n,j,i,2}}(x)}{\mathbb P(I_{k,n}^{j,i})} \int_{I_{k,n}^{j,i}}|f(t)|^{q_-(I_{k,n}^{j,i})} \, dt\right)^{q(x)/q_-(I_{k,n}^{j,i})}\right)^{p_0} \, dx \\
	&\lesssim \int_{\Omega} \left( \sup_{n\in \N} \sum_{k=0}^{2^n-1} \chi_{I_{k,n}}(x) \sum_{j=0}^{n-1}\sum_{i=j}^{n-1} 2^{(j-n)(\alpha-\alpha_0)} 2^{(j-i)(\alpha_0+s-r)} 2^{(j-i)rq(x)} 2^{iq(x)/q_-(I_{k,n}^{j,i})} \right. \\
	&\qquad \left.\chi_{I_{k,n,j,i,2}}(x)\left(\int_{I_{k,n}^{j,i}}|f(t)|^{q_-(I_{k,n}^{j,i})} \, dt\right)^{q(x)/q_-(I_{k,n}^{j,i})}\right)^{p_0} \, dx.
\end{align*}
Since
\[
	\int_{I_{k,n}^{j}}|f(t)|^{q_-(I_{k,n}^{j})} \, dt \leq \int_{I_{k,n}^{j}}|f(t)|^{p(t)} \, dt \leq  \frac{1}{2},
\]
we can see that
\begin{align*}
	(D)	&\lesssim \int_{\Omega} \left( \sup_{n\in \N} \sum_{k=0}^{2^n-1} \chi_{I_{k,n}}(x) \sum_{j=0}^{n-1}\sum_{i=j}^{n-1} 2^{(j-n)(\alpha-\alpha_0)} 2^{(j-i)(\alpha_0+s-r)} 2^{(j-i)rq(x)} 2^{iq(x)/q_-(I_{k,n}^{j,i})} \right. \\
	&\qquad \left. \chi_{I_{k,n,j,i,2}}(x)\frac{2^{-i}}{\mathbb P(I_{k,n}^{j,i})} \int_{I_{k,n}^{j,i}}|f(t)|^{q_-(I_{k,n}^{j,i})} \, dt\right)^{p_0} \, dx.
\end{align*}

Similarly to the proof of Theorem \ref{t15}, we get
\begin{align*}
	(D)	&\lesssim \int_{\Omega} \left( \sup_{n\in \N} \sum_{k=0}^{2^n-1} \chi_{I_{k,n}}(x) \sum_{j=0}^{n-1}\sum_{i=j}^{n-1} 2^{(j-n)(\alpha-\alpha_0)} 2^{(j-i)(\alpha_0+s-r)} 2^{(j-i)\left(rq(x) - \frac{q(x)}{q_-(I_{k,n}^{j,i})}+1 \right)} \right. \\
	&\qquad \left. \frac{1}{\mathbb P(I_{k,n}^{j,i})}\int_{I_{k,n}^{j,i}}|f(t)|^{q(t)} \, dt \right)^{p_0} \, dx\\
	&\lesssim \int_{\Omega} \left( \sup_{n\in \N} \sum_{k=0}^{2^n-1} \chi_{I_{k,n}}(x) \sum_{j=0}^{n-1}\sum_{i=j}^{n-1} 2^{(j-n)(\alpha-\alpha_0)} 2^{(j-i)(\alpha_0+s-r+r_0)} \frac{1}{\mathbb P(I_{k,n}^{j,i})}\int_{I_{k,n}^{j,i}}|f(t)|^{q(t)} \, dt \right)^{p_0} \, dx\\
	&\lesssim \left\| V_{\alpha- \alpha_0,\alpha_0+s-r+r_0}(|f|^{q(\cdot)}) \right\|_{p_0}^{p_0} \lesssim \left\| |f|^{q(\cdot)} \right\|_{p_0}^{p_0} \leq C,
\end{align*}
whenever \eqref{e101} holds. Note that $r_0$ was defined just before \eqref{e101}. Since $r$ can be arbitrarily near to $s+ \alpha_0$ and $\alpha_0$ to $\alpha$, this completes the proof.
\end{proof}

\begin{remark}
Inequality \eqref{e105} and Theorem \ref{t17} hold if $p_->\max(1/(\alpha+s),1)$.
\end{remark}

The operator $V_{\alpha,s}$ is not bounded on $L_{p(\cdot)}$ if \eqref{e105} is not true.

\begin{theorem} \label{t100}
Let $p(\cdot)\in\mathcal{P}(\Omega)$ satisfy condition \eqref{log}, $1<p_-\leq p_+<\infty$ and $0<\alpha,s<\infty$. If 
\begin{equation*}
	\frac{1}{p_-(I_{0,n}\dot + 2^{-1})}-\frac{1}{p_+(I_{0,n})} > \alpha+s
\end{equation*}
for all $n \in \N$, then $V_{\alpha,s}$ is not bounded on $L_{p(\cdot)}$.
\end{theorem}

\begin{proof}
	Choosing $j=0$ and $i=n-1$, the theorem can be shown in the same way as Theorem \ref{t10}.
\end{proof}

\begin{remark}
Similarly to Remark \ref{rem:U}, we know that $V_{\alpha,s}$ is also bounded on $L_{p(\cdot),q}$ for $p(\cdot)\in \mathcal P(\Omega)$ satisfying $\eqref{log}$ and \eqref{e105}, $1<p_-\leq p_+<\infty$ and $0<q\leq\infty$.
\end{remark}

\subsection{The maximal Fej{\'e}r operator on $H_{p(\cdot)}$}\label{7.2}
In this subsection,  we apply the atomic characterization via $(3,p(\cdot),\infty)$-atoms to prove the boundedness of $\sigma_*$ from $H_{p(\cdot)}$ to $L_{p(\cdot)}$. We first generalize Theorem \ref{t14} to the result below.

\begin{theorem}\label{t1 Leb}
Let $p(\cdot)\in\mathcal{P}(\Omega)$ satisfy \eqref{log} and $0<t<\underline p$. Suppose that the $\sigma$-sublinear operator $T:L_\infty\rightarrow L_\infty$ is bounded and
\begin{equation}\label{e1 Leb}
\left\|\sum_k\mu_k^{t}T(a^k)^{t}\chi_{\{\tau_k=\infty\}}\right\|_{\frac{p(\cdot)}{t}}\lesssim \left\|\sum_k 2^{kt} \chi_{\{\tau_k<\infty\}}\right\|_{\frac{p(\cdot)}{t}},
\end{equation}
where $\tau_k$ is the stopping time associated with $(3,p(\cdot),\infty)$-atom $a^k$.
Then we have
$$\|Tf\|_{p(\cdot)}\lesssim \|f\|_{H_{p(\cdot)}}.$$
\end{theorem}

\begin{proof}
According to Corollary \ref{cor:ad for M}, $f$ can be written as
$$
f=\sum_k \mu_k a^k, \quad \mbox{where} \quad \mu_k=3\cdot 2^k \|\chi_{\{\tau_k<\infty\}}\|_{p(\cdot)}
$$
and $a^k$'s are $(3,p(\cdot),\infty)$-atoms associated with stopping times $(\tau_k)_{k\in\mathbb Z}$.
Then
\begin{align*}
\|Tf\|_{p(\cdot)} &\lesssim \left\| \sum_{k} \mu_k T(a^k)\chi_{\{\tau_k<\infty\}} \right\|_{p(\cdot)} +\left\| \sum_{k} \mu_k T(a^k)\chi_{\{\tau_k=\infty\}} \right\|_{p(\cdot)}
\\& =:Z_1+Z_2.
\end{align*}

We first estimate  $Z_1$. The sets $\left\{\tau_k=j\right\}$ are disjoint and there exist disjoint atoms $I_{k,j,i}\in \mathcal{F}_j$ such that
$\left\{\tau_k=j\right\}=\bigcup_i I_{k,j,i}.$
Thus
$$\{\tau_k<\infty\}=\bigcup_{j\in \mathbb N} \bigcup_i I_{k,j,i},$$
where $I_{k,j,i} $ are disjoint for fixed $k$. For convenience,  we will write
$$\{\tau_k<\infty\}=\bigcup_{l}  I_{k_l}.$$
Since $0<t <\underline p\leq 1$, we have
$$
Z_1\leq  \left\|\sum_k\mu_k^{t} \sum_{l}T(a^k)^{t}\chi_{I_{k_l}}\right\|_{\frac{p(\cdot)}{t}}^{\frac{1}{t}}.
$$
By Lemma \ref{lem:duality for variabl p}, choose $g\in L_{(\frac{p(\cdot)}{t})'}$  with norm less than $1$ such that
$$\left\|\sum_k\mu_k^{t} \sum_{l}T(a^k)^{t}\chi_{I_{k_l}}\right\|_{\frac{p(\cdot)}{t}}= \int_\Omega \sum_k\mu_k^{t} \sum_{l}T(a^k)^{t}\chi_{I_{k_l}}g d\mathbb P.$$
Note that $T:L_\infty\rightarrow L_\infty$.  Then, by H\"{o}lder's inequality for $p_+/t<r<\infty$ and the definition of $(3,p(\cdot),\infty)$-atoms, we obtain
\begin{align*}
Z_1^t& \leq \int \sum_k\mu_k^{t} \sum_{l}T(a^k)^{t}\chi_{I_{k_l}} g d\mathbb P
\\& \leq \sum_k\mu_k^{t} \sum_{l} \|T(a^k)^{t}\chi_{I_{k_l}}\|_{r} \|\chi_{I_{k_l}} g \|_{r'}
\\& \lesssim \sum_{k}\sum_{l}(3\cdot 2^k)^{t} \|\chi_{\tau_k<\infty}\|_{p(\cdot)}^{t}  \|T(a^k)^{t}\|_\infty \|\chi_{I_{k_l}}\|_{r} \|\chi_{I_{k_l}} g \|_{r'}
\\&\leq  \sum_{k}\sum_{l}(3\cdot 2^k)^{t } \mathbb P(I_{k_l}) \left(\frac{1}{\mathbb P(I_{k_l})}\int_{I_{k_l}}g^{r'} \right)^{\frac{1}{r'}}
\\& \leq \sum_{k}\sum_{l} (3\cdot 2^k)^{t} \int\chi_{I_{k_l}} [M(g^{r'})]^{\frac{1}{r'}}d\mathbb P
\\& \leq \left\|\sum_{k}\sum_{l} (3\cdot 2^k)^{t}\chi_{I_{k_l}} \right \|_{\frac{p(\cdot)}{t}} \|[M(g^{r'})]^{\frac{1}{r'}}\|_{{(\frac{p(\cdot)}{t})'}}.
\end{align*}
Note that $t<p_-$ and $p_+/t<r$ imply that
$$\left(\left(\frac{p(\cdot)}{t}\right)'\right)_+<\infty\quad \mbox{and}\quad \left(\frac{p(\cdot)}{t}\right)'>r' .$$
By Theorems \ref{thm:maximal inequality} and Theorem \ref{ad Leb}, we get
\begin{align*}
  Z_1&\lesssim \left\|\sum_{k}\sum_{l} (3\cdot 2^k)^{t}\chi_{I_{k_l}} \right \|_{\frac{p(\cdot)}{t}}^{\frac{1}{t}} \|g\|_{{(\frac{p(\cdot)}{t})'}}^{\frac 1t}
\\&\lesssim\left\|\sum_{k} (3\cdot 2^k)^{t}\chi_{\{\tau_k<\infty\}} \right \|_{\frac{p(\cdot)}{t}}^{\frac{1}{t}}\lesssim \|f\|_{H_{p(\cdot)}}.
\end{align*}
Again, by the condition of the theorem, Corollary \ref{cor:ad for M} and  Theorem \ref{ad Leb}, we have
\begin{align*}
Z_2 &\leq\left\|\sum_k\mu_k^{t}T(a^k)^{t}\chi_{\{\tau_k=\infty\}}\right\|_{\frac{p(\cdot)}{t}}^{\frac{1}{t}}
 \lesssim \left\|\sum_k 2^{kt} \chi_{\{\tau_k<\infty\}}\right\|_{\frac{p(\cdot)}{t}}^{\frac{1}{t}}
\lesssim \|f\|_{H_{p(\cdot)}}.
\end{align*}
Combing the estimates of $Z_1$ and $Z_2$, we  complete the proof.
\end{proof}

\begin{theorem}\label{t3 Leb}
Let $p(\cdot)\in\mathcal{P}(\Omega)$ satisfy \eqref{log} and $1/2<t<\underline p$. If
\begin{equation}\label{e104}
	\frac{1}{p_-}-\frac{1}{p_+} <1,
\end{equation}
then 
\begin{equation}
\left\|\sum_k\mu_k^{t}\sigma_*(a^k)^{t}\chi_{\{\tau_k=\infty\}}\right\|_{\frac{p(\cdot)}{t}}\lesssim \left\|\sum_k 2^{kt} \chi_{\{\tau_k<\infty\}}\right\|_{\frac{p(\cdot)}{t}},
\end{equation}
where $\tau_k$ is the stopping time associated with $a^k$.
\end{theorem}

\begin{proof} We divide this proof into three steps.

\noindent\textbf{Step 1: estimate for $\sigma_*(a)$.}
The sets $\left\{\tau=j\right\}$ are disjoint and there exist disjoint dyadic intervals $I_{j,i}\in \mathcal{F}_j$ such that
$$
\left\{\tau=j\right\}=\bigcup_i I_{j,i}.
$$
Thus
$$
\left\{\tau<\infty\right\}=\bigcup_{j\in \N} \bigcup_i I_{j,i},
$$
where the dyadic intervals $I_{j,i}$ are disjoint. It follows from the definition of the atom that $\int_{I_{j,i}} a\, d\mathbb{P}=0$. For simplicity, instead of $a \chi_{I_{j,i}}$, we will write $b^l$, and so
$$
a= \sum_{j\in \N} \sum_{i} a \chi_{I_{j,i}} = \sum_{l} b^l.
$$
Then the support of $b^l$ is the dyadic interval $I_l$ with length $2^{-K_l}$ $(K_l\in \N)$, the sets $I_l$ are disjoint and $\int_{I_l}b^l\, d\lambda=0$.

It is easy to see that $\widehat b^l(n) =0$ if $n<2^{K_l}$ and in this case $\sigma_{n}a=0$. Therefore we can suppose that $n \geq 2^{K_l}$. If $j\geq K_l$ and $x \not\in I_l$, then $x \dot + 2^{-j-1} \not\in I_l$. Thus for $x \not\in I_l$, $t\in I_l$ and $i\geq j\geq K_l$, we have
$$
b^l(t) D_{2^i}(x \dot + t) = b^l(t) D_{2^i}(x \dot + t \dot + 2^{-j-1})=0.
$$
Since $n \geq 2^{K_l}$ and $2^{N} > n \geq  2^{N-1}$, one has $N-1\geq K_l$. By (\ref{e16}) we obtain for $x\not\in I_l$ that
\begin{align*}
|\sigma_{n}b^l(x)|
&\leq \sum_{j=0}^{N-1} 2^{j-N} \sum_{i=j}^{N-1}
\int_0^1 |a(t)| \Bigl(
(D_{2^i}(x \dot + t) + D_{2^i}(x \dot + t \dot + 2^{-j-1}))\Bigr) \, dt \\
&\lesssim \|\chi_{\{\tau<\infty\}}\|_{p(\cdot)}^{-1} 2^{-K_l} \sum_{j=0}^{K_l-1} 2^{j} \sum_{i=j}^{K_l-1}
\int_{I_l}
\Bigl(D_{2^i}(x \dot + t) + D_{2^i}(x \dot + t \dot + 2^{-j-1})\Bigr) \, dt\\
&\quad + \|\chi_{\{\tau<\infty\}}\|_{p(\cdot)}^{-1} \sum_{j=0}^{K_l-1} 2^{j} \sum_{i=K_l}^{\infty} 2^{-i}
\int_{I_l} \Bigl(D_{2^i}(x \dot + t) + D_{2^i}(x \dot + t \dot + 2^{-j-1})\Bigr) \, dt.
\end{align*}
Observe that the right hand side is independent of $n$.
Using (\ref{e5}), we can verify that for $x \not\in I_l$,
$$
\int_{I_l} D_{2^i}(x \dot + t \dot + 2^{-j-1}) \, dt = 2^{i-K_l}
1_{I_l \dot + [2^{-j-1},2^{-j-1} \dot + 2^{-i})}(x) = 2^{i-K_l} 1_{I_l^{j,i}}(x)
$$
if $j\leq i \leq K_l-1$,
$$
\int_{I_l} D_{2^i}(x \dot + t) \, dt = 2^{i-K_l} 1_{I_l\dot +[2^{-K_l},2^{-i})}(x)
$$
if $i \in \N$ and
$$
\int_{I_l} D_{2^i}(x \dot + t \dot + 2^{-j-1}) \, dt =
1_{I_l\dot +[2^{-j-1},2^{-j-1} \dot + 2^{-K_l})}(x)  =
1_{I_l^{j}}(x)
$$
if $i \geq K_l$. Therefore, for $x \not\in I_l$,
\begin{eqnarray*}
\sigma_*b^l(x)
&\lesssim& \|\chi_{\{\tau<\infty\}}\|_{p(\cdot)}^{-1} \sum_{j=0}^{K_l-1} 2^{j} \sum_{i=K_l}^\infty 2^{-i}
1_{I_l^{j}}(x) \\
&&{} + \|\chi_{\{\tau<\infty\}}\|_{p(\cdot)}^{-1} 2^{-K_l} \sum_{j=0}^{K_l-1} 2^{j} \sum_{i=j}^{K_l-1} 
\Bigl(2^{i-K_l} (I_l\dot +1_{[2^{-K},2^{-i})}(x) + 1_{I_l^{j,i}}(x)) \Bigr)\\
&\lesssim& \|\chi_{\{\tau<\infty\}}\|_{p(\cdot)}^{-1} \sum_{j=0}^{K_l-1} 2^{j-K_l}
1_{I_l^{j}}(x)+ \|\chi_{\{\tau<\infty\}}\|_{p(\cdot)}^{-1} \sum_{j=0}^{K_l-1} 2^{j-K_l} \sum_{i=j}^{K_l-1} 2^{i-K_l} 1_{I_l^{j,i}}(x).
\end{eqnarray*}
Consequently, for $x \in \{\tau=\infty\}$,
\begin{eqnarray}\label{e7}
\sigma_*a(x)
&\lesssim& \|\chi_{\{\tau<\infty\}}\|_{p(\cdot)}^{-1} \left(\sum_{l} \sum_{j=0}^{K_l-1} 2^{j-K_l}
1_{I_l^{j}}(x) \n + \sum_{l} \sum_{j=0}^{K_l-1} 2^{j-K_l} \sum_{i=j}^{K_l-1} 2^{i-K_l} 1_{I_l^{j,i}}(x)\right) \n \\
&=:& \|\chi_{\{\tau<\infty\}}\|_{p(\cdot)}^{-1} \left(A(x)+B(x)\right).
\end{eqnarray}
For the atom $a^k$, we denote $l$, $K_l$ $A$ and $B$ above by $k_l$, $K_{k_l}$, $A_k$ and $B_k$.
Then
\begin{equation*}
\left\|\sum_k\mu_k^{t}\sigma_*(a^k)^{t}\chi_{\{\tau_k=\infty\}}\right\|_{\frac{p(\cdot)}{t}}\lesssim \left\|\sum_k 2^{kt} A_k^t\right\|_{\frac{p(\cdot)}{t}}+\left\|\sum_k 2^{kt} B_k^t\right\|_{\frac{p(\cdot)}{t}}
=:Z_1+Z_2.
\end{equation*}

\smallskip
\noindent\textbf{Step 2: estimate for $Z_1$.}
By Lemma \ref{lem:duality for variabl p}, there is  $g\in L_{(\frac{p(\cdot)}{t})'}$  with norm less than $1$ such that
\begin{align*}
Z_1&\lesssim  \int_\Omega \sum_k 2^{kt}\sum_{l} \sum_{j=0}^{K_{k_l}-1} 2^{(j-K_{k_l}) t }
\chi_{I_{k_l}^{j}}|g|d\mathbb P \\
& \leq  \sum_k 2^{kt}\sum_{{l}} \sum_{j=0}^{K_{k_l}-1} 2^{(j-K_{k_l}) t} \|\chi_{I_{k_l}^{j}}\|_{{\frac{r}{t}}} \| \chi_{I_{k_l}^{j}} g\|_{{(\frac{r}{t})'}}   \\
&\lesssim\sum_k2^{kt}\sum_{{l}} \sum_{j=0}^{K_{k_l}-1} 2^{(j-K_{k_l}) t} \int \chi_{I_{k_l}} \left(\frac{1}{\mathbb P(I_{k_l}^{j})}\int_{I_{k_l}^{j}} \left| g\right|^{(\frac{r}{t})'}\right)^{1/(\frac{r}{t})'}d\mathbb P
\end{align*}
because $\mathbb P(I_{k_l})=\mathbb P(I_{k_l}\dot + 2^{-j-1})=2^{-K_{k_l}}$. Choosing $\max(1,p_+)<r<\infty$ and applying H\"{o}der's inequality again, we conclude
\begin{eqnarray*}
Z_1&\lesssim& \int \sum_k 2^{kt}\sum_{{l}} \chi_{I_{k_l}} \sum_{j=0}^{K_{k_l}-1} 2^{(j-K_{k_l}) t (1/(\frac{r}{t})+1/(\frac{r}{t})')} \left(\frac{1}{\mathbb P(I_{k_l}^{j})}\int_{I_{k_l}^{j}} \left| g\right|^{(\frac{r}{t})'}\right)^{1/(\frac{r}{t })'}d\mathbb P \\
&\lesssim& \int  \sum_k 2^{kt}\sum_{{l}} \chi_{I_{k_l}} \left(\sum_{j=0}^{K_{k_l}-1} 2^{(j-K_{k_l})t} \right) ^{1/(\frac{r}{t})} \\
&&{}
\left(\sum_{j=0}^{K_{k_l}-1} 2^{(j-K_{k_l})t}
\left(\frac{1}{\mathbb P(I_{k_l}^{j})}\int_{I_{k_l}^{j}} \left| g\right|^{(\frac{r}{t})'}\right)\right) ^{1/(\frac{r}{t})'} d\mathbb P.
\end{eqnarray*}
Furthermore,
\begin{eqnarray*}
Z_1&\lesssim& \int  \sum_k 2^{kt}\sum_{l} \chi_{I_{k_l}}
\left(\sum_{j=0}^{K_{k_l}-1} 2^{(j-K_{k_l})t} \left(\frac{1}{\mathbb P(I_{k_l}^{j})}\int_{I_{k_l}^{j}} \left| g\right|^{(\frac{r}{t})'}\right) \right) ^{1/(\frac{r}{t})'}d\mathbb P \\
&\leq& \int \sum_k2^{kt}\sum_{{l}} \chi_{I_{k_l}} \left(U_t(|g|^{(\frac{r}{t})'})\right)^{1/(\frac{r}{t})'} d\mathbb P \\
&\leq& \left \|\sum_k 2^{kt}\sum_{{l}} \chi_{I_{k_l}} \right \|_{{p(\cdot)/t}}  \left\|\left(U_t(|g|^{(\frac{r}{t})'})\right)^{1/(\frac{r}{t })'}\right\|_{{(p(\cdot)/t)'}}.
\end{eqnarray*}
Using Theorem \ref{t15} and Corollary \ref{cor:ad for M}, we get
\begin{align*}
Z_1 &\lesssim \left\|\sum_{k}\sum_{l} 2^{kt}\chi_{I_{k_l}} \right \|_{\frac{p(\cdot)}{t}} \|g\|_{L_{(\frac{p(\cdot)}{t})'}}
\lesssim\left\|\sum_{k} 2^{kt}\chi_{\{\tau_k<\infty\}} \right \|_{\frac{p(\cdot)}{t}},
\end{align*}
whenever
\begin{equation}\label{e107}
	\frac{1}{\left((p(\cdot)/t)'/(r/t)'\right)_-}-\frac{1}{\left((p(\cdot)/t)'/(r/t)'\right)_+} = \frac{r/(r-t)}{p_+/(p_+-t)} - \frac{r/(r-t)}{p_-/(p_--t)} <t.
\end{equation}

Since $r$ can be arbitrarily large, this means that
\[
	\frac{p_+-t}{p_+}-\frac{p_--t}{p_-} <t,
\]
which is exactly \eqref{e104}.

\smallskip
\noindent\textbf{Step 3: estimate for $Z_2$.} In this estimate,  we have to use  $p_->t>1/2$.
We choose again a function $g\in L_{(\frac{p(\cdot)}{\varepsilon})'}$ with $\|g\|_{L_{(\frac{p(\cdot)}{\varepsilon})'}}\leq 1$ such that
\begin{align*}
 \left\|\sum_k\mu_k^{t} B_k^{t}\right\|_{\frac{p(\cdot)}{t}} &= \int_\Omega \sum_k\mu_k^t B_k(x)^{t} \chi_{\{\tau=\infty\}} gd\mathbb P.
\end{align*}
Take $\max(1,p_+)<r<\infty$  large enough such that $2t>r/(r-t)$. Let us apply H\"{o}lder's inequality to obtain
\begin{align*}
Z_2 &\lesssim \int_{\Omega} \sum_k 2^{kt}\sum_{l} \sum_{j=0}^{K_{k_l}-1}\sum_{i=j}^{K_{k_l}-1}  2^{(j-K_{k_l})t} 2^{(i-K_{k_l})t} \chi_{I_{k_l}^{j,i}}|g| \, dx \\
&\lesssim  \sum_k 2^{kt}\sum_{l} \sum_{j=0}^{K_{k_l}-1}\sum_{i=j}^{K_{k_l}-1} 2^{(j-K_{k_l})t} 2^{(i-K_{k_l})t} \|\chi_{I_{k_l}^{j,i}}\|_{{\frac{r}{t}}}
\| \chi_{I_{k_l}^{j,i}} g\|_{{(\frac{r}{t})'}}
\\
&\lesssim  \sum_k 2^{kt} \sum_{l} \sum_{j=0}^{K_{k_l}-1}\sum_{i=j}^{K_{k_l}-1} 2^{(j-K_{k_l})t} 2^{(i-K_{k_l})t} 2^{K_{k_l}-i} \int \chi_{I_{k_l}} \left(\frac{1}{\mathbb P(I_{k_l}^{j,i})}\int_{I_{k_l}^{j,i}} \left| g\right|^{(\frac{r}{t})'}\right)^{1/(\frac{r}{t})'}d\mathbb P.
\end{align*}
Moreover,
\begin{eqnarray*}
Z_2
&\lesssim& \int  \sum_k 2^{kt}\sum_{l} \chi_{I_{k_l}} \sum_{j=0}^{K_{k_l}-1}\sum_{i=j}^{K_{k_l}-1} 2^{(i+j-2K_{k_l})t(1/(\frac{r}{t})+1/(\frac{r}{t})')} 2^{K_{k_l}-i}  \left(\frac{1}{\mathbb P(I_{k_l}^{j,i})}\int_{I_{k_l}^{j,i}} \left| g\right|^{(\frac{r}{t})'}\right)^{1/(\frac{r}{t})'}d\mathbb P\\
&\lesssim& \int  \sum_k 2^{kt}\sum_{l} \chi_{I_{k_l}} \left(\sum_{j=0}^{K_{k_l}-1}\sum_{i=j}^{K_{k_l}-1} 2^{(j-K_{k_l})t} 2^{(i-K_{k_l})t} \right)^{\frac{t}{r}} \left(\sum_{j=0}^{K_{k_l}-1}\sum_{i=j}^{K_{k_l}-1} 2^{(j-K_{k_l})t} 2^{(i-K_{k_l})t}\right.\\
&&{} \left. 2^{(K_{k_l}-i)(\frac{r}{t})'} \frac{1}{\mathbb P(I_{k_l}^{j,i})}\int_{I_{k_l}^{j,i}} \left| g\right|^{(\frac{r}{t})'}\right)^{1/(\frac{r}{t})'}d\mathbb P\\
&\lesssim& \int  \sum_k 2^{kt}\sum_{l} \chi_{I_{k_l}} \left(\sum_{j=0}^{K_{k_l}-1}\sum_{i=j}^{K_{k_l}-1} 
2^{(j-K_{k_l})(2t-r/(r-t))} 2^{(j-i)(r/(r-t)-t)} \frac{1}{\mathbb P(I_{k_l}^{j,i})}\int_{I_{k_l}^{j,i}} \left| g\right|^{(\frac{r}{t})'}\right)^{1/(\frac{r}{t})'}d\mathbb P.
\end{eqnarray*}
Note that $((p(\cdot)/t)')_+<\infty$ and $(\frac{r}{t})'<(p(\cdot)/t)'$. Taking into account the definition of the maximal operator $V$, Theorem \ref{t17} and Corollary \ref{cor:ad for M}, we obtain
\begin{eqnarray*}
Z_2
&\leq& \int \sum_k 2^{kt} \sum_{l} \chi_{I_{k_l}} \left(V_{2t-r/(r-t),r/(r-t)-t}(|g|^{(\frac{r}{t})'})\right)^{1/(\frac{r}{t})'} d\mathbb P \\
&\leq& \left \| \sum_k 2^{kt}\sum_{l} \chi_{I_{k_l}} \right \|_{{p(\cdot)/t}}  \left\|\left(V_{2t-r/(r-t),r/(r-t)-t}(|g|^{(\frac{r}{t})'})\right)^{1/(\frac{r}{t})'}\right\|_{{(p(\cdot)/t)'}} \\
&\lesssim& \left \| \sum_k 2^{kt}\chi_{\{\tau<\infty\}} \right \|_{{\frac{p(\cdot)}t}},
\end{eqnarray*}
whenever \eqref{e107}
and \eqref{e104} hold.
Combining the estimates of $Z_1$ and $Z_2$, we finish the proof.
\end{proof}

\begin{remark}
If $1 \leq  p_- <\infty$, then \eqref{e104} holds for all $p_+$. If $1/2<p_-<1$, then $p_+$ can be chosen such that $p_+>1$ and \eqref{e104} holds. 
\end{remark}

We immediately get the boundedness of $\sigma_*$ from $H_{p(\cdot)}$ to $L_{p(\cdot)}$ by the above theorems. For the constant $p=1$ it is due to Fujii \cite{fu} (see also Schipp and Simon \cite{schs}). For other constant $p$'s with $1/2<p\leq \infty$, the theorem was proved by the third author in \cite{wces3}.

\begin{theorem} \label{t4 Leb}
Let $p(\cdot)\in\mathcal{P}(\Omega)$ satisfy conditions \eqref{log} and \eqref{e104}. If $1/2<p_-<\infty$, then
$$
\left\|\sigma_*f\right\|_{{p(\cdot)}}\lesssim \left\|f\right\|_{H_{p(\cdot)}},\quad f\in H_{p(\cdot)}.
$$
\end{theorem}

If $p(\cdot)=p$ and $p\leq 1/2$, then the theorem is not true anymore (see Simon and Weisz \cite{wgyenge}, Simon \cite{si7} and G{\'a}t and Goginava \cite{gago}). If \eqref{e104} does not hold, then a counterexample can be found in Theorem \ref{t101} below. This theorem implies the next consequences about the convergence of $\sigma_nf$. First we consider the almost everywhere convergence.

\begin{corollary} \label{c20}
Let $p(\cdot)\in\mathcal{P}(\Omega)$ satisfy conditions \eqref{log} and \eqref{e104}. If $1/2<p_-<\infty$ and $f\in H_{p(\cdot)}$, then $\sigma_{n} f$ converges almost everywhere on $[0,1)$.
\end{corollary}

\begin{proof}
Fix $f\in H_{p(\cdot)}$ and set
$$
g_N(x):= \sup_{n,k\geq N} |\sigma_nf(x)-\sigma_kf(x)|, \qquad g(x):= \lim_{N \to \infty} g_N(x) \qquad (x\in [0,1)).
$$
It is sufficient to show that $g=0$ almost everywhere.

Observe that $f_m$ is a Walsh polynomial,
$$
\|f-f_n\|_{H_{p(\cdot)}} \to 0 \qquad \mbox{and} \qquad \sigma_nf_m\to f_m
$$
as $n \to \infty$. Since
$$
|\sigma_nf(x)-\sigma_kf(x)| \leq 2\sigma_*f(x)
$$
and
$$
g_N(x) \leq \sup_{n\geq N} |\sigma_n(f-f_m)(x)| +
\sup_{n,k\geq N} |\sigma_nf_m(x)-\sigma_kf_m(x)| + \sup_{k\geq N}|\sigma_k(f_m-f)(x)|,
$$
we conclude that
$$
g(x) \leq 4 \sigma_*(f-f_m)(x)
$$
for all $m \in \N$ and $x\in I$. Henceforth, by Theorem \ref{t4 Leb},
$$
\left\|g\right\|_{{p(\cdot)}}
\leq 4 \left\|\sigma_*(f-f_m)\right\|_{{p(\cdot)}} \lesssim \left\|f-f_m\right\|_{H_{p(\cdot)}} \to 0
$$
as $m\to \infty$. Hence $g=0$ almost everywhere.
\end{proof}

For an integrable function $f$ belonging to the Hardy spaces, the limit of $\sigma_nf$ is exactly the function $f$.
Let $I\in \F_k$ be an atom of $\F_k$. The restriction of a martingale $f$ to the atom $I$ is defined by
$$
f\chi_I:=(\mathbb E_nf \chi_I,n\geq k).
$$

\begin{corollary}\label{c50}
Let $p(\cdot)\in\mathcal{P}(\Omega)$ satisfy conditions \eqref{log} and \eqref{e104}, $1/2<p_-<\infty$ and $f\in H_{p(\cdot)}$. If there exists a dyadic interval $I$ such that the restriction $f \chi_I \in L_1(I)$, then
$$
\lim_{n\to\infty}\sigma_{n} f(x)=f(x) \qquad \mbox{for a.e. $x\in I$.}
$$
\end{corollary}

\begin{proof}
For $f\in H_{p(\cdot)}$, let
$$
g(x):= \limsup_{n \to \infty} |\sigma_nf(x)-f(x)| \qquad (x\in I).
$$
There exists $k\in \N$ such that $I$ is an atom of $\F_k$. Observe that
\begin{eqnarray*}
g(x) &\leq& \limsup_{n \to \infty} |\sigma_n(f-f_m)(x)| +
\limsup_{n \to \infty} |\sigma_nf_m(x)-f_m(x)| + |f_m(x)-f(x)| \\
&\leq& \sigma_*(f-f_m)(x) + |f(x)-f_m(x)|
\end{eqnarray*}
for all $m \in \N$ and $x\in I$. Theorem \ref{t4} implies
\begin{eqnarray*}
\left\|g\right\|_{{p(\cdot)}}
&\leq& \left\|\sigma^*(f-f_m)\chi_I\right\|_{{p(\cdot)}} + \left\|(f-f_m)\chi_I\right\|_{{p(\cdot)}} \\
&\leq& \left\|\sigma^*(f-f_m)\right\|_{{p(\cdot)}} + \left\|\sup_{n\geq k} \left|\mathbb E_n(f-f_m)\chi_I\right|\right\|_{{p(\cdot)}} \\
&\leq& 2\left\|f-f_m\right\|_{H_{p(\cdot)}} \to 0
\end{eqnarray*}
as $m\to \infty$. Hence $g=0$ almost everywhere.
\end{proof}

Since $f\in H_{p(\cdot)}$ with $1\leq p_-<\infty$ implies that $f$ is integrable, we obtain the next corollary.

\begin{corollary}\label{c51}
Let $p(\cdot)\in\mathcal{P}(\Omega)$ satisfy conditions \eqref{log} and \eqref{e104}, $1\leq p_-<\infty$
and $f\in H_{p(\cdot)}$. Then
$$
\lim_{n\to\infty}\sigma_{n} f(x)=f(x) \qquad \mbox{for a.e. $x\in [0,1)$.}
$$
\end{corollary}

In the next subsection in Corollary \ref{c1}, we will show the almost everywhere convergence for all integrable functions.

For the norm convergence, we can prove the following consequences similarly.

\begin{corollary} \label{c30}
Let $p(\cdot)\in\mathcal{P}(\Omega)$ satisfy conditions \eqref{log} and \eqref{e104}. If $1/2<p_-<\infty$ and $f\in H_{p(\cdot)}$, then $\sigma_{n} f$ converges in the $L_{p(\cdot)}$-norm.
\end{corollary}

\begin{corollary}\label{c52}
Let $p(\cdot)\in\mathcal{P}(\Omega)$ satisfy conditions \eqref{log} and \eqref{e104}, $1/2<p_-<\infty$
and $f\in H_{p(\cdot)}$. If there exists a dyadic interval $I$ such that the restriction $f \chi_I \in L_1(I)$, then
$$
\lim_{n\to\infty}\sigma_{n} f=f \qquad \mbox{in the $L_{p(\cdot)}(I)$-norm.}
$$
\end{corollary}

\begin{corollary}\label{c53}
Let $p(\cdot)\in\mathcal{P}(\Omega)$ satisfy conditions \eqref{log} and \eqref{e104}, $1\leq p_-<\infty$
and $f\in H_{p(\cdot)}$. Then
$$
\lim_{n\to\infty}\sigma_{n} f=f  \qquad \mbox{in the $L_{p(\cdot)}$-norm.}
$$
\end{corollary}

Note that $H_{p(\cdot)}$ is equivalent to $L_{p(\cdot)}$ if $1<p_-<\infty$.
Considering only $\sigma_{2^n}f$, we do not need the restriction $1/2<p_-$ about $p(\cdot)\in\mathcal{P}(\Omega)$.

\begin{theorem} \label{t9}
Let $p(\cdot)\in\mathcal{P}(\Omega)$ satisfy condition \eqref{log} and \eqref{e104} and $0<t<\underline p$.
Then
\begin{equation}
\left\|\sum_k\mu_k^{t}\sup_{n\in\mathbb N}|\sigma_{2^n}(a^k)|^{t}\chi_{\{\tau_k=\infty\}}\right\|_{\frac{p(\cdot)}{t}}\lesssim \left\|\sum_k 2^{kt} \chi_{\{\tau_k<\infty\}}\right\|_{\frac{p(\cdot)}{t}},
\end{equation}
where $\tau_k$ is the stopping time associated with $a^k$.
\end{theorem}

\begin{proof}
Taking into account (\ref{e17}) and the proof of Theorem \ref{t3 Leb}, we can suppose that $2^n\geq 2^{K_l}$. For $x\not\in I_l$, we obtain  that
\begin{eqnarray*}
\left|\sigma_{2^n}b^l(x)\right|
&\lesssim& \sum_{j=0}^{n} 2^{j-n} \int_0^1 |a(t)| D_{2^n}(x \dot + t \dot + 2^{-j-1}) \, dt \\
&\lesssim& \|\chi_{\{\tau<\infty\}}\|_{p(\cdot)}^{-1} \sum_{j=0}^{K_l-1} 2^{j-K_l}
\int_{I_l} D_{2^n}(x \dot + t \dot + 2^{-j-1}) \, dt \\
&\lesssim& \|\chi_{\{\tau<\infty\}}\|_{p(\cdot)}^{-1} \sum_{j=0}^{K_l-1} 2^{j-K_l}
\chi_{I_l^{j}}(x).
\end{eqnarray*}
If $x \in \{\tau=\infty\}$, then
$$\sup_{n\in \N}|\sigma_{2^n}a|
\lesssim \|\chi_{\{\tau<\infty\}}\|_{p(\cdot)}^{-1} \sum_{l} \sum_{j=0}^{K_l-1} 2^{j-K_l}
\chi_{I_l^{j}}(x) = \|\chi_{\{\tau<\infty\}}\|_{p(\cdot)}^{-1} A(x)
$$
and the proof can be finished as in Theorem \ref{t3 Leb}.
\end{proof}

We deduce the next result from this and Theorem \ref{t1 Leb}.

\begin{theorem} \label{t5 Leb}
If $p(\cdot)\in\mathcal{P}(\Omega)$ satisfies conditions \eqref{log} and \eqref{e104}, then
$$\left\|\sup_{n\in \mathbb N}|\sigma_{2^n}f|\right\|_{{p(\cdot)}}\lesssim \left\|f\right\|_{H_{p(\cdot)}},\quad f\in H_{p(\cdot)}.$$
\end{theorem}

Neither Theorem \ref{t5 Leb} nor \ref{t4 Leb} hold if \eqref{e104} is not satisfied. More exactly, we show

\begin{theorem}\label{t101}
Let $p(\cdot)\in\mathcal{P}(\Omega)$ satisfy condition \eqref{log}. If 
\begin{equation}\label{e109}
	\frac{1}{p_-(I_{0,n-1})}-\frac{1}{p_+(I_{0,n}^{0})} >1
\end{equation}
for all $n \in \N$, then $\sigma_*$ as well as $\sup_{n\in \mathbb N}|\sigma_{2^n}|$ are not bounded from $H_{p(\cdot)}$ to $L_{p(\cdot)}$.
\end{theorem}

\begin{proof}
	Let 
	$$
	a_{n-1}(t)=2^{(n-1)/p_-(I_{0,n-1})} (\chi_{I_{0,n}}-\chi_{I_{1,n}})
	$$
	and $x \notin I_{0,n-1}$. It is easy to see that $a_{n-1}$ is an atom for all $n \geq 1$, and so $\left\|a_{n-1}\right\|_{H_{p(\cdot)}} \leq 1$. As in Theorem \ref{t9},
	\begin{align*}
		\left|\sigma_{2^n}a_{n-1}(x)\right| &= \left|\sum_{j=0}^{n} 2^{j-n} \int_0^1 a_{n-1}(t) D_{2^n}(x \dot + t \dot + 2^{-j-1}) \, dt \right|\\
		&= \left|\sum_{j=0}^{n-2} 2^{j-n} \chi_{I_{0,n-1}^{j}}(x)\int_{I_{0,n-1}} a_{n-1}(t) D_{2^n}(x \dot + t \dot + 2^{-j-1}) \, dt \right|\\
		&= \sum_{j=0}^{n-2} 2^{j-n} \chi_{I_{0,n-1}^{j}}(x)\left|\int_{I_{0,n-1}} a_{n-1}(t) D_{2^n}(x \dot + t \dot + 2^{-j-1}) \, dt \right|.
	\end{align*}
	We choose $j=0$ and the left half of $I_{0,n-1}^{0}$:
	\begin{align*}
		\left|\sigma_{2^n}a_{n-1}(x)\right|
		& \geq \chi_{I_{0,n}^{0}}(x) 2^{-n} 2^{(n-1)/p_-(I_{0,n-1})}.
	\end{align*}
	Then
	\begin{align*}
	\int_{\Omega}\sup_{k\in \mathbb N}|\sigma_{2^k}a_{n-1}(x)|^{p(x)} \,dx 
	&\geq \int_{\Omega}|\sigma_{2^n}a_{n-1}(x)|^{p(x)} \,dx  \\
	&\geq \int_{I_{0,n}^{0}} 2^{-np(x)} 2^{(n-1)p(x)/p_-(I_{0,n-1})} \, dx \\
	&\geq C \int_{I_{0,n}^{0}} 2^{np_+(I_{0,n}^{0})(1/p_-(I_{0,n-1})-1)} \, dx \\
	&= C 2^{np_+(I_{0,n}^{0})(1/p_-(I_{0,n-1})-1)} 2^{-n}
	\end{align*}
	which tends to infinity as $n\to \infty$ if \eqref{e109} holds.
\end{proof}

The following corollaries can be shown as above.

\begin{corollary} \label{c60}
Let $p(\cdot)\in\mathcal{P}(\Omega)$ satisfy conditions \eqref{log} and \eqref{e104}. If $f\in H_{p(\cdot)}$, then $\sigma_{2^n} f$ converges almost everywhere on $[0,1)$.
\end{corollary}

\begin{corollary}\label{c70}
Let $p(\cdot)\in\mathcal{P}(\Omega)$ satisfy conditions \eqref{log} and \eqref{e104}. If $f\in H_{p(\cdot)}$ and there exists a dyadic interval $I$ such that the restriction $f \chi_I \in L_1(I)$, then
$$
\lim_{n\to\infty}\sigma_{2^n} f(x)=f(x) \qquad \mbox{for a.e. $x\in I$.}
$$
\end{corollary}

\begin{corollary} \label{c40}
Let $p(\cdot)\in\mathcal{P}(\Omega)$ satisfy conditions \eqref{log} and \eqref{e104}. If $f\in H_{p(\cdot)}$, then $\sigma_{2^n} f$ converges in the $L_{p(\cdot)}$-norm.
\end{corollary}

\begin{corollary}\label{c54}
Let $p(\cdot)\in\mathcal{P}(\Omega)$ satisfy conditions \eqref{log} and \eqref{e104}. If $f\in H_{p(\cdot)}$ and there exists a dyadic interval $I$ such that the restriction $f \chi_I \in L_1(I)$, then
$$
\lim_{n\to\infty}\sigma_{2^n} f=f \qquad \mbox{in the $L_{p(\cdot)}(I)$-norm.}
$$
\end{corollary}

\subsection{The maximal Fej{\'e}r operator on $H_{p(\cdot),q}$}
In this subsection, we extend the main results in  Subsection \ref{7.2} to the variable Hardy-Lorentz space setting. Our method is new even in the classical case (\cite{wces3}).

\begin{theorem} \label{t1}
Let $p(\cdot)\in\mathcal{P}(\Omega)$ satisfy condition \eqref{log}, $0<q\leq \infty$ and $1<r\leq \infty$ with $p_+<r$. Suppose that $T:H_r^s \rightarrow L_r$ is a bounded $\sigma$-sublinear operator and
\begin{equation}\label{e2}
\left\| |Ta|^{\beta} \chi_{\{\tau=\infty\}} \right\|_{p(\cdot)} \leq C \left\|\chi_{\{\tau<\infty\}}\right\|_{p(\cdot)}^{1-\beta }
\end{equation}
for some $0<\beta<1$ and all $(1,p(\cdot),\infty)$-atoms $a$, where $\tau$ is the stopping time associated with $a$. Then we have
$$\|Tf\|_{L_{p(\cdot),q}}\lesssim \|f\|_{H_{p(\cdot),q}^s},\quad f\in H_{p(\cdot),q}^s.$$
\end{theorem}

\begin{proof}
Let $r=\infty$. We decompose again the martingale $f\in H_{p(\cdot),q}^s$ into the sum of $F_1$ and $F_2$, $f=F_1+F_2$ as in the proof of Theorem \ref{theorem of boundedness 1}. Then (\ref{e2}) holds and
\begin{eqnarray*}
\|TF_1\|_\infty
&\leq& \sum_{k=-\infty}^{k_0-1} \mu_k \|Ta^k\|_\infty
\leq  \sum_{k=-\infty}^{k_0-1} \mu_k \|s(a^k)\|_\infty \\
&\leq& \sum_{k=-\infty}^{k_0-1} \mu_k \|\chi_{\{\tau_k<\infty\}}\|_{p(\cdot)}^{-1}\leq 3 \cdot 2^{k_0}.
\end{eqnarray*}
Thus
$$
2^{k_0}\|\chi_{\{Tf>6 \cdot 2^{k_0}\}}\|_{p(\cdot)} \leq 2^{k_0}\|\chi_{\{TF_2> 3\cdot2^{k_0}\}}\|_{p(\cdot)},
$$
so we have to consider
\begin{equation}\label{e14}
|TF_2|
\leq \sum_{k=k_0}^{\infty} \mu_k |Ta^k| \chi_{\{\tau_k<\infty\}} + \sum_{k=k_0}^{\infty} \mu_k |Ta^k| \chi_{\{\tau_k=\infty\}}.
\end{equation}
For the first term, we obtain similarly to Step 2 that
$$
\left\|\chi_{\{\sum_{k=k_0}^{\infty} \mu_k |Ta^k| \chi_{\{\tau_k<\infty\}}> 3 \cdot 2^{k_0-1}\}} \right\|_{p(\cdot)}
\leq\| \sum_{k=k_0}^\infty\chi_{\{\tau_k<\infty\}}\|_{p(\cdot)}
$$
and
$$
\left\|\sum_{k=k_0}^{\infty} \mu_k |Ta^k| \chi_{\{\tau_k<\infty\}}\right\|_{L_{p(\cdot),q}}\lesssim \Bigg(\sum_{k=-\infty}^\infty\mu_k^q\Bigg)^{1/q}\lesssim \|f\|_{H_{p(\cdot),q}^s} \qquad (f\in H_{p(\cdot),q}^s)
$$
if $q<\infty$. In the last inequality we have used Theorem \ref{theorem of atomic decomposition pq}. If $q=\infty$, then similarly to (\ref{estimate for F2}),
\begin{eqnarray*}
\left\|\chi_{\{\sum_{k=k_0}^{\infty} \mu_k |Ta^k| \chi_{\{\tau_k<\infty\}}> 3 \cdot 2^{k_0-1}\}} \right\|_{p(\cdot)}
&\leq& \| \sum_{k=k_0}^\infty\chi_{\{\tau_k<\infty\}}\|_{p(\cdot)}\\
&\leq & \left(\sum_{k=k_0}^\infty 2^{-k{\varepsilon}}2^{k{\varepsilon}}\|\chi_{\{\tau_k<\infty\}}\|_{p(\cdot)}
^{{\varepsilon}}\right)^{1/{\varepsilon}} \\
&\leq & \left(\sum_{k=k_0}^\infty 2^{-k{\varepsilon}}\right)^{1/{\varepsilon}} \sup_{k} \mu_k\\
&\lesssim& 2^{-k_0} \sup_{k} \mu_k.
\end{eqnarray*}
Hence
$$
\left\|\sum_{k=k_0}^{\infty} \mu_k |Ta^k| \chi_{\{\tau_k<\infty\}}\right\|_{L_{p(\cdot),\infty}}\lesssim \sup_{k} \mu_k \lesssim \|f\|_{H_{p(\cdot),\infty}^s} \qquad (f\in H_{p(\cdot),\infty}^s).
$$

To investigate the second term of (\ref{e14}), let $\beta<\delta<1$ and $\epsilon<\min\{\underline{p},q\}$.
Observe that inequality (\ref{e2}) is equivalent to
$$
\left\| |\sigma_*a|^{\beta \epsilon} \chi_{\{\tau=\infty\}} \right\|_{p(\cdot)/\epsilon} \leq C \left\|\chi_{\{\tau<\infty\}}\right\|_{p(\cdot)/\epsilon} \left\|\chi_{\{\tau<\infty\}}\right\|_{p(\cdot)}^{-\beta \epsilon}.
$$
From this it follows that
\begin{align}\label{e15}
\left\| \chi_{\{\sum_{k=k_0}^{\infty} \mu_k |Ta^k| \chi_{\{\tau_k=\infty\}}>3 \cdot 2^{k_0-1} \}} \right\|_{p(\cdot)}
&\leq \left\|\frac{\sum_{k=k_0}^{\infty} \mu_k^\beta |Ta^k|^\beta \chi_{\{\tau_k=\infty\}}}{3^\beta 2^{\beta(k_0-1)}} \right\|_{{p(\cdot)}} \n\\
&\lesssim 2^{-\beta k_0} \left\|\sum_{k=k_0}^{\infty} \mu_k^{\beta \epsilon} |Ta^k|^{\beta \epsilon} \chi_{\{\tau_k=\infty\}} \right\|_{{p(\cdot)/\epsilon}}^{1/\epsilon} \n\\
&\lesssim 2^{-\beta k_0} \left(\sum_{k=k_0}^{\infty} \mu_k^{\beta \epsilon} \left\||Ta^k|^{\beta \epsilon} \chi_{\{\tau_k=\infty\}}\right\|_{{p(\cdot)}/\epsilon} \right) ^{1/\epsilon} \n\\
&\lesssim 2^{-\beta k_0} \left(\sum_{k=k_0}^{\infty} 2^{k\beta \epsilon} \|\chi_{\{\tau_k<\infty\}}\|_{{p(\cdot)}/\epsilon} \right) ^{1/\epsilon} \n\\
&\leq 2^{-\beta k_0} \left(\sum_{k=k_0}^{\infty} 2^{k(\beta-\delta)\epsilon} 2^{k\delta\epsilon} \|\chi_{\{\tau_k<\infty\}}\|_{p(\cdot)}^{\epsilon} \right) ^{1/\epsilon}.
\end{align}
If $q<\infty$, let us again use H\"{o}lder's inequality with $\frac{q-\varepsilon}{q}+\frac{\varepsilon}{q}=1$:
\begin{eqnarray*}
\lefteqn{\left\| \chi_{\{\sum_{k=k_0}^{\infty} \mu_k |Ta^k| \chi_{\{\tau_k=\infty\}}>3 \cdot 2^{k_0-1} \}} \right\|_{p(\cdot)}  } \n\\
&\lesssim& 2^{-\beta k_0} \left(\sum_{k=k_0}^{\infty} 2^{k(\beta-\delta)\epsilon\frac{q}{q-\epsilon}}\right)^{\frac{q-\epsilon}{\epsilon q}}
\left(\sum_{k=k_0}^{\infty} 2^{k\delta q}\|\chi_{\{\tau_k<\infty\}}\|_{p(\cdot)}^{q}\right)^{1/q}\\
&\lesssim& 2^{-k_0\delta}  \left(\sum_{k=k_0}^{\infty} 2^{k\delta q}\|\chi_{\{\tau_k<\infty\}}\|_{p(\cdot)}^{q}\right)^{1/q}.
\end{eqnarray*}
By changing the order of the sums, we obtain
 \begin{align*}
 \sum_{k_0=-\infty}^\infty 2^{k_0q}\left\| \chi_{\{\sum_{k=k_0}^{\infty} \mu_k |Ta^k| \chi_{\{\tau_k=\infty\}}>3 \cdot 2^{k_0-1} \}} \right\|_{p(\cdot)}^{q}
&\lesssim  \sum_{k_0=-\infty}^\infty 2^{k_0(1-\delta)q}
 \sum_{k=k_0}^{\infty}2^{k\delta q}\|\chi_{\{\tau_k<\infty\}}\|_{p(\cdot)}^{q}
 \\&=\sum_{k=-\infty}^\infty 2^{k\delta q}\|\chi_{\{\tau_k<\infty\}}\|_{p(\cdot)}^{q} \sum_{k_0=-\infty}^{k} 2^{k_0(1-\delta)q}
 \\&\lesssim \sum_{k=-\infty}^\infty 2^{kq}\|\chi_{\{\tau_k<\infty\}}\|_{p(\cdot)}^{q}.
 \end{align*}
This implies that
$$
\left\|\sum_{k=k_0}^{\infty} \mu_k |Ta^k| \chi_{\{\tau_k=\infty\}}\right\|_{L_{p(\cdot),q}}
\lesssim \Bigg(\sum_{k=-\infty}^\infty\mu_k^q\Bigg)^{1/q} \lesssim \|f\|_{H_{p(\cdot),q}^s}.
$$

If $q=\infty$, we use (\ref{e15}) with $\delta=1$ to obtain
\begin{align*}
\left\| \chi_{\{\sum_{k=k_0}^{\infty} \mu_k |Ta^k| \chi_{\{\tau_k=\infty\}}>3 \cdot 2^{k_0-1} \}} \right\|_{p(\cdot)}
&\leq 2^{-\beta k_0} \left(\sum_{k=k_0}^{\infty} 2^{k(\beta-1)\epsilon} 2^{k\epsilon} \|\chi_{\{\tau_k<\infty\}}\|_{p(\cdot)}^{\epsilon} \right) ^{1/\epsilon} \\
&\leq 2^{-\beta k_0} \left(\sum_{k=k_0}^{\infty} 2^{k(\beta-1)\epsilon} \right) ^{1/\epsilon} \sup_{k} \mu_k\\
&\lesssim 2^{-k_0} \sup_{k} \mu_k
\end{align*}
and so
$$
\left\|\sum_{k=k_0}^{\infty} \mu_k |Ta^k| \chi_{\{\tau_k=\infty\}}\right\|_{L_{p(\cdot),\infty}}\lesssim \sup_{k} \mu_k \lesssim \|f\|_{H_{p(\cdot),\infty}^s} \qquad (f\in H_{p(\cdot),\infty}^s).
$$

Now let $r<\infty$. Similarly to the proof of Theorem \ref{theorem of boundedness 1}, we obtain
$$
\|T(F_1)\|_{L_{p(\cdot),q}}\lesssim \|f\|_{H_{p(\cdot),q}^s}.
$$
On the other hand, the inequality
$$
\|T(F_2)\|_{L_{p(\cdot),q}}\lesssim \|f\|_{H_{p(\cdot),q}^s}
$$
holds in the same way as above. This completes the proof.
\end{proof}

The $\sigma$-sublinearity cannot be omitted in general (see  Bownik, Li, Yang and Zhou \cite{{Bownik2005},Bownik2010,Yang2009}). However, if $T$ is a linear operator and $q<\infty$, then $T$ can be uniquely extended.

\begin{theorem} \label{t7}
Let $p(\cdot)\in\mathcal{P}(\Omega)$ satisfy condition \eqref{log}, $0<q<\infty$ and $1<r<\infty$ with $p_+<r$. Suppose that $T:H_r^s \rightarrow L_r$ is a bounded linear operator and
$$
\left\| |Ta|^{\beta} \chi_{\{\tau=\infty\}} \right\|_{p(\cdot)} \leq C \left\|\chi_{\{\tau<\infty\}}\right\|_{p(\cdot)}^{1-\beta }
$$
for some $0<\beta<1$ and all $(1,p(\cdot),\infty)$-atoms $a$, where $\tau$ is the stopping time associated with $a$. Then
$$
\|Tf\|_{L_{p(\cdot),q}}\lesssim \|f\|_{H_{p(\cdot),q}^s},\quad f\in H_{r}^s \cap H_{p(\cdot),q}^s
$$
and $T$ can be uniquely extended to a bounded linear operator from $H_{p(\cdot),q}^s$ to $L_{p(\cdot),q}^s$.
The theorem holds for $q=\infty$ as well if we change $H^s_{p(\cdot),\infty}$ by $\mathscr H^s_{p(\cdot),\infty}$.
\end{theorem}

\begin{proof}
By Remark \ref{remark of desenty}, the atomic decomposition converges in the $H_{p(\cdot),q}^s$-norm. Similarly, writing $p(\cdot)=q=r$, the atomic decomposition converges to $f$ in the $H_{r}^s$-norm if $f\in H_{r}^s$. Moreover, $H_{r}^s \cap H_{p(\cdot),q}^s$ is dense in $H_{p(\cdot),q}^s$. Let us define $F_1$ and $F_2$ again as in the proof of Theorem \ref{theorem of boundedness 1}. Then for $f\in H_{r}^s \cap H_{p(\cdot),q}^s$,
$$
F_1=\sum_{k=-\infty}^{k_0-1} \mu_k a^k\quad \mbox{in the $H_{r}^s$-norm}.
$$
Since $T:H_r^s \rightarrow L_r$ is a bounded linear operator, we have
$$
TF_1=\sum_{k=-\infty}^{k_0-1} \mu_k Ta^k\quad \mbox{in the $L_{r}$-norm}
$$
and
$$
\left| TF_1\right|=\sum_{k=-\infty}^{k_0-1} \mu_k \left| Ta^k\right|.
$$
The analogous inequality holds for $TF_2$. The proof can be finished as in Theorem \ref{t1}.
\end{proof}

The next two theorems can be shown similarly.

\begin{theorem} \label{t2}
Let $p(\cdot)\in\mathcal{P}(\Omega)$ satisfy condition \eqref{log}, $0<q\leq \infty$ and $1<r\leq \infty$ with $p_+<r$.
Suppose that $T:L_r \rightarrow L_r$ is a bounded $\sigma$-sublinear operator and
$$
\left\| |Ta|^{\beta} \chi_{\{\tau=\infty\}} \right\|_{p(\cdot)} \leq C \left\|\chi_{\{\tau<\infty\}}\right\|_{p(\cdot)}^{1-\beta }
$$
for some $0<\beta<1$ and all $(3,p(\cdot),\infty)$-atoms $a$, where $\tau$ is the stopping time associated with $a$. Then
$$\|Tf\|_{L_{p(\cdot),q}}\lesssim \|f\|_{P_{p(\cdot),q}},\quad f\in P_{p(\cdot),q}.$$
\end{theorem}

\begin{theorem} \label{t8}
Let $p(\cdot)\in\mathcal{P}(\Omega)$ satisfy condition \eqref{log}, $0<q<\infty$ and $1<r<\infty$ with $p_+<r$. Suppose that $T:L_r \rightarrow L_r$ is a bounded linear operator and
$$
\left\| |Ta|^{\beta} \chi_{\{\tau=\infty\}} \right\|_{p(\cdot)} \leq C \left\|\chi_{\{\tau<\infty\}}\right\|_{p(\cdot)}^{1-\beta }
$$
for some $0<\beta<1$ and all $(3,p(\cdot),\infty)$-atoms $a$, where $\tau$ is the stopping time associated with $a$. Then
$$
\|Tf\|_{L_{p(\cdot),q}}\lesssim \|f\|_{P_{p(\cdot),q}},\quad f\in L_{r} \cap P_{p(\cdot),q}
$$
and $T$ can be uniquely extended to a bounded linear operator from $P_{p(\cdot),q}$ to $L_{p(\cdot),q}$.
The theorem holds for $q=\infty$ as well if we change $P_{p(\cdot),\infty}$ by $\mathscr P_{p(\cdot),\infty}$.
\end{theorem}

For linear operators $T_n$ let the maximal operators be defined by
$$
T_*f:=\sup_{n\in \N} \left|T_nf\right|, \qquad T_{N,*}f:=\sup_{n\leq N} \left|T_nf\right|.
$$

\begin{theorem} \label{t6}
Let $p(\cdot)\in\mathcal{P}(\Omega)$ satisfy condition \eqref{log}, $0<q<\infty$ and $1<r\leq \infty$ with $p_+<r$. Suppose that $T_n:L_1\to L_1$ is a bounded linear operator for each $n\in \N$ and
\begin{equation}\label{e10}
T_kf_n=T_kf \qquad \mbox{for} \qquad 0\leq k\leq 2^n.
\end{equation}
Suppose that $T_*:L_r \rightarrow L_r$ is bounded and
$$
\left\| |T_*a|^{\beta} \chi_{\{\tau=\infty\}} \right\|_{p(\cdot)} \leq C \left\|\chi_{\{\tau<\infty\}}\right\|_{p(\cdot)}^{1-\beta }
$$
for some $0<\beta<1$ and all $(3,p(\cdot),\infty)$-atoms $a$, where $\tau$ is the stopping time associated with $a$. Then
$$
\|T_*f\|_{L_{p(\cdot),q}}\lesssim \|f\|_{H^S_{p(\cdot),q}},\quad f\in H^S_{p(\cdot),q}.
$$
The theorem holds for $q=\infty$ as well if we change $H^S_{p(\cdot),\infty}$ by $\mathscr H^S_{p(\cdot),\infty}$.
\end{theorem}

\begin{proof}
It is easy to see that the atomic decomposition of Theorem \ref{theorem of atomic decomposition pq} converges in the $L_1$-norm,
\begin{equation*}
\sum_{k\in \mathbb{Z}} \mu_k a^k=f \quad\mbox{in the $L_1$-norm}
\end{equation*}
if $f\in H^S_1$. Thus, in this case,
$$
T_nf = \sum_{k\in \mathbb{Z}} \mu_k T_na^k
$$
and
$$
T_*f \leq \sum_{k\in \mathbb{Z}}|\mu_k| T_*a^k.
$$
Observe that for $f\in H^S_{p(\cdot),q}$, $f_n\in H^S_1$ because $f_n$ is integrable $(n\in \N)$. Theorem \ref{t2} implies that
$$
\|T_*f\|_{L_{p(\cdot),q}}\lesssim \|f\|_{H^S_{p(\cdot),q}} \qquad (f\in H^S_{1}).
$$
Hence
$$
\|T_{2^n,*}f_n\|_{L_{p(\cdot),q}} \leq \|T_*f_n\|_{L_{p(\cdot),q}}\lesssim \|f_n\|_{H^S_{p(\cdot),q}}
\qquad (f\in H^S_{p(\cdot),q}).
$$
Since
$$
\lim_{n\to\infty}f_n=f \qquad  \mbox{in the $H^S_{p(\cdot),q}$-norm}
$$
because of the dominated convergence theorem Lemma \ref{lemma of control},
$T_{2^n,*}f_n$ converges in the $L_{p(\cdot),q}$-norm, say
$$
\lim_{n\to\infty}T_{2^n,*}f_n=Vf \qquad  \mbox{in the $L_{p(\cdot),q}$-norm}.
$$
However,
$$
T_{2^n,*}f_n=T_{2^n,*}f \qquad (f\in H^S_{p(\cdot),q})
$$
by the condition of the theorem. Obviously,
$$
\lim_{n\to\infty} T_{2^n,*}f = T_*f \qquad \mbox{a.e.} \qquad (f\in H^S_{p(\cdot),q})
$$
increasingly and so in the $L_{p(\cdot),q}$-norm, too. Hence $T_*f=Vf$ for all $f\in H^S_{p(\cdot),q}$, which proves the theorem.
\end{proof}

Now we are able to prove the boundedness of $\sigma_*$ from $H_{p(\cdot),q}$ to $L_{p(\cdot),q}$.

\begin{theorem} \label{t3}
Let $p(\cdot)\in\mathcal{P}(\Omega)$ satisfy conditions \eqref{log} and \eqref{e104}. If $1/2<p_-<\infty$, then
\begin{equation}\label{e6}
\left\| |\sigma_*a|^{\beta } \chi_{\{\tau=\infty\}} \right\|_{p(\cdot)} \leq C \left\|\chi_{\{\tau<\infty\}}\right\|_{p(\cdot)}^{1-\beta }
\end{equation}
for some $0<\beta<1$ and for all $(3,p(\cdot),\infty)$-atoms $a$, where $\tau$ is the stopping time associated with $a$.
\end{theorem}

\begin{proof}
We can chose $0<\beta<1$ and $1/2<\epsilon<\underline{p}$ such that $\beta\epsilon>1/2$. Instead of (\ref{e6}), we will show that
$$
\left\| |\sigma_*a|^{\beta \epsilon} \chi_{\{\tau=\infty\}} \right\|_{p(\cdot)/\epsilon} \leq C \left\|\chi_{\{\tau<\infty\}}\right\|_{p(\cdot)/\epsilon} \left\|\chi_{\{\tau<\infty\}}\right\|_{p(\cdot)}^{-\beta \epsilon}.
$$
We use same symbols as in the proof of Theorem \ref{t3 Leb}.  Then, by the estimate of $\sigma_*(a)$ in (\ref{e7}),  we have
\begin{equation}\label{e20}
\left\| |\sigma_*a|^{\beta \epsilon} \chi_{\{\tau=\infty\}} \right\|_{p(\cdot)/\epsilon} \leq \left\|\chi_{\{\tau<\infty\}}\right\|_{p(\cdot)}^{-\beta \epsilon}
\left(\left\| A^{\beta \epsilon} \chi_{\{\tau=\infty\}} \right\|_{p(\cdot)/\epsilon} + \left\| B^{\beta \epsilon} \chi_{\{\tau=\infty\}} \right\|_{p(\cdot)/\epsilon}\right).
\end{equation}
By Lemma \ref{lem:duality for variabl p}, we can choose a function $g\in L_{(\frac{p(\cdot)}{\varepsilon})'}$ with $\|g\|_{{(\frac{p(\cdot)}{\varepsilon})'}}\leq 1$ such that
\begin{align*}
\left\| A^{\beta \epsilon} \chi_{\{\tau=\infty\}} \right\|_{p(\cdot)/\epsilon} &= \int_\Omega A(x)^{\beta \epsilon} \chi_{\{\tau=\infty\}} gd\mathbb P.
\end{align*}
Choosing $\max(1,\beta p_+)<r<\infty$ and applying H\"{o}lder's inequality, we obtain
\begin{eqnarray*}
\left\| A^{\beta \epsilon} \chi_{\{\tau=\infty\}} \right\|_{p(\cdot)/\epsilon}
&\leq & \int_\Omega \sum_{l} \sum_{j=0}^{K_l-1} 2^{(j-K_l)\beta \epsilon}
\chi_{I_l^{j}}|g|d\mathbb P \\
& \leq & \sum_{l} \sum_{j=0}^{K_l-1} 2^{(j-K_l)\beta \epsilon} \|\chi_{I_l^{j}}\|_{{\frac{r}{\beta\epsilon}}} \| \chi_{I_l^{j}} g\|_{{(\frac{r}{\beta\epsilon})'}}
\\
&\lesssim& \sum_{l} \sum_{j=0}^{K_l-1} 2^{(j-K_l)\beta \epsilon} \int \chi_{I_l} \left(\frac{1}{\mathbb P(I_l^{j})}\int_{I_l^{j}} \left| g\right|^{(\frac{r}{\beta\epsilon})'}\right)^{1/(\frac{r}{\beta\epsilon})'}d\mathbb P
\end{eqnarray*}
because $\mathbb P(I_l)=\mathbb P(I_l^{j})=2^{-K_l}$. Again by H\"older's inequality,
\begin{eqnarray*}
\left\| A^{\beta \epsilon} \chi_{\{\tau=\infty\}} \right\|_{p(\cdot)/\epsilon}
&\lesssim& \int \sum_{l} \chi_{I_l} \sum_{j=0}^{K_l-1} 2^{(j-K_l)\beta \epsilon (1/(\frac{r}{\beta\epsilon})+1/(\frac{r}{\beta\epsilon})')} \left(\frac{1}{\mathbb P(I_l^{j})}\int_{I_l^{j}} \left| g\right|^{(\frac{r}{\beta\epsilon})'}\right)^{1/(\frac{r}{\beta\epsilon})'}d\mathbb P \\
&\lesssim& \int \sum_{l} \chi_{I_l} \left(\sum_{j=0}^{K_l-1} 2^{(j-K_l)\beta \epsilon} \right) ^{1/(\frac{r}{\beta\epsilon})} \\
&&{}
\left(\sum_{j=0}^{K_l-1} 2^{(j-K_l)\beta \epsilon}
\left(\frac{1}{\mathbb P(I_l^{j})}\int_{I_l^{j}} \left| g\right|^{(\frac{r}{\beta\epsilon})'}\right)\right) ^{1/(\frac{r}{\beta\epsilon})'} d\mathbb P.
\end{eqnarray*}
From this it follows that
\begin{eqnarray*}
\left\| A^{\beta \epsilon} \chi_{\{\tau=\infty\}} \right\|_{p(\cdot)/\epsilon}
&\lesssim& \int \sum_{l} \chi_{I_l}
\left(\sum_{j=0}^{K_l-1} 2^{(j-K_l)\beta \epsilon}
\left(\frac{1}{\mathbb P(I_l^{j})}\int_{I_l^{j}} \left| g\right|^{(\frac{r}{\beta\epsilon})'}\right) \right) ^{1/(\frac{r}{\beta\epsilon})'}d\mathbb P \\
&\leq& \int \sum_{l} \chi_{I_l} \left(U_{\beta \epsilon}(|g|^{(\frac{r}{\beta\epsilon})'})\right)^{1/(\frac{r}{\beta\epsilon})'} d\mathbb P \\
&\leq& \left \|\sum_{l} \chi_{I_l} \right \|_{{p(\cdot)/\epsilon}}  \left\|\left(U_{\beta \epsilon}(|g|^{(\frac{r}{\beta\epsilon})'})\right)^{1/(\frac{r}{\beta\epsilon})'}\right\|_{{(p(\cdot)/\epsilon)'}}.
\end{eqnarray*}
Since  $r>\beta p_+$ and $\epsilon <p_-$, we get
$$(\frac{r}{\beta\epsilon})' <{(p(\cdot)/\epsilon)'} \quad\mbox{and }\quad ((p(\cdot)/\epsilon)')_+<\infty  .
$$
Then Theorem \ref{t15} implies that
$$
\left\|\left(U_{\beta \epsilon}(|g|^{(\frac{r}{\beta\epsilon})'})\right)^{1/(\frac{r}{\beta\epsilon})'}\right\|_{{(p(\cdot)/\epsilon)'}} 
= \left\|U_{\beta \epsilon}(|g|^{(\frac{r}{\beta\epsilon})'}) \right\|_{{\frac{(p(\cdot)/\varepsilon)'}{(r/\beta\epsilon)'}}}^{1/(\frac{r}{\beta\epsilon})'}
\lesssim \|g\|_{{(p(\cdot)/\varepsilon)'}}\leq 1,
$$
which shows that
\begin{equation}\label{e21}
\left\| A^{\beta \epsilon}  \chi_{\{\tau=\infty\}} \right\|_{p(\cdot)/\epsilon} \lesssim \left \|\chi_{\{\tau<\infty\}} \right \|_{{p(\cdot)/\epsilon}},
\end{equation}
whenever
\begin{equation}\label{e108}
	\frac{1}{\left((p(\cdot)/\epsilon)'/(r/\beta \epsilon)'\right)_-}-\frac{1}{\left((p(\cdot)/\epsilon)'/(r/\beta \epsilon)'\right)_+} = \frac{r/(r- \beta \epsilon)}{p_+/(p_+-\epsilon)} - \frac{r/(r- \beta \epsilon)}{p_-/(p_--\epsilon)} < \beta \epsilon.
\end{equation}

This means that
\[
	\frac{p_+- \epsilon}{p_+}-\frac{p_--\epsilon}{p_-} < \beta \epsilon,
\]
in other words,
\[
	\frac{1}{p_-}-\frac{1}{p_+} < \beta.
\]
Since $\beta$ can arbitrarily near to $1$, we obtain \eqref{e104}.

Now let us investigate the second term of (\ref{e20}). We choose again a function $g\in L_{(\frac{p(\cdot)}{\varepsilon})'}$ with $\|g\|_{{(\frac{p(\cdot)}{\varepsilon})'}}\leq 1$ such that
\begin{align*}
\left\| B^{\beta \epsilon} \chi_{\{\tau=\infty\}} \right\|_{p(\cdot)/\epsilon} &= \int_\Omega B(x)^{\beta \epsilon} \chi_{\{\tau=\infty\}} gd\mathbb P.
\end{align*}
Let us apply H\"{o}lder's inequality to obtain
\begin{eqnarray*}
\lefteqn{\left\| B^{\beta \epsilon} \chi_{\{\tau=\infty\}} \right\|_{p(\cdot)/\epsilon} } \n\\
&\lesssim& \int_{\Omega} \sum_{l} \sum_{j=0}^{K_l-1}\sum_{i=j}^{K_l-1}  2^{(j-K_l)\beta\epsilon} 2^{(i-K_l)\beta\epsilon} \chi_{I_l^{j,i}}|g| \, dx \\
&\lesssim& \sum_{l} \sum_{j=0}^{K_l-1}\sum_{i=j}^{K_l-1} 2^{(j-K_l)\beta\epsilon} 2^{(i-K_l)\beta\epsilon}
\|\chi_{I_l^{j,i}}\|_{{\frac{r}{\beta\epsilon}}} \| \chi_{I_l^{j,i}} g\|_{{(\frac{r}{\beta\epsilon})'}}
\\
&\lesssim& \sum_{l} \sum_{j=0}^{K_l-1}\sum_{i=j}^{K_l-1} 2^{(j-K_l)\beta\epsilon} 2^{(i-K_l)\beta\epsilon} 2^{K_l-i} \int \chi_{I_l} \left(\frac{1}{\mathbb P(I_l^{j,i})}\int_{I_l^{j,i}} \left| g\right|^{(\frac{r}{\beta\epsilon})'}\right)^{1/(\frac{r}{\beta\epsilon})'}d\mathbb P,
\end{eqnarray*}
whenever $\max(1,\beta p_+)<r<\infty$ is large enough such that $2\beta\epsilon>r/(r-\beta\epsilon)$. Moreover,
\begin{eqnarray*}
\lefteqn{\left\| B^{\beta \epsilon} \chi_{\{\tau=\infty\}} \right\|_{p(\cdot)/\epsilon} } \n\\
&\lesssim& \int \sum_{l} \chi_{I_l} \sum_{j=0}^{K_l-1}\sum_{i=j}^{K_l-1} 2^{(i+j-2K_l)\beta\epsilon(1/(\frac{r}{\beta\epsilon})+1/(\frac{r}{\beta\epsilon})')} 2^{K_l-i} \left(\frac{1}{\mathbb P(I_l^{j,i})}\int_{I_l^{j,i}} \left| g\right|^{(\frac{r}{\beta\epsilon})'}\right)^{1/(\frac{r}{\beta\epsilon})'}d\mathbb P\\
&\lesssim& \int \sum_{l} \chi_{I_l} \left(\sum_{j=0}^{K_l-1}\sum_{i=j}^{K_l-1} 2^{(j-K_l)\beta\epsilon} 2^{(i-K_l)\beta\epsilon} \right)^{1/(\frac{r}{\beta\epsilon})} \left(\sum_{j=0}^{K_l-1}\sum_{i=j}^{K_l-1} 2^{(j-K_l)\beta\epsilon} 2^{(i-K_l)\beta\epsilon} \right.\\
&&{} \left.2^{(K_l-i)(\frac{r}{\beta\epsilon})'}\frac{1}{\mathbb P(I_l^{j,i})}\int_{I_l^{j,i}} \left| g\right|^{(\frac{r}{\beta\epsilon})'}\right)^{1/(\frac{r}{\beta\epsilon})'}d\mathbb P\\
&\lesssim& \int \sum_{l} \chi_{I_l} \left(\sum_{j=0}^{K_l-1}\sum_{i=j}^{K_l-1} 2^{(j-K_l)(2\beta\epsilon-r/(r-\beta\epsilon))} 2^{(j-i)(r/(r-\beta\epsilon)-\beta\epsilon)} \frac{1}{\mathbb P(I_l^{j,i})}\int_{I_l^{j,i}} \left| g\right|^{(\frac{r}{\beta\epsilon})'}\right)^{1/(\frac{r}{\beta\epsilon})'}d\mathbb P.
\end{eqnarray*}
Taking into account the definition of the maximal operator $V_{\alpha,t}$ and Theorem \ref{t17}, we obtain
\begin{align*}
\left\| B^{\beta \epsilon} \chi_{\{\tau=\infty\}} \right\|_{p(\cdot)/\epsilon}
&\leq \int \sum_{l} \chi_{I_l} \left(V_{2\beta\epsilon-r/(r-\beta\epsilon),r/(r-\beta\epsilon)-\beta\epsilon}(|g|^{(\frac{r}{\beta\epsilon})'})\right)^{1/(\frac{r}{\beta\epsilon})'} d\mathbb P \\
&\leq \left \|\sum_{l} \chi_{I_l} \right \|_{{p(\cdot)/\epsilon}}  \left\|\left(V_{2\beta\epsilon-r/(r-\beta\epsilon),r/(r-\beta\epsilon)-\beta\epsilon}(|g|^{(\frac{r}{\beta\epsilon})'})\right)^{1/(\frac{r}{\beta\epsilon})'}\right\|_{{(p(\cdot)/\epsilon)'}} \\
&\lesssim \left \|\chi_{\{\tau<\infty\}} \right \|_{{p(\cdot)/\epsilon}}
\end{align*}
as in (\ref{e21}), whenever \eqref{e108} and \eqref{e104} hold. The proof of the theorem is complete.
\end{proof}

Now we are able to prove the next theorem. It was proved by the third author in \cite{wces3} for $H_{p,q}$ with constant $p$.

\begin{theorem} \label{t4}
Let $p(\cdot)\in\mathcal{P}(\Omega)$ satisfy conditions \eqref{log} and \eqref{e104}. If $1/2<p_-<\infty$ and $0<q\leq \infty$, then
$$
\left\|\sigma_*f\right\|_{L_{p(\cdot),q}}\lesssim \left\|f\right\|_{H_{p(\cdot),q}},\quad f\in H_{p(\cdot),q}.
$$
\end{theorem}

\begin{proof}
It is easy to see that (\ref{e10}) holds in this case, i.e.,
$$
\sigma_kf_n=\sigma_kf \qquad \mbox{for} \qquad 0\leq k\leq 2^n.
$$
Applying Theorems \ref{t6} and \ref{t3}, we can complete the proof.
\end{proof}

For a constant $p$ with $p\leq 1/2$, the theorem does not hold (see Simon and Weisz \cite{wgyenge}, Simon \cite{si7} and G{\'a}t and Goginava \cite{gago}).
Since the Walsh polynomials are dense in $H_{p(\cdot),q}$ as well, the next three consequences can be proved as Corollaries \ref{c20} and \ref{c50} and \ref{c51}.

\begin{corollary} \label{c2}
Let $p(\cdot)\in\mathcal{P}(\Omega)$ satisfy conditions \eqref{log} and \eqref{e104}, $1/2<p_-<\infty$ and $0<q\leq \infty$. If $f\in H_{p(\cdot),q}$, then $\sigma_{n} f$ converges almost everywhere on $[0,1)$.
\end{corollary}

\begin{corollary}\label{c5}
Let $p(\cdot)\in\mathcal{P}(\Omega)$ satisfy conditions \eqref{log} and \eqref{e104}, $1/2<p_-<\infty$, $0<q\leq \infty$ and $f\in H_{p(\cdot),q}$. If there exists a dyadic interval $I$ such that the restriction $f \chi_I \in L_1(I)$, then
$$
\lim_{n\to\infty}\sigma_{n} f(x)=f(x) \qquad \mbox{for a.e. $x\in I$.}
$$
\end{corollary}

\begin{corollary}\label{c56}
Let $p(\cdot)\in\mathcal{P}(\Omega)$ satisfy conditions \eqref{log} and \eqref{e104}, $1\leq p_-<\infty$, $0<q\leq \infty$ and $f\in H_{p(\cdot),q}$. Then
$$
\lim_{n\to\infty}\sigma_{n} f(x)=f(x) \qquad \mbox{for a.e. $x\in [0,1)$.}
$$
\end{corollary}

Now we prove that for integrable functions, the limit of $\sigma_nf$ is exactly the function. Since $L_1 \subset H_{1,\infty}$, more exactly,
$$
\|f\|_{H_{1,\infty}} = \sup_{\rho >0} \rho \, \mathbb P(M(f) > \rho) \leq C \|f\|_1 \qquad (f \in L_1),
$$
(see e.g. Weisz \cite{Weisz1994book}), we obtain the next corollary, which was shown by
Fine \cite{fi2}, Schipp \cite{sch2} and Weisz \cite{wca1}.

\begin{corollary}\label{c1}
If $f \in L_1$, then
$$
\lim_{n\to\infty}\sigma_{n} f(x)=f(x) \qquad \mbox{for a.e. $x\in [0,1)$.}
$$
\end{corollary}

The results about the norm convergence can be shown in the same way.

\begin{corollary} \label{c3}
Let $p(\cdot)\in\mathcal{P}(\Omega)$ satisfy conditions \eqref{log} and \eqref{e104}, $1/2<p_-<\infty$  and $0<q\leq \infty$. If $f\in H_{p(\cdot),q}$, then $\sigma_{n} f$ converges in the $L_{p(\cdot),q}$-norm.
\end{corollary}

\begin{corollary}\label{c57}
Let $p(\cdot)\in\mathcal{P}(\Omega)$ satisfy conditions \eqref{log} and \eqref{e104}, $1/2<p_-<\infty$, $0<q\leq \infty$ and $f\in H_{p(\cdot),q}$. If there exists a dyadic interval $I$ such that the restriction $f \chi_I \in L_1(I)$, then
$$
\lim_{n\to\infty}\sigma_{n} f=f \qquad \mbox{in the $L_{p(\cdot),q}(I)$-norm.}
$$
\end{corollary}

\begin{corollary}\label{c58}
Let $p(\cdot)\in\mathcal{P}(\Omega)$ satisfy conditions \eqref{log} and \eqref{e104}, $1\leq p_-<\infty$, $0<q\leq \infty$ and $f\in H_{p(\cdot),q}$. Then
$$
\lim_{n\to\infty}\sigma_{n} f=f  \qquad \mbox{in the $L_{p(\cdot),q}$-norm.}
$$
\end{corollary}

Note that $H_{p(\cdot),q}$ is equivalent to $L_{p(\cdot),q}$ if $1<p_-<\infty$.

Similarly to Theorem \ref{t9}, we do not need the restriction $1/2<p_-$ in the next results.

\begin{theorem} \label{t5}
If $p(\cdot)\in\mathcal{P}(\Omega)$ satisfies conditions \eqref{log} and \eqref{e104}. If $0<q\leq \infty$, then
$$
\left\|\sup_{n\in \N}|\sigma_{2^n}f|\right\|_{L_{p(\cdot),q}}\lesssim \left\|f\right\|_{H_{p(\cdot),q}},\quad f\in H_{p(\cdot),q}.
$$
\end{theorem}

This implies the following corollaries.

\begin{corollary} \label{c6}
Let $p(\cdot)\in\mathcal{P}(\Omega)$ satisfy conditions \eqref{log} and \eqref{e104}. If $0<q\leq \infty$ and $f\in H_{p(\cdot),q}$, then $\sigma_{2^n} f$ converges almost everywhere on $[0,1)$.
\end{corollary}

\begin{corollary}\label{c7}
Let $p(\cdot)\in\mathcal{P}(\Omega)$ satisfy conditions \eqref{log} and \eqref{e104}, $0<q\leq \infty$
and $f\in H_{p(\cdot),q}$. If there exists a dyadic interval $I$ such that the restriction $f \chi_I \in L_1(I)$, then
$$
\lim_{n\to\infty}\sigma_{2^n} f(x)=f(x) \qquad \mbox{for a.e. $x\in I$.}
$$
\end{corollary}

\begin{corollary} \label{c4}
Let $p(\cdot)\in\mathcal{P}(\Omega)$ satisfy conditions \eqref{log} and \eqref{e104}. If $0<q\leq \infty$ and $f\in H_{p(\cdot),q}$, then $\sigma_{2^n} f$ converges in the $L_{p(\cdot),q}$-norm.
\end{corollary}

\begin{corollary}\label{c59}
Let $p(\cdot)\in\mathcal{P}(\Omega)$ satisfy conditions \eqref{log} and \eqref{e104}, $0<q\leq \infty$ and $f\in H_{p(\cdot),q}$. If there exists a dyadic interval $I$ such that the restriction $f \chi_I \in L_1(I)$, then
$$
\lim_{n\to\infty}\sigma_{2^n} f=f \qquad \mbox{in the $L_{p(\cdot),q}(I)$-norm.}
$$
\end{corollary}

\medskip
\textbf{Acknowledgements}
Yong Jiao is supported by NSFC(11471337) and Hunan Provincial Natural Science Foundation(14JJ1004). Ferenc Weisz is supported by the Hungarian National Research, Development and Innovation Office - NKFIH, K115804 and KH130426. Lian Wu is supported by NSFC(11601526). Dejian Zhou is supported by the Hunan Provincial Innovation Foundation for Postgraduate.

\bibliographystyle{amsalpha}

\end{document}